\documentclass[12pt]{amsart}
\usepackage{mathrsfs,amssymb}
\usepackage{euscript}
\usepackage[usenames,dvipsnames]{pstricks}
\usepackage{epsfig}
\usepackage{pst-grad} 
\usepackage{pst-plot} 
\usepackage[all]{xy}
\usepackage[hypertex]{hyperref}
\catcode`\@=11
\def\@evenfoot{\rule{0pt}{20pt}[\today] \hfill [{\tt \jobname.tex}]}
\def\@oddfoot{\rule{0pt}{20pt}{[\tt \jobname.tex}]\hfill [\today]}
\catcode`\@=13
\newtheorem{theorem}{Theorem}[section]
\newtheorem{proposition}[theorem]{Proposition}
\newtheorem{lemma}[theorem]{Lemma}

\newtheorem{corollary}[theorem]{Corollary}

\theoremstyle{definition}
\newtheorem{definition}[theorem]{Definition}
\newtheorem{example}[theorem]{Example}

\newtheorem{remark}[theorem]{Remark}

\def\fAss{\underline{{\mathcal Ass}}}
\def\E#1{{}^{#1}\hskip -.2emE}
\def\scD{{\sf D}}
\def\eqv{
{
\unitlength=.4pt
\begin{picture}(60.00,20.00)(-10.00,4)
\thinlines
\put(40.00,15.00){\vector(2,-1){0.00}}
\put(40.00,5.00){\vector(2,1){0.00}}
\qbezier(0.00,5.00)(20.00,-5.00)(40.00,5.00)
\qbezier(0.00,15.00)(20.00,25.00)(40.00,15.00)
\end{picture}}
}
\def\timesred{\hbox{$\hskip .1em\times$}}
\def\Fib{{\it Fib}}\def\calK{{\mathscr K}}
\def\op{{\it op}}\def\calL{{\mathscr L}}
\def\ttDelta{{\tt \Delta}}
\def\cosA{\hbox {\AA}}
\def\opTm{{\hbox{$\mathcal T \hskip -.2em m$}}}
\def\sopTm{{\hbox{\scriptsize $\mathcal T \hskip -.2em m$}}}
\def\scE{{\mathscr E}}
\def\opK{{\mathcal K}}
\def\opTam{\hbox {$\mathcal T \hskip -.2em am$}}
\def\funny{\EuScript{S}p_2}
\def\qOmega{{\Omega}}
\def\frS{{\boldsymbol \EuS}}
\def\frT{{\boldsymbol \EuT}}
\def\frO{{\boldsymbol \EuO}}\def\frN{{\boldsymbol \EuN}}
\def\liota{\hbox{\large $\iota$}}\def\siota{\hbox{$\iota$}}
\def\bfjedna{{\boldsymbol 1}}
\def\bfU{{\boldsymbol \EuU}}
\def\bfS{{\boldsymbol \EuS}}
\def\bfT{{\boldsymbol \EuT}}
\def\EuN{{\EuScript N}}
\def\EuO{{\EuScript O}}
\def\EuU{{\EuScript U}}
\def\calP{{\mathcal P}}

\def\calO{{\mathcal O}}

\def\ittau{t}\def\itomega{w}\def\itkappa{s}
\def\itS{S}
\def\itT{T}

\def\Cplusone{\hbox{$\ttC \!+\!1$}}
\def\frC{{\mathfrak C}}
\def\OpCat{{\tt OpCat}}
\def\ttV{{V}}
\def\ttCat{{\tt Cat}}\def\CH{{\it CH}}\def\ttM{{\sf M}}
\def\cernybily{\psscalebox{.5 .5}
{\hskip -1em
\begin{pspicture}(0,-0.6257164)(4.146097,0.257164)
\pscustom[linecolor=black, linewidth=0.05]
{
\newpath
\moveto(8.9,-3.9742837)
}
\pscustom[linecolor=black, linewidth=0.05]
{
\newpath
\moveto(21.5,-2.6742837)
}

\psdots[linecolor=black, dotsize=0.49939746](3.9,-0.37428364)
\psline[linecolor=black, linewidth=0.05, arrowsize=0.05291666666666672cm 6.0,arrowlength=2.0,arrowinset=0.0]{<-}(1.2,-0.37389573)(3.6974843,-0.37389573)
\psdots[linecolor=black, fillstyle=solid, dotstyle=o, dotsize=0.49939746](1.0,-0.37428364)
\psdots[linecolor=black, dotstyle=otimes, dotsize=0.49939746](1.0,-0.37428364)
\end{pspicture}
}}
\def\pro{{{\vbox to .65em{\vss \hbox to
        1.6em{$\widetilde{\Tam}$}}}^\bbN_2}}
\def\pr{{{\vbox to .65em{\vss \hbox to 1.1em{$\widetilde{\Tm}$}}}^\bbN_2}}
\def\Tot{{\it Tot}}
\def\Nat{{\it Nat}}
\def\Alg{{\tt Alg\/}}
\def\ttLO{{\tt LO}}
\def\ttU{{\tt U}}\def\Lat{{\EuScript L}}
\def\ttC{{\tt C}}
\def\ttQ{{\tt Q}}\def\ttR{{\tt R}}
\def\Deltaalg{{\tt \Delta}_{\it alg}}
\def\ttOmega{{\hskip .1em\tt \Omega}}\def\Ord{{\tt Ord}}
\def\dash{{\hbox{\hskip .15em -\hskip .1em}}}
\def\Ordk{{{\Ord}_k}}\def\OrdkN{{{\Ord}^\bbN_k}}
\def\cerny{\hskip -.25em\raisebox{.0em}{
\psscalebox{.6 .6}
{
\begin{pspicture}(0,-0.19711548)(0.39423096,0.19711548)
\psdots[linecolor=black, dotsize=0.4](0.19711548,0.0)
\end{pspicture}
}}}
\def\tlusty{\hskip -.25em\raisebox{.0em}{
\psscalebox{.6 .6}
{
\begin{pspicture}(0,-0.20138916)(0.40277833,0.20138916)
\psdots[linecolor=black, dotstyle=o, dotsize=0.4](0.20138916,0.0)
\psdots[linecolor=black, dotstyle=otimes, dotsize=0.4](0.20138916,0.0)
\end{pspicture}
}}}
\def\bily{\hskip -.25em\raisebox{.0em}{
\psscalebox{.6 .6}
{
\begin{pspicture}(0,-0.20138916)(0.40277833,0.20138916)
\psdots[linecolor=black, dotstyle=o, dotsize=0.4](0.20138916,0.0)
\end{pspicture}
}}}

\def\ttP{{\tt P}}\def\ttO{{\tt O}}
\def\Tm{{\tt Tm}}\def\seq{{\mathbf {seq}}}

\def\coll#1{\{#1(n)\}_{n \geq 0}}

\def\rada#1#2{{#1,\ldots,#2}}
\def\tricases#1#2#3#4#5#6{
                  \left\{
                         \begin{array}{ll}
                           #1,\ &\mbox{#2}
                           \\
                           #3,\ &\mbox{#4}
                           \\
                           #5,\ &\mbox{#6}
                          \end{array}
                   \right.
}

\def\vlra{{\hbox{$-\hskip-1mm-\hskip-2mm\longrightarrow$}}}
\def\lra{\longrightarrow}
\def\uAss{\underline{\hbox{$\mathscr A \hskip -.2em ss$}}}
\def\doubless#1#2{{
\def\arraystretch{.5}
\begin{array}{c}
\mbox{\scriptsize $\scriptstyle #1$}
\\
\mbox{\scriptsize $\scriptstyle #2$}
\end{array}\def\arraystretch{1}
}}
\def\OmegaN{{\tt \Omega_2^\bbN}}\def\OmegakN{{{\tt \Omega}_k^\bbN}}
\def\Tam{{\tt Tam}}
\def\Ar{{\it Ar}}
\def\id{{\it id}}
\def\vdotsred{{\multiput(0,0)(0,-.7){3}{\scriptsize$\cdot$}}}
\def\krtecek#1#2{
\unitlength .22em
\linethickness{0.4pt}
\ifx\plotpoint\undefined\newsavebox{\plotpoint}\fi 
\begin{picture}(17,15)(6.25,1)
\thicklines
\multiput(13.8,5.5)(.9,0){3}{\makebox(0,0)[cc]{$\cdot$}}
\qbezier(11,7)(11.5,12.5)(14,12)
\qbezier(17,7)(16.5,12.5)(14,12)
\put(11,7){\line(1,0){6}}
\put(14,12){\vector(0,1){5.5}}
\put(11.2,7){\line(0,-1){2}}
\put(13.1,7){\line(0,-1){2}}
\put(16.7,7){\line(0,-1){2}}
\put(14.2,7.6){\makebox(0,0)[b]{\scriptsize $\tau^#1_{#2}$}}
\end{picture}
}

\def\krtecekmod#1#2{
\unitlength .22em
\linethickness{0.4pt}
\ifx\plotpoint\undefined\newsavebox{\plotpoint}\fi 
\begin{picture}(17,10)(6.25,9)
\thicklines
\multiput(14.3,5.5)(.9,0){3}{\makebox(0,0)[cc]{$\cdot$}}
\qbezier(11,7)(11.5,12.5)(14,12)
\qbezier(17,7)(16.5,12.5)(14,12)
\put(11,7){\line(1,0){6}}
\put(14,12){\vector(0,1){5.5}}
\put(11.2,7){\line(0,-1){2}}
\put(13.1,7){\line(0,-1){2}}
\put(16.7,7){\line(0,-1){2}}
\put(14.2,7.6){\makebox(0,0)[b]{\scriptsize $\tau^#1_{#2}$}}
\end{picture}
}

\def\krteceksfouskama#1#2{
\unitlength .22em
\linethickness{0.4pt}
\ifx\plotpoint\undefined\newsavebox{\plotpoint}\fi 
\begin{picture}(17,15)(6.25,1)
\thicklines
\multiput(14,5.5)(.9,0){3}{\makebox(0,0)[cc]{$\cdot$}}
\qbezier(11,7)(11.5,12.5)(14,12)
\qbezier(17,7)(16.5,12.5)(14,12)
\put(11,7){\line(1,0){6}}
\put(14,12){\vector(0,1){5.5}}
\put(11.3,7){\line(0,-1){2}}
\put(13.2,7){\line(0,-1){2}}
\put(16.7,7){\line(0,-1){2}}
\put(14.2,7.6){\makebox(0,0)[b]{\scriptsize $\tau^#1_{#2}$}}
\thinlines
\put(10,8.8){\makebox(0,0){\vdotsred}}
\put(17,8.8){\makebox(0,0){\vdotsred}}
\put(9.8,7.5){\line(1,0){1.2}}
\put(9.8,10.5){\line(1,0){1.7}}
\put(9.8,11.5){\line(1,0){2.4}}
\put(18.2,7.5){\line(-1,0){1.2}}
\put(18.2,10.5){\line(-1,0){1.7}}
\put(18.2,11.5){\line(-1,0){2.4}}
\end{picture}
}

\def\ctver#1{{
\unitlength .4em
\linethickness{0.4pt}
\ifx\plotpoint\undefined\newsavebox{\plotpoint}\fi 
\begin{picture}(13,13)(5.4,2.7)
\thicklines
\multiput(11.1,6.2)(.4,0){3}{\makebox(0,0)[cc]{$\cdot$}}
\put(9,7){\line(0,1){5}}
\put(9,12){\line(1,0){4}}
\put(13,12){\line(0,-1){5}}
\put(13,7){\line(-1,0){4}}
\put(11,12){\vector(0,1){2}}
\put(9.5,6){\line(0,1){1}}
\put(10.5,6){\line(0,1){1}}
\put(12.5,6){\line(0,1){1}}
\put(11,9){\makebox(0,0)[b]{\scriptsize $#1$}}
\end{picture}
}}

\def\except{\hskip .2em\raisebox{-.16em}{\rule{.8pt}{1em}}  \hskip .2em}
\def\kol#1{
\unitlength .3em
\linethickness{0.4pt}
\ifx\plotpoint\undefined\newsavebox{\plotpoint}\fi 
\begin{picture}(17,14)(7,4.2)
\thicklines
\put(15,11){\makebox(0,0)[cc]{$\circ$}}
\multiput(14.8,6.8)(.7,0){3}{\makebox(0,0)[cc]{$\cdot$}}
\put(15.5,11.6){\makebox(0,0)[lb]{\scriptsize$#1$}}
\put(0,-.5){
\put(14.75,11.1){\vector(1,2){.07}}\qbezier(13,7)(14,10.5)(15,11)
\put(15.25,11.1){\vector(-1,2){.07}}\qbezier(17,7)(16,10.5)(15,11)
\put(15,11){\vector(1,3){.07}}\qbezier(14,7)(14,8.5)(15,11)
}
\put(15,16){\vector(0,1){.07}}\qbezier(15,11.7)(15,12)(15,15)
\end{picture}
}

\def\bfone{{\mathfrak I}}

\def\Bq{{\tt Bq}}
\def\OrN{{\tt Ord}_2^{\bbN}}

\def\Wh{\raisebox{.0em}{\hglue -.2em\bily}}
\def\LTr{{\tt LTr}}\def\sLTr{\hbox{\scriptsize \tt LT\/r}}
\def\OpSet#1{{\tt Op}^{#1}(\Set)}\def\sFSet{{\tt sFSet}}
\def\OpV#1{{\tt Op}^{#1}(\ttV)}\def\ColV#1{{\tt Coll}^{#1}(\ttV)}
\def\ot{\otimes}\def\FSet{{\tt FSet}}
\def\Leaf{{\it L \hskip .03emf\!}}
\def\arraystretch{1.2}
\def\Boxx{\raisebox{-.1em}{\hbox{\hskip .2em\large $\Box$\hskip .1em}}}

\def\boxx{{\hskip .2em \Box}}
\def\End{\hbox{${\mathcal E}${\hskip -.1em\it nd\hskip .1em}}} 
\def\Coend{\hbox{\it co\/${\mathcal E}${\hskip -.1em\it nd\hskip .1em}}}

\def\Rada#1#2#3{#1_{#2},\dots,#1_{#3}}\def\tr{{\rm tr}}
\def\funny{\EuScript{S}p_2}\def\duo{\mathscr{D}}
\def\EuT{{\EuScript T}}\def\EuS{{\EuScript S}}
\def\scA{\mathscr{A}}
\def\bbN{{\mathbb N}}

\def\Set{\tt Set}

\def\fAss{\underline{{\mathcal Ass}}}

\def\cases#1#2#3#4{
                  \left\{
                         \begin{array}{ll}
                           #1,\ &\mbox{#2}
                           \\
                           #3,\ &\mbox{#4}
                          \end{array}
                   \right.
}

\title{Operadic categories and Duoidal Deligne's conjecture}
\author[M.~Batanin]{Michael Batanin}
\thanks{The first author acknowledges  the financial
support of Scott Russel Johnson Memorial Foundation 
and Australian Research Council (grant
No.~DP1095346).}
\author[M.~Markl]{Martin~Markl}
\thanks{The second author was supported by Eduard \v Cech
  Institute P201/12/G028 and RVO: 67985840.}

\catcode`\@=11
\address{Macquarie University, NSW 2109, Australia}
\email{michaelbatanin@mq.edu.au}
\address{Mathematical Institute of the Academy, {\v Z}itn{\'a} 25,
         115 67 Prague 1, The Czech Republic}
\address{Faculty of Mathematics and Physics, Charles University,
186 75 Sokolovsk\'a 83, Prague~8, The Czech Republic}
\email{markl@math.cas.cz}
\catcode`\@=13


\begin{document}

\begin{abstract}
  The purpose of this paper is two-fold. In {\bf Part~1} 
  we introduce a new theory of operadic categories and their operads.  
  This theory is, in our opinion, of an independent~value. 

  In {\bf Part~2} we use this new theory together with our previous
  results to prove that multiplicative $1$-operads in duoidal
  categories admit, under some mild conditions on the underlying
  monoidal category, natural actions of contractible $2$-operads. The
  result of D.~Tamarkin on the structure of dg-categories, as well as
  the classical Deligne conjecture for the Hochschild cohomology, is a
  particular case of this statement.
\end{abstract}

\maketitle
\bibliographystyle{plain}
\baselineskip 17pt plus 2pt minus 1 pt

\tableofcontents

\section*{Introduction}

In  \cite{batanin-markl:Adv} we proposed a notion of center
and homotopy center of a monoid in a monoidal category enriched in a
duoidal category.  Examples include classical
centers but also the $2$-category of categories, the symmetric
monoidal closed category $\mathbf{Gray}$ of $2$-categories,
$2$-functors and pseudonatural transformations \cite{GPS} and
Tamarkin's homotopy 2-category of $dg$-categories, $dg$-functors and
their coherent natural transformations~\cite{tamarkin:CM07}.

It is well-known that the center of an associative algebra is a
commutative algebra. A~homotopical analogue of this statement is the
famous Deligne conjecture which states that there is a natural action
of an $E_2$-operad on the Hochschild complex of an associative algebra
lifting the Gerstenhaber algebra structure from the cohomology to the
chain level.  We conjectured in \cite{batanin-markl:Adv} that our
generalized (homotopical) center admits a closely related algebraic
structure. Namely, there is a natural action of a contractible
$2$-operad in the sense of the first author~\cite{batanin:globular} on
the homotopical center of a monoid \cite[Corollary
11.20]{batanin-markl:Adv}. We call this statement duoidal Deligne's
conjecture.  Tamarkin's main theorem from \cite{tamarkin:CM07} is a
particular case of this conjecture. Classical Deligne's conjecture
also follows from the duoidal version \cite[Corollary
11.22]{batanin-markl:Adv}.  The {\bf main goal} of this paper is to
prove this conjecture under some mild homotopical conditions on the
base symmetric monoidal category $V.$

Our {\bf secondary goal} is to advertise a new theory which, as we
believe, has an independent value.  During our work on the proof of
duoidal Deligne's conjecture we discovered that the existing language
is not adequate for our purposes. Some operad-like structures that we
wanted to use were not operads in any of the existing senses.  To
overcome these difficulties, we introduce a concept of an operadic
category and of an operad corresponding to such a category.

Examples of operadic categories are abundant. They include categories
like finite sets, finite ordinals, the categories of $n$-ordinals
\cite{batanin:conf} and $n$-trees
\cite{batanin:globular,batanin-street}, Barwick operator
categories~\cite{barwick}, ordered graphs and many other.  Examples of
the corresponding operads are (colored) classical symmetric and
nonsymmetric operads, $n$-operads~\cite{batanin:globular},
hyperoperads of Getzler and Kapranov~\cite{getzler-kapranov:CompM98}, 
charades of Kapranov~\cite{kapranov:langlands}, \&c.
As classical operads, our generalized operads have
algebras, which now include other operad-like structures such as
(wheeled) properads or PROPs, cyclic operads and (twisted) modular
operads.

We believe that operadic categories admit a rich and interesting
theory. They are closely connected to other existing and emerging
approaches to generalized operad-like structures such as Feynman
categories of Kaufmann \cite{kaufmann:Fey}, polynomial monads
\cite{batanin13:_homot}, moment categories of Berger
\cite{Berger:personal} and operator categories of Barwick~\cite{barwick}. To
keep the focus, we decided to choose a
`minimalistic' approach and to include as much or as little theory of
operadic categories as necessary for the proof of Deligne's
conjecture.  A deeper theory, including $2$-categorical aspects of
operadic categories and relation to other notions with more applications,
will be developed in subsequent papers.

\section*{Plan of the paper}

According to our goals we decided to subdivide our paper into two parts.
{\bf Part~1\/} contains all necessary definitions and facts
about operadic categories.

An operadic category is a category over the (skeletal) category of finite sets
whose morphisms come with a finite set of objects of the same
category, called fibers. The motivating example is the category of
finite sets itself, where the set of fibers of a map $f:T\to S$ is
the set of preimages $f^{-1}(i),\ i\in S.$ The axioms of operadic
category are designed to make the assignment of fibers to morphisms 
an abstract algebraic
structure on a category.  Operadic functors are functors which
preserve fibers and some other elements of this structure.

With each operadic category one associates its category of operads
with values in a symmetric monoidal category. In many respects the
category of operads plays the r\^ole of the presheaf category over a
small category, but other aspects of operadic categories make it
closer to multicategories. We show that several standard
categorical notions such as the left Kan extensions, discrete
fibrations, the Grothendieck
construction or Beck-Chevalley squares extend to operadic
categories and operads. On the other hand, some important notions from
the theory of classical operads, such as the Day-Street convolution,
multitensors and the condensation can be easily carried
over to the context of operadic categories.

{\bf Part 2} of our paper shows how all these notions work together in
the proof of duoidal Deligne's conjecture. In
Section~\ref{sec:oper-categ-levell} we describe an operadic
category $\tt LTr$ and an $\tt LTr$-operad $\bfone^{\tt LTr}$ that 
canonically acts on any
multiplicative $1$-operad in any duoidal category.
 
In the subsequent sections we demonstrate that this action induces an action
of a colored $2$-operad $\opTm^\bbN_2$ on the same multiplicative
$1$-operad.  This induced action is crucial and the most
complicated part of our proof.  The $2$-operad $\opTm^\bbN_2$ is
constructed as a pullback of the Tamarkin-Tsygan colored symmetric
operad $\Lat^{(2)}$ \cite{tamarkin-tsygan:LMP01,batanin-berger} whose
definition we recall in Section~\ref{sec:tamark-tsyg-oper}. Algebras
of $\Lat^{(2)}$ are multiplicative nonsymmetric operads. While this
observation was enough to prove classical Deligne's conjecture
\cite{batanin-berger}, for the duoidal version we have to take into
account that the two units of a duoidal category can be different, though
connected by a noninvertible morphism. This asymmetry cannot be
captured by the classical approach via symmetric operads. This was
our reason for introducing operadic categories.

Once this induced action of $\opTm^\bbN_2$ 
is constructed, we obtain the {\bf proof} of duoidal
Deligne's conjecture  using the condensation of
\cite{batanin-berger}  generalized to the context of operadic categories.
This is the subject of the last section of our paper.

\section*{Conventions} 
Throughout this article, $\ttV$ will be a complete, cocomplete closed
symmetric monoidal category.  A tree will always mean a rooted
(i.e.~directed) tree~\cite[II.1.5]{markl-shnider-stasheff:book}. 
The arity  of a~vertex of
a directed tree is the number of incoming edges of that vertex.

Categories will be denoted by typewriter letters $\ttO$, ${\ttP}$,
$\Tam_2^\bbN$, $\OmegaN$, \&c. Exceptions are our basic monoidal
category which we keep denoting $V$ from historical reasons, and the
basic duoidal category $\duo$ used in Part~2. Operads in $\ttV$ will
be denoted by the calligraphic letters $\calO$, $\calP$, \&c., while a
typical 
operad in $\duo$ will be denoted by the script $\scA$. A more specific
notation used in Part~2 is summarized in
Table~\ref{za_tyden_z_Bonnu_do_Prahy}. 
Finally, $\bbN$ denotes the set of natural numbers (inducing $0$).

\hskip .2em
\noindent {\bf Acknowledgment.} We enjoyed the wonderful atmosphere of
the Max-Planck Institut f\"ur Mathematik in Bonn during the period when
the first draft of this paper was completed.  We also wish to express
our gratitude to C.~Berger, R.~Garner, R.~Kaufmann, S.~Lack, R.~Street,
D.~Tamarkin and M.~Weber for many useful comments and
conversations.

\part{Theory of operadic categories}

\section{Operadic categories and their operads}

Let $\sFSet$ be the skeletal category of finite sets. The objects of
this category are linearly ordered sets $\bar{n} =\{1,\ldots,n\} ,
n\in \bbN.$ Morphisms are arbitrary maps between these sets.  We
define the $i$th fiber $f^{-1}(i)$ of a morphism $f: T \to S$,
$i\in S$, as the pullback of $f$ along the map $\bar{1}\to S$ which picks up the element $i$, so this is an  object $f^{-1}(i)=\bar{n}_i \in \sFSet$ which is isomorphic as
a linearly ordered set to the preimage $\big\{j\in T \ | \ f(j) =i\big\}$.
Any commutative diagram in $\sFSet$
  \[
    \xymatrix@C = +1em@R = +1em{
      T      \ar[rr]^f \ar[dr]_h & & S \ar[dl]^g
      \\
      &R&
    }
\]
then induces a map $f_i:h^{-1}(i)\to g^{-1}(i)$ for any $i\in R.$ This
assignment is a functor $\Fib_i:\sFSet/R\to \sFSet.$ Moreover, for
any $j\in S$ we have the equality $f^{-1}(j) =
f^{-1}_{g(j)}(j).$ The above structure on the category $\sFSet$ motivates the
following abstract definition.

A {\em strict  operadic category\/} is a 
category $\ttO$  equipped with a `cardinality' 
functor  $|\dash|:\ttO\to \sFSet$ having the following properties. We
require that each connected component of $\ttO$ has a chosen terminal object
$U_c$, $c\in \pi_0(\ttO)$.  We also assume that for every $f:T\to S$
in $\ttO$ and every element $i\in |S|$ there is given an object
$f^{-1}(i) \in \ttO,$ which we will call {\it the $i$th fiber\/} of~$f$,
such that $|f^{-1}(i)| = |f|^{-1}(i).$ We also require that 
\begin{itemize}  \item[(i)] For any $c\in \pi_0(\ttO)$, $|U_c| = 1$.\end{itemize}

A {\em trivial\/} 
 morphism $f:T\to S$ in $\ttO$ is a morphism such that, for each $i\in
 |S|$,  
$f^{-1}(i) = U_{d_i}$ for some $d_i\in \pi_0(\ttO).$

The remaining axioms for a strict operadic category are:
\begin{itemize}
\item[(ii)]The identity morphism $id:T\to T$ is trivial for any $T\in \ttO;$

\item[(iii)] For any commutative diagram in $\ttO$
  \begin{equation}
    \label{Kure}
    \xymatrix@C = +1em@R = +1em{
      T      \ar[rr]^f \ar[dr]_h & & S \ar[dl]^g
      \\
      &R&
    }
\end{equation}
  and every $i\in |R|$ one is given a map
  \[
  f_i: h^{-1}(i)\to g^{-1}(i)
  \]
  such that $|f_i|: |h^{-1}(i)|\to |g^{-1}(i)|$ is the map
  $|h|^{-1}(i)\to |g|^{-1}(i)$ of sets induced by
  \[
  \xymatrix@C = +1em@R = +1em{ |T| \ar[rr]^{|f|} \ar[dr]_{|h|} & & |S|
    \ar[dl]^{|g|}
    \\
    &|R|& }.
  \]
  We moreover require that this assignment forms a functor ${\Fib}_i: \ttO/R \to \ttO$. If
  $R=U_c$, the functor ${\Fib}_1$ is required to be the domain functor
  $\ttO/R \to \ttO.$

  \item[(iv)] 
In the situation of (iii), for any $j\in |S|$,  one has the equality
\begin{equation}
\label{karneval_u_retardacku}
  f^{-1}(j) = f_{|g|(j)}^{-1}(j).
\end{equation}

  \item[(v)] 
Let 
\[
\xymatrix@C = +2.5em@R = +1em{ & S \ar[dd]^(.3){g} \ar[dr]^a & 
\\
T  \ar[ur]^f    \ar@{-}[r]^(.7){b}\ar[dr]_h & \ar[r] & Q \ar[dl]^c
\\
&R&
}
\]
be a commutative diagram in $\ttO$ and let $j\in |Q|, i = |c|(j).$
Then by axiom (iii) the diagram
\[
    \xymatrix@C = +1em@R = +1em{
      h^{-1}(i)      \ar[rr]^{f_i} \ar[dr]_{b_i} & & g^{-1}(i) \ar[dl]^{a_i}
      \\
      &c^{-1}(i)&
    }
\]
commutes, so it induces a morphism $(f_i)_j: b_i^{-1}(j)\to a_i^{-1}(j).$ By axiom (iv) we have $$a^{-1}(j)=a_i^{-1}(j) \ \mbox{and} \ b^{-1}(j)=b_i^{-1}(j).$$
We  then require the equality
$$f_i = (f_i)_j.$$
\end{itemize}
We will also assume that the set $\pi_0(\ttO)$ of connected components
is {\em small\/} with respect to  a sufficiently big ambient universe.
  \begin{remark} It follows from axiom (iii) that  the unique fiber of the canonical morphism $!_T: T\to
  U_c$ is $T$. \end{remark}

A {\em strict  operadic functor\/} between two strict operadic categories is a
functor $F: \ttP \to \ttO$ over $\sFSet$ which preserves fibers in
the sense that $F\big(f^{-1}(i)\big) = F(f)^{-1}(i)$, for any $f : T
\to S \in \ttP$ and $i \in |S| = |F(S)|$.  We also require that $F$
preserves the chosen terminal objects in each connected component, and
equality~(\ref{karneval_u_retardacku}).  This gives the
category ${\OpCat}$ of strict operadic categories and their strict operadic
functors.

\begin{remark} Our notion of operadic category is not invariant under
  categorical equivalences. It is, nevertheless, very convenient and
  sufficient for our applications.  There is a~more general non-strict
  version of the above definition which we are going to consider in 
 a~subsequent paper. We will also consider non-strict operadic functors
  and operadic natural transformations. We hope to prove a coherence
  theorem saying that every general operadic category is equivalent in
  an appropriate sense to a strict one.  Since in this paper we only use
  strict operadic categories we will call them simply operadic
  categories for brevity.
  \end{remark}

\begin{example}
\label{onefirst} 
The terminal category  $\tt 1$ with the cardinality 
functor $|\dash|:{\tt
    1}\to \sFSet$ which sends the unique object of $\tt 1$ to $\bar{1}\in
  \sFSet$ is operadic.
\end{example}

\begin{example}
\label{finset+deltaalg}
The category $\Deltaalg$ of
finite ordinals (including the empty one) has an obvious structure of
an operadic category
\end{example}

\begin{example} 
The categories  $\sFSet$ and $\Deltaalg$  are examples of operator categories in the sense of
  Barwick~\cite{barwick}.  It is easy to see that, in fact, any
  Barwick's operator category is equivalent to an operadic category in our sense.
  Recall that an {\em operator category\/} is an essentially small
  category $\Phi$ which satisfies the following conditions:
\begin{itemize}
\item[(i)] 
the set of morphisms $\Phi(T,S)$ is finite for any pair of objects
$T,S \in \Phi$,
\item[(ii)] 
the category $\Phi$ has a terminal object $1$, and
\item[(iii)] 
there is a pullback $T\times_S i $ of $i$ along $f$ 
for any morphisms $f:T\to S$ and $i: 1\to S$.
\end{itemize} 

To find an equivalent operadic category we take a skeletal version of   $\Phi,$   fix $1$ as 
the chosen terminal object, take $ \Phi(1,T)) = |T|\in \sFSet$ as the 
cardinality of $T$ and choose   pullbacks $T\times_S i , i \in |S|$, as the fiber functors. The rest of the structure is clear. 
\end{example}

\begin{example}
\label{taut1}
Each category $\ttC$ determines the `tautological' operadic
category  $\Cplusone$ which, as a~category, is $\ttC$ with a formally
added terminal object $1$. The cardinality \hbox{$|\dash|
\!:\!\Cplusone \to \sFSet$} is defined  by
\[
|T| := \cases{\bar{0}}{if $T \in \ttC$, and}{\bar{1}}{if $T = 1$.}
\]
The axioms  dictate that 
the only maps that have fibers are $!_T : T \to 1$ with the unique
fiber~$T$, and the identity $\id : 1 \to 1$ whose fiber is $1$.
This construction constitutes a fully 
faithful embedding of the category of small categories 
to the category of operadic categories.  
\end{example}

\begin{example}
  \label{puget}
  Let $\frC$ be a set.  A {\em $\frC$-bouquet\/} is a map $b:
  X\!+\!1\to \frC,$ where $X\in \sFSet.$ In other words,
  a~$\frC$-bouquet is an ordered $(k+1)$-tuple $(c_1,\ldots,c_k;c), X
  = \bar{k}$, of elements of $\frC.$ It can also be thought of as a
  planar corolla whose all edges (including the root) are colored.
  The extra color $b(1) \in \frC$ is called the {\em root color\/}. The finite set
  $X$ is the {\em underlying set\/} of the bouquet~$b$.

  A map of $\frC$-bouquets $b \to c$ whose root colors coincide is an
  arbitrary map $f: X\to Y$ of their underlying sets. Otherwise there
  is no map between $\frC$-bouquets.  We denote the resulting category of
  $\frC$-bouquets by $\Bq(\frC)$.

  The cardinality functor $|\dash|:\Bq(\frC)\to \sFSet$ 
assigns to a bouquet $b :X+1\to \frC$ its underlying
  set~$X$.  The fiber of a map $b \to c$ given by $f : X \to Y$ over
  an element $y\in Y$ is a $\frC$-bouquet whose underlying set is
  $f^{-1}(y),$ the root color coincides with the color of $y$ and the
  colors of the elements are inherited from the colors of the elements
  of $X$.
\end{example}

It is easy to see that $\Bq(\frC)$ is an operadic category with $\frC$
its set of connected components. It is an example of an operadic
category whose fibers are not pullbacks. It has the following important
property:

\begin{proposition}
  \label{Jarca}
  For each operadic category $\ttO$ with $\pi_0(\ttO)=\frC$, there is a
  canonical operadic `arity' functor $\Ar: \ttO\to \Bq(\frC)$ giving
  rise to the factorization
  \begin{equation}
    \label{eq:factor}
\xymatrix@R=1em{&\ttO \ar@/_/[ld]_{\Ar} \ar@/^/[rd]^{|\dash|}   & 
\\
 \Bq(\frC) \ar[rr]^{|\dash|} &&\sFSet
}
\end{equation}
of the cardinality functor $|\dash| : \ttO \to \sFSet$. 
\end{proposition}

\begin{proof}
  Let the {\em source\/} $s(T)$ of $T \in \ttO$ be the set of fibers of
  the identity $\id : T \to T$.  We define $\Ar(T) \in \Bq(\frC)$ as the
  bouquet $b: s(T)+1 \to \frC$, where $b$ associates to each fiber $U_c
  \in s(T)$ 
the corresponding connected component $c \in \frC$, and $b(1)
  := \pi_0(T)$.  We leave as an exercise to check that the assignment
  $T \mapsto \Ar(T)$\/ extends into an operadic functor.
\end{proof}

The assignment $\frC \mapsto \Bq(\frC)$ extends to a functor ${\it
  Bq} : \Set \to \OpCat$, and the functor $\Ar: \ttO\to \Bq(\frC)$
is the initial object of the comma-category $\ttO/{\it Bq}$. 
This explains why the
bouquets will play such a prominent r\^ole among operadic categories.
Indeed, the arity functor will be used to define the
endomorphism operads in Example~\ref{Zitra_do_Osnabrucku}.

\begin{proposition}
  \label{pullback-create-pullbacks}
  Pullbacks in the category of categories create pullbacks in the
  category of operadic categories and operadic functors.
\end{proposition}

\begin{proof}
  Let us consider the ordinary pullback
  \begin{equation}
    \label{eq:1}
    \xymatrix@C = +4em{
      \ttR \ar[r]^r \ar[d]_\varpi&\ttO \ar[d]^\pi
      \\
      \ttQ \ar[r]^p & \ttP
    }
  \end{equation}
  of the diagram $\xymatrix@1{\ttQ\ \ar[r]^p&\ \ttP \ & \ar[l]_\pi\ \ttO}$ of operadic
  categories and their operadic functors. We may assume that objects
  of $\ttR$ are pairs $(\ittau,\itS)$, $\ittau \in \ttO$, $\itS \in \ttQ$, such
  that $\pi(\ittau) = p(\itS)$. Morphisms $(\ittau,\itT) \to (\itkappa,\itS)$
  in $\ttR$ are couples $(\sigma,f)$, where $\sigma : \ittau \to \itkappa$ is
  a morphism in $\ttO$ and $f : \itT \to \itS$ a morphism in $\ttQ$ such that
  $\pi(\sigma) =p (f)$. The functors $r : \ttR \to \ttO$ and $\varpi : \ttR \to \ttQ$ are
  the obvious projections to the first resp.~the second factor.

  We equip $\ttR$ with a structure of an operadic category as follows. We
  define the cardinality functor $|\dash| : \ttR \to \sFSet$ by
  $|(\ittau,\itT)| := |\ittau|$.\footnote{Since $|\ittau| = |\itS|$ we could
    as well put $|(\ittau,\itT)| := |\itT|$.}  The chosen terminal objects are
\[
U_{\pi_0(\ittau,\itS)} :=
  (U_{\pi_0(\ittau)},U_{\pi_0(\itS)}).
\]
The fibers are defined componentwise, i.e.~for a morphism
  $(\sigma,f): (\ittau,\itT) \to (\itkappa,\itS)$ we put
  \[
  (\sigma,f)^{-1}(i) := \big(\sigma^{-1}(i),f^{-1}(i)\big),\ i \in
  |\itkappa| = |\itS|.
  \]
  Notice that, since $p$ and $\pi$ are operadic functors,
  \[
  \pi\big(\sigma^{-1}(i)\big) = (\pi\sigma)^{-1}(i) = (pf)^{-1}(i) =
  p\big(f^{-1}(i)\big),
  \]
  so indeed $(\sigma,f)^{-1}(i) \in \ttR$.

  We leave the verification that the diagram~(\ref{eq:1})
  is indeed a pullback in the category of operadic categories as an
  exercise.
\end{proof}

Pullbacks can be used to define colored versions of operadic
categories. Given an operadic category $\ttO$ and a finite
set $\frC$, we define the operadic category $\ttO^\frC$ of $\frC$-colored objects
in $\ttO$ as the pullback
\begin{equation}
    \label{eq:Jaruska}
    \xymatrix@C = +4em{
      \ttO^\frC \ar[r] \ar[d]& \Bq(\frC) \ar[d]^{|\dash|}
      \\
      \ttO \ar[r]^{|\dash|} & \sFSet
}    
\end{equation} 
Notice that \hbox{$\pi_0(\ttO^\frC) \cong \pi_0(\ttO)\! \times\!
  \frC$}. 
  \begin{remark} 
Since $\sFSet$ is the terminal object in the category of operadic categories, the pullback $\ttO^\frC$
is actually the product $\ttO\times \Bq(\frC) $ in $\tt OpCat.$
\end{remark}

A {\em $\ttO\/$-collection} in $\ttV$ is a collection $E = \{E(T)\}_{T
  \in \ttO}$ of objects of $\ttV$ indexed by the objects of the
category $\ttO$. The category of $\ttO$-collections in $\ttV$ will be
denoted $\ColV{\ttO}$. For a $\ttO\/$-collection $E$ and a morphism $f:T\to S$ 
let $$E(f) = \bigotimes_{i \in |S|} E({T_i})$$ 
In the following definition  we tacitly use
equalities~(\ref{karneval_u_retardacku}).

\begin{definition} 
An $\ttO$-operad is  a collection  $\calP = \{\calP(T)\}_{T
\in \ttO}$ in $\ttV$ together with units
  \[
  I\to \calP(U_c),\ c \in \pi_0(\ttO),
  \]
and structure maps
  \[
  \mu(f): \calP(f) \otimes \calP(S)\to \calP(T),\ f:T\to S,
  \]
satisfying the following axioms.

  \begin{itemize}
  \item[(i)] Let $T \stackrel f\to S \stackrel g\to R$ be morphisms in
    $\ttO$ and $h := gf : T \to R$ as in~(\ref{Kure}).  Then the
    following diagram of structure maps of $\calP$ combined with the
    canonical isomorphisms of products in $\ttV$ commutes:
    \[
    \xymatrix@C = 2em@R = .4em{
\ar@/^2.5ex/[rrd]^(.56){\hskip .5em\bigotimes_{i}\mu(f_i) \ot \id}
      \ar[dd]_(.45){\id \ot \mu(g)}
\displaystyle\bigotimes_{i
        \in |R|} 
      \calP(f_i) \ot \calP(g) \ot \calP(R) & &    
      \\  & &  \ar@/^/[dl]_{\mu(h)}
     \calP(h) \ot \calP(R)\ .
\\
      \ar[r]^(.77){\mu(f)}{\rule{0pt}{2em}}  
\displaystyle\bigotimes_{i \in |R|}
      \calP(f_i) \ot \calP(S) \cong   \calP(f) \ot \calP(S)&
 \calP(T)&
    }
    \]
  \item[(ii)] The composition
    \[
    \xymatrix@1@C = +2em{ \calP(T) \ar[r]& \rule{0pt}{2em} \displaystyle
      \bigotimes_{i\in |T|} I\! \ot\! \calP(T) \ar[r]&\rule{0pt}{2em}
      \displaystyle \bigotimes_{i\in |T|} \calP(U_{c_i})\! \ot\! \calP(T)\ar[r]^(.56)= &
      \calP(\id_T)\! \ot\! \calP(T)
 \ar[r]^(.65){\mu(\id)}&\calP(T) }
    \]
    is the identity for each $T \in \ttO$, as well as the identity is
  \item[(iii)] the composition
    \[
    \xymatrix@1@C = +2.2em{ \calP(T)\! \ot \! I \ \ar[r]&\ \calP(T)\! \ot\!
      \calP(U_c) \ \ar[r]^=
&\ \calP(!_T)\! \ot\! \calP(U_c)\
      \ar[r]^{\hskip 1.8em \mu(!_T)}&\ \calP(T)},\ c:= \pi_0(T).
    \]
  \end{itemize}
\end{definition}

Notice that for an arbitrary operad $\calP$ and $c\in \pi_0(\ttO)$,
$\calP(U_c)$ with the multiplication
\[
\mu(\id) : \calP(U_c) \ot \calP(U_c) \to \calP(U_c)
\]
forms a unital monoid in $\ttV$.

A {\em morphism\/}  $\varsigma : \calP' \to \calP''$  of $\ttO$-operads in $\ttV$ is a
collection $\{\varsigma_T\}_{T \in \ttO}$ of $\ttV$-morphisms  
\hbox{$\varsigma_T : \calP'(T) \to \calP''(T)$} commuting with the
structure operations.  
We denote by $\OpV \ttO$ the category of
$\ttO$-operads in $\ttV$.  Each operadic functor $F:\ttO\to \ttP$ 
obviously induces the restriction $F^*: \OpV P\to \OpV \ttO$.

We can put the definition of $\ttO$-operad in a $2$-categorical context as follows\footnote{We were inspired by the definition of a
non-symmetric operad as a strict monoidal lax-functor $\Deltaalg
\to \Sigma \ttV$ given by Day and Street in~\cite{day-street:lax}.}.
Let $\Sigma \ttV$ denote the symmetric
monoidal bicategory with one object $\star$ and $\ttV$ as its category
of morphisms $\star \to \star$. 
Recall that a part of a  lax-functor structure on $\calP$ from a category $\ttO$ to the bicategory  $\Sigma \ttV$ are morphisms
$$\calP(f)\otimes \calP(g)\to \calP(h)$$
for each commutative diagram like (\ref{Kure}), as well as morphisms
$I\to \calP(\id).$ For such a lax-functor and an object $T\in \ttO$ we denote
$\calP(T) := \calP(T\stackrel{!_{T}}{\to}U_c)$. 

\begin{definition} 
\label{jeste_musim_koupit_parecky} 
 An {\em operad-like} functor from $\ttO$ to $V$ is a lax-functor $\calP : \ttO
  \to \Sigma\ttV$ equipped, for each \hbox{ $f:T\to S$} with fibers $T_i :=
  f^{-1}(i)$, $i \in |\itS|$, with an isomorphism
\begin{equation}\label{structure-iso}
\calP(f) \cong \bigotimes_{i \in |S|} \calP({T_i})
\end{equation}  
which satisfies the following axioms:
\begin{itemize}

\item[(i)] 
For any commutative diagram (\ref{Kure}) the following diagram commutes
\[
\xymatrix{
\calP(f)\otimes\calP(g)
\ar[r]\ar[d]_{\cong}
&
\calP(h)
\ar[d]^{\cong}
\\
\displaystyle
\bigotimes_{j\in |S|}\calP\big(f^{-1}(j)\big) \otimes  
\bigotimes_{i\in |R|}\calP\big(g^{-1}(i)\big)\ar[d]_{id}
& \displaystyle\bigotimes_{i\in |R|}\calP\big(h^{-1}(i)\big)
\\ 
\displaystyle\bigotimes_{j\in |S|}\calP\big(f_{|g|(j)}^{-1}(j)\big) 
\otimes  \bigotimes_{i\in |R|}\calP\big(g^{-1}(i)\big)\ar[r]^(.47){\cong}
& \displaystyle \ar[u] \bigotimes_{j\in |R|}
\Big(\bigotimes_{|g|(j)=i}\calP\big(f_{i}^{-1}(j)\big) 
\otimes  \calP\big(g^{-1}(i)\big)\Big)
}
\]
where the bottom vertical arrow  on the left side 
is an identity due to equality
\begin{equation}
\label{eq:39}
\calP\big(f_{|g|(j)}^{-1}(j)\big)\stackrel=\lra \calP\big(f^{-1}(j)\big)
\end{equation}
and the up-going 
right vertical arrow is given by the
lax-constraints induced by commutative diagrams:
\[
    \xymatrix@C = +1em@R = +1em{
      h^{-1}(i)      \ar[rr]^{f_i} \ar[dr]_{!} & & g^{-1}(i)\ar[dl]^{!}
      \\
      &U_c&
    }.
\]
\item[(ii)] For each object $T\in \ttO$, the following diagram commutes
\[
\xymatrix@C = +3em{
I
\ar[d]_{} 
\ar[r]^{ } 
& \calP(\id_T)
\ar[d]^{\cong} 
\\
\displaystyle\bigotimes_{i\in |T|} \calP(U_{c_i})
\ar[r]^{ =}
& \displaystyle\bigotimes_{i\in |T|} \calP(\id_T^{-1}(i))\ .}
\]
\end{itemize}
\end{definition}
The proof of the following lemma is straightforward:
\begin{lemma}The category $\OpV \ttO$ is equivalent to the category of operad-like functors from $\ttO$ to $V$ and their lax-natural transformations which commute
with the structure isomorphisms~(\ref{structure-iso}).\end{lemma}
\begin{example}
\label{monoids}
Operads over the terminal category $\tt 1$ of Example~\ref{onefirst} 
are monoids in~$V$.
\end{example}

\begin{example}
\label{finset+deltaalgalgeb}
The category of operads over the category $\sFSet$ is isomorphic to the category of classical symmetric operads. Operads over $\Deltaalg$ are ordinary non-$\Sigma$
operads~\cite[Sec.~3, Prop.~3.1]{batanin:AM08}. Applying the construction 
of diagram~(\ref{eq:Jaruska}) to the operadic
category $\Deltaalg$  we obtain the category
$\Deltaalg^\frC$ describing $\frC$-colored non-$\Sigma$-operads. More
examples of this construction will be given in
Section~\ref{sec:oper-categ-levell}.
\end{example}

\begin{example}
\label{zitra}
An operad over the category $\Cplusone$ from Example~\ref{taut1} is the
same as a monoid $M = \calP(1)$ in $\ttV$, together with the `actions'
\begin{equation}
\label{chobotnicka}
\mu(!_T) : M \ot \calP(T) \to \calP(T), \ T \in \ttC,
\end{equation}
and a contravariant functor $\Phi : \ttC \to \ttV$ such that the maps
\[
\Phi(f) : = \calP(f) : \calP(T) \to \calP(S),\ f:T \to S \in \ttC,
\]
commute with the actions~(\ref{chobotnicka}).
 In
particular,  $\Cplusone$-operads with $\calP(1) = I$ are precisely presheaves
$\ttC^{\tt op} \to \ttV$. 
\end{example}

\begin{example}
  \label{end}
  Operads over the category $\Bq(\frC)$ of $\frC$-bouquets introduced in
  Example~\ref{puget} are ordinary $\frC$-colored operads. Therefore, for
  each $\frC$-colored collection $E = \{E_c\}_{c\in \frC}$ of objects of
  $\ttV$ one has the {\em endomorphism $\Bq(\frC)$-operad\/}
  $\End^{\Bq(\frC)}_E$, namely the ordinary colored 
  endomorphism operad~\cite[\S1.2]{berger-moerdijk:CM07}.
\end{example}

\begin{example}
  The category $\OpV \ttO$ of $\ttO$-operads in $\ttV = (\ttV,\ot, I)$ has a
  monoidal structure given by the 'componentwise' multiplication in
  $\ttV$. The unit for this structure is the operad $\bfone^\ttO$ with
  $\bfone^\ttO(T) := I$ for each $T\in \ttO$. Clearly $F^*(\bfone^P) =
  \bfone^\ttO$ for any operadic functor
  $F:\ttO\to P$.
\end{example}

\begin{example}
\label{Zitra_do_Osnabrucku}
  For a $\frC$-colored collection $E = \{E_c\}_{c\in \frC}$ in $\ttV$ and an
  operadic category $\ttO$ with $\pi_0(\ttO) = \frC$, one defines the {\em
    endomorphism $\ttO$-operad\/} $\End^\ttO_E$ as the restriction
  \[
  \End^\ttO_E := \Ar^*\big(\End^{\Bq(\frC)}_E\big)
  \]
  of the $\Bq(\frC)$-endomorphism operad of Example~\ref{end} along the
  arity functor $\Ar$~of Proposition~\ref{Jarca}.
\end{example}

\begin{definition}
\label{Jaruska_ma_chripecku}
An {\em algebra\/} over an $\ttO$-operad $\calP$ in $\ttV$ is a
collection $A = \{A_c\}_{c\in \pi_0(\ttO)}$, $A_c \in \ttV$, equipped with an
$\ttO$-operad map $\alpha : \calP \to \End^\ttO_A$. 
\end{definition}

An algebra is thus
given by suitable structure maps
\[
\alpha_T : \bigotimes_{c \in \pi_0(s(T))} A_c \ot \calP(T) \to
A_{\pi_0(T)}, \ T \in \ttO,
\]
where $s(T)$ denotes, as before, the set of fibers of the identity
$\id: T\to T$. This notion of $\calP$-algebras will further be
generalized in Section~\ref{Jaruska_ma_streptokoka}.

\begin{example}
  The category $\Gamma$ of stable labelled
  graphs~\cite[Section~7]{markl:handbook} is an operadic category.
  Morphisms are given by contractions of subgraphs. The cardinality functor
  associates to a graph its set of vertices. Fibers of a morphism are
  the subgraphs contracted to a vertex.

  If $\ttV$ is the category of differential graded vector spaces, then
  $\Gamma$-operads are precisely {\em hyperoperads\/} in the sense of
  \cite{getzler-kapranov:CompM98}.  Algebras over these operads are
  (twisted) modular operads, see~\cite{getzler-kapranov:CompM98} or
  \cite[Def.~II.5.5]{markl-shnider-stasheff:book}.
\end{example}

\begin{example} This is our only example of a large operadic
  category. Let $\tt A$ be an abelian category and let $\tt Epi(A)$ be
  its subcategory of epimorphisms. The cardinality functor on $\tt
  Epi(A)$ maps all objects to the   one point set $\bar{1}.$ The (unique)
  fiber of any morphism is its (chosen) kernel. It is easy to check that
  this defines an operadic category structure on $\tt Epi(A).$ An $\tt
  Epi(A)$-operad in the category of vector spaces is the same as a
  charade over $\tt A$ in Kapranov's sense
  \cite[Definition~3.2]{kapranov:langlands}.\footnote{Generally  speaking, $\tt Epi(A)$  is not a strict operadic category but we can easily ``strictify'' it if we use a skeletal version.} 
\end{example}

\subsection{The category of $k$-trees}
\label{Omegan}
We are going to recall briefly the category
$\ttOmega_k$ of $k$-trees, for $k \geq
0$; the details can be found in~\cite[Sec.~3,
Example~8]{batanin:globular} or~\cite{batanin-street}. 
The category of {\em $0$-trees $\ttOmega_0$} is the terminal
category~$1$. Its unique object is denoted by $\EuU_0$.

The category of {\em $1$-trees $\ttOmega_1$} is the category of finite
ordinals $(n) := \{\rada 1n\}$, $n \geq 0$, and their order-preserving
maps.  As usual, we interpret $\{1,\ldots,n\}$  for $n=0$ as the empty~set.
The terminal object of $\ttOmega_1$ is $\EuU_1 := (1)$. When the
meaning is clear from the context, we will simplify the notation and
denote the object $(n) \in \ttOmega_1$ simply by $n$. The category
$\ttOmega_1$ is isomorphic to the operadic category $\Deltaalg$
recalled in Example~\ref{finset+deltaalg}.

Let $k \geq 2$. A {\em $k$-tree\/} is a chain 
\begin{equation}
\label{k-tree} 
\EuT =
(\xymatrix{n_k \ar[r]^(.4){t_{k-1}} & n_{k-1}
  \ar[r]^(.5) {t_{k-2}} &\ \cdots \ \ar[r]^(.6) {t_{1}} & n_1})
\end{equation}
of morphisms in $\ttOmega_1$. A morphism
\begin{equation}
\label{Puska}
\sigma:
(\xymatrix{n_k \ar[r]^(.4){t_{k-1}} & n_{k-1}
  \ar[r]^(.5) {t_{k-2}} &\ \cdots \ \ar[r]^(.6) {t_{1}} & n_1})
\longrightarrow (\xymatrix{m_k \ar[r]^(.4){s_{k-1}} & m_{k-1}
  \ar[r]^(.5) {s_{k-2}} &\ \cdots \ \ar[r]^(.6) {s_{1}} & m_1})
\end{equation}
of $k$-trees is a diagram in $\Set$
\[
\xymatrix{n_k \ar[r]^(.4){t_{k-1}} \ar[d]_{\sigma_k}   & n_{k-1}
\ar[d]_{\sigma_{k-1}} 
  \ar[r]^(.5) {t_{k-2}} &\ \cdots \ \ar[r]^(.6) {t_{1}} & n_1
\ar[d]^{\sigma_1} 
\\
m_k \ar[r]^(.4){s_{k-1}} & m_{k-1}
  \ar[r]^(.5) {s_{k-2}} &\ \cdots \ \ar[r]^(.6) {s_{1}} & m_1
}
\]
such that
\begin{itemize}
\item[(i)] 
$\sigma_1$ is order preserving and
\item[(ii)]  
for any $p$, $k > p \geq 1$, and  $i\in n_p$, the restriction of
$\sigma_{p+1}$ to $t_p^{-1}(i)$ is order-preserving.
\end{itemize}

We denote by $\ttOmega_k$ the category of $k$-trees and their
morphisms as defined above. Its terminal
object is the $k$-tree $\EuU_k := (1\to 1 \to \cdots \to
1)$, its initial object is $z^k\EuU_0 := (0\to 0 \to \cdots \to
0)$. Notice that we have the obvious {\em truncation\/} functors
\[
\xymatrix{
\ttOmega_k  \ar[r]^{\tr_{k-1}}& \ttOmega_{k-1} \ar[r]^{\tr_{k-2}} &
\cdots  \ar[r]^{\tr_{2}}&  \ttOmega_2  \ar[r]^(.35){\tr_{1}}  & \ttOmega_1
= \Deltaalg}.
\]
One also has the {\em suspension\/} functor $z :\ttOmega_k \to
\ttOmega_{k+1}$, $k \geq 0$, that to an $k$-tree as in~(\ref{k-tree})
associates the $(k+1)$-tree
\[ 
z\EuT: =
(\xymatrix{n_k \ar[r]^(.4){t_{k-1}} & n_{k-1}
  \ar[r]^(.5) {t_{k-2}} &\ \cdots \ \ar[r]^(.6) {t_{1}} & \ar[r]
  n_1 &1}).
\]

An {\em $s$-leaf\/} (or a {\em leaf of height $s$\/}) of a $k$-tree $\EuT$ as
in~(\ref{k-tree}) is, for $s=k$, by definition 
an element of $n_k$.  For $1 \leq s <
k$ an $s$-leaf is an element $i \in n_s$ such that $t^{-1}_s(i) =
\emptyset$. We denote by $\Leaf_s(\EuT)$ the set of all $s$-leaves of
$\EuT$.

Let $\sigma : \EuT \to \EuS$ be a map of $k$-trees as in~(\ref{Puska}) and
$i \in m_k = \Leaf_k(\EuS)$ a $k$-leaf of $\EuS$. Let us define the fiber
$\sigma^{-1}(i)$ over $i$ as the chain
\begin{equation}
\label{Eva_reditelem?}
\sigma^{-1}(i) :=
\big(\xymatrix{\sigma^{-1}_k(i) \ar[r] &  \sigma^{-1}_{k-1}s_{k-1}(i))
  \ar[r] &\ \cdots \ \ar[r] & \sigma^{-1}_{1}(s_1 \cdots s_{k-1}(i)) }\big)
\end{equation}
of the restrictions of the maps in~(\ref{k-tree}). Analogously one may
define also fibers over the $s$-leaves for $s < k$, but we will not use
them in this article.

The category $\ttOmega_k$ with the cardinality functor
$|\EuT| := \Leaf_k(\EuT)$, with the fibers defined
as above and the chosen terminal object the $k$-tree $\EuU_k$, 
is an operadic category. 
Operads in $\OpV {\ttOmega_k}$ are precisely $k$-operads in the monoidal
globular $k$-category $\Sigma^k \ttV$,
see~\cite[\S11.3]{batanin-markl:Adv}.

\subsection{The category of $k$-ordinals}
\label{nordinal}
Let, as in the Section \ref{Omegan} , $k \geq 0$.
Recall~\cite[Sec.~II]{batanin:conf} that a {\em $k$-ordinal\/} is a 
finite set $\EuO$  equipped with $k$ binary relations  $<_0, \ldots, <_{k-1}$
such that 
\begin{itemize}
\item[(i)] 
 $<_p$ is nonreflexive,
\item[(ii)] 
for every pair $a,b$ of distinct elements of $\EuO$ there
exists exactly one $p$ such that
\[
a<_p b \ \ \mbox{or} \ \ b<_p a,
\]
\item[(iii)] 
if $a<_p b $  and  $b<_q c$  then
$a<_{\min(p,q)} c$.
\end{itemize}
A morphism of $k$-ordinals
$\sigma: \EuO \rightarrow \EuN$ is a map 
of the underlying sets such that  
$i<_p j$ in $\EuO$  implies that
\begin{itemize}
\item[(i)]  
$\sigma(i) <_r \sigma(j)$ for some $r\ge p$, or
\item [(ii)]  
$\sigma(i)= \sigma(j)$, or
\item [(iii)]
$\sigma(j) <_r \sigma(i)$ for $r>p$.
\end{itemize}

Let $\Ordk$ be the skeletal category of $k$-ordinals and their morphisms. 
The  category $\Ordk$ is operadic.
The cardinality $|\dash|: \Ordk \to \sFSet$ associates to a
$k$-ordinal $\EuO$ its underlying set (denoted $\EuO$ again). The fiber
of a map $\sigma: \EuO \to \EuN$ over $i \in \EuN$ is
the preimage $\sigma^{-1}(i)$ with the induced structure of a
$k$-ordinal. The category $\Ordk$ is connected, the unique terminal
object being the one-point $k$-ordinal $1_k$.
Operads in $\OpV \Ordk$ are pruned  $k$-operads in the
monoidal globular $n$-category $\Sigma^k \ttV$ \cite{batanin:conf}.

There is a natural $k$-ordinal structure on the set  of
$k$-leaves of 
each $k$-tree in $\ttOmega_k$. Let $\EuT$ be as in~(\ref{k-tree}) and  
$a,b \in n_k = \Leaf_k(\EuT)$ its distinct $k$-leaves. We say that $a <_p b$ if
$a$ precedes $b$ in $n_k$ and $p$ is such that
\[
t_pt_{p-1} \cdots t_{k-1} (a) = t_pt_{p-1} \cdots t_{k-1} (b)
\]
but
\[
t_{p-1} \cdots t_{k-1} (a) \not= t_{p-1} \cdots t_{k-1} (b).
\]
If such a $p$ does not exist, we put $a <_0 b$. It is easy to show
that this construction extends to 
an operadic functor  
\begin{equation}
\label{pruning}
p : \ttOmega_k \to \Ordk.
\end{equation}

On the other hand, $\Ordk$ can be identified with the full
subcategory of $\ttOmega_k$ consisting of  pruned trees.
Recall that a $k$-tree $\EuT \in \ttOmega_k$ as in~(\ref{k-tree}) 
is {\em pruned\/} if all $\rada{t_{k-1}}{t_1}$ are
epimorphisms. Equivalently, all leaves of $\EuT$ are its $k$-leaves, so
$\EuT$ is ``fully grown.''
For example, the $2$-ordinal
\[
0 <_0 2,\ 0 <_0 3,\ 0 <_0 4,\ 1 <_0 2,\  1<_0 3,\ 1 <_0 4,\ 
0 <_1 1,\ 2 <_1 3,\ 2 <_1 4,\ 3 <_1 4, 
\]
is represented by the following pruned tree
\[
\psscalebox{.4 .4} 
{
\begin{pspicture}(0,-1.7822068)(6.81,1.7822068)
\psline[linecolor=black, linewidth=0.05](1.06,-0.35779327)(3.06,-1.7577933)(5.06,-0.35779327)(4.96,-0.35779327)
\psline[linecolor=black, linewidth=0.05](0.06,1.0422068)(1.06,-0.35779327)(2.06,1.0422068)(2.06,1.0422068)
\psline[linecolor=black, linewidth=0.05](3.46,1.0422068)(5.06,-0.35779327)(6.66,1.0422068)(6.66,1.0422068)
\psline[linecolor=black, linewidth=0.05](5.06,1.0422068)(5.06,-0.35779327)(5.06,-0.35779327)
\rput[b](0.06,1.5422068){\Huge 0}
\rput[b](2.06,1.5422068){\Huge 1}
\rput[b](3.46,1.5422068){\Huge 2}
\rput[b](5.06,1.5422068){\Huge 3}
\rput[bl](6.66,1.5422068){\Huge 4}
\end{pspicture}
}
\]
See \cite[Theorem~2.1]{batanin:conf} for a more detailed discussion.  
We thus have an inclusion of categories $l : \Ordk \hookrightarrow
\ttOmega_k$ which  is left adjoint to the pruning
functor~(\ref{pruning}). 

It will sometimes be useful to identify a 
$k$-ordinal $\EuO \in \Ordk$ with the corresponding 
pruned tree $l(\EuO)
\in \ttOmega_k$. The functor $p : \ttOmega_k \to \Ordk$ then appears as the {\em
  pruning\/} associating to each $\EuT \in \ttOmega_k$ its
maximal pruned subtree. We must emphasize 
that  $l : \Ordk \hookrightarrow \ttOmega_k$ \underline{is} \underline{not} an 
operadic functor, since it does not preserve fibers in general.  It is only lax operadic in an
appropriate sense.

\section{Discrete fibrations of  operadic categories}

\begin{definition}
\label{psano_v_Myluzach}
  An operadic functor $F:\ttO\to \ttP$ is called a {\it discrete operadic
    fibration} if 
\begin{itemize}
\item[(i)]
$F$ induces an epimorphism $\pi_0(\ttO) \twoheadrightarrow
\pi_0(\ttP)$ and
\item[(ii)]
for any morphism $f : \itT\to \itS$ in $\ttP$ and any
  $\ittau_i, \itkappa \in \ttO$, where $i \in |\itS|$, such that
  \[
  F(\itkappa) = \itS \mbox { and } F(\ittau_i) = f^{-1}(i) \mbox { for } i
  \in |S|,
  \]
  there exists a unique $\sigma : \ittau\to \itkappa$ in $\ttO$ such that
  \[
  F(\sigma) = f \mbox { and } \ittau_i = \sigma^{-1}(i)\ \mbox { for } i
  \in |S|.\footnote{Notice that $F(\itkappa) = \itS$ implies $|\itkappa| =
    |\itS|$ and $F(\sigma) = f$ implies $F(\ittau) = \itT$.}
  \]
\end{itemize}
\end{definition}

We have the following simple

\begin{lemma}
\label{Vecer_jdu_ke_trem_kralum}
A discrete operadic fibration $F:\ttO\to \ttP$ 
induces an isomorphism \hbox{$\pi_0(\ttO) \stackrel\cong\to
\pi_0(\ttP)$}. 
\end{lemma}

\begin{proof}
Assume that $V_{a'}$,  $V_{a''}$, $a',a'' \in \pi_0(\ttO)$, 
are chosen terminal objects in components of the category
$\ttO$ such that $F(V_{a'}) =  F(V_{a''}) = U_c$ for $c \in
\pi_0(\ttP)$. 
Then (ii) of Definition~\ref{psano_v_Myluzach}
taken with $T = S = U_c$, $f: U_c \to U_c$ 
the identity map, $t_1 =  V_{a'}$ and $s =
V_{a''}$ produces a map $\sigma : t \to   V_{a''}$ whose unique
fiber $\sigma^{-1}(1)$ is $V_{a'}$. Since  ${\Fib}_1:\ttO/V_{a''} \to
\ttO$ 
is the domain functor
by (iii) of the definition of the operadic
  category, $t = V_{a'}$, so $\sigma$ is in fact a map   $V_{a'} \to
  V_{a''}$, therefore  $V_{a'}$ and $V_{a''}$ belong to the
  same component of $\ttO$, i.e.\ $a' = a''$. 
\end{proof}

For an $\ttO$-operad $\calO \in \OpV \ttO$ and $\itT \in \ttP$ put
\begin{equation}
  \label{eq:7}
  F_!\calO(\itT) := \coprod_{F(\ittau) = \itT} \calO(\ittau).
\end{equation}

\begin{proposition}
  Assume that $F$ is a discrete operadic fibration. Then~(\ref{eq:7})
  is the underlying collection of a naturally defined $\ttP$-operad.
\end{proposition}

\begin{proof}
  For each $f : \itT \to \itS \in \ttP$ and $\itT_i := f^{-1}(i)$, $i \in
  |\itS|$, we need the structure map
  \[
  \mu(f) : \bigotimes_{i \in |\itS|} F_!\calO(\itT_i) \ot F_!\calO(\itS) \to
  F_!\calO(\itT).
  \]
  Expanding the definition of $F_!\calO$ and invoking the distributivity
  of the monoidal product over coproducts we see that it is the same
  as to give a map
  \[
  \mu(f) : \coprod_\doubless{F(\ittau_j) = \itT_j,\ j \in
    |\itS|}{F(\itkappa) = \itS} \bigotimes_{i \in |\itS|} \calO(\ittau_i) \ot
  \calO(\itkappa) \longrightarrow \coprod_{F(\itomega) = \itT} \calO(\itomega).
  \]
  It clearly amounts to specifying, for each $\ittau_i$'s and $\itkappa$
  as above, a map
  \begin{equation}
    \label{eq:8}
    \bigotimes_{i \in |\itS|}
    \calO(\ittau_i) \otimes \calO(\itkappa) \longrightarrow
    \coprod_{F(\itomega) = \itT} \calO(\itomega).
  \end{equation}

  The defining property of discrete operadic fibrations provides a
  unique $\ittau \in \ttO$ with $F(\ittau) = \itT$, and a morphism $\sigma :
  \ittau \to \itkappa$ such that $\sigma^{-1} (i) = \ittau_i$ for $i \in
  |\itS|$. We then choose (\ref{eq:8}) to be the composition
  \[
  \bigotimes_{i \in |\itS|} \calO(\ittau_i) \ot \calO(\itkappa) \stackrel {\mu(
    \sigma)}\vlra \calO(\ittau) \stackrel{\siota_\ittau}\longrightarrow
  \coprod_{F(\itomega) = \itT} \calO(\itomega),
  \]
  where ${\mu(\sigma)}$ is the structure map of the $\ttO$-operad $\calO$ and
  $\liota_\ittau$ the canonical map to the coproduct.

  Let $U_c$, $c \in \pi_0(\ttP)$, be a chosen terminal object of a
  component of $\ttP$.  By Lemma~\ref{Vecer_jdu_ke_trem_kralum}, there
  is a unique chosen terminal object $V_a$, $a \in \pi_0(\ttO)$, 
such that $F(V_a) = U_c$.
We define the unit $I \to F_!\calO(U_c)$ as the composition
\[
I \lra \calO(V_a) \lra \coprod_{F(\ittau) = U_c} \calO(\ittau) =
F_!\calO(U_c). 
\] 
of the unit map for $\calO$ with the coprojection.
  Since all constructions were functorial lifts of the operad
  structure of $\calO$, the resulting structure is an operad again.
\end{proof}

\begin{theorem}
  \label{Jaruska}
  Let $F:\ttO\to \ttP$ be a discrete operadic fibration of operadic
  categories. 
Then the assignment
  $\calO \mapsto F_!\calO$ described above defines a left adjoint $F_!:
  \OpV{\ttO}\to \OpV{\ttP}$ to the restriction functor \hbox{$F^*: \OpV \ttP\to
    \OpV \ttO$}.
\end{theorem}

\begin{proof}
  We need to establish, for $\calO \in \OpV \ttO$ and $\calP \in \OpV \ttP$, a
  natural isomorphism
  \[
  \OpV \ttP\big(F_!\calO,\calP\big) \cong \OpV \ttO\big(\calO,F^*\calP\big).
  \]
 There is an isomorphism of the sets of
  morphisms of collections
  \begin{equation}
    \label{eq:9}
    \ColV \ttP\big(F_!\calO,\calP\big) \cong \ColV \ttO\big(\calO,F^*\calP\big).
  \end{equation}
  This follows immediately since~(\ref{eq:7}) 
is the formula for the left Kan extension along the induced functor  
${\it Ob}(F): {\it Ob}(\ttO)\to {\it Ob}(\ttP)$
between the discrete categories of objects.

  The proof is finished by showing that a morphism of collections in
  the left hand side of~(\ref{eq:9}) is an operad morphism (i.e.~it
  commutes with the operad structure maps) if and only if the
  corresponding morphism in the right hand side of~(\ref{eq:9})
  does. We leave this as an exercise.
\end{proof}

Let now $\ttV$ be the monoidal category $\Set$ of sets,  $F:\ttO\to
\ttP$ a 
discrete operadic fibration and $\bfone^\ttO \in
\OpSet \ttO$ the 
terminal $\ttO$-operad with $\bfone^\ttO(\ittau) = \{\ittau\}$.
Theorem~\ref{Jaruska} gives the operad 
\begin{equation}
\label{eq:22}
\calO := F_!(\bfone^\ttO) \in \OpSet \ttP
\end{equation}
with $\calO(\itT) =\big\{\ittau \in \ttO;\ F(\ittau) = \itT\big\}$.
\def\bigtimes#1{\mathop{\hbox{\LARGE $\times$}}\limits_{\mbox{\scriptsize $#1$}}}

Vice versa, assume that one is given an operad $\calO \in \OpSet \ttP$. One
then has the category $\ttO$ whose objects are $\ittau \in \calO(\itT)$ for
some $\itT \in \ttP$. A morphism $\sigma : \ittau \to \itkappa$ from $\ittau
\in \calO(\itT)$ to $\itkappa \in \calO(\itS)$ is a couple $(\varepsilon,f)$
consisting of a morphism $f : \itT \to \itS$ in $\ttP$ and of some
$\varepsilon \in \bigtimes{i \in |\itS|} \calO\big(\ittau_i)$, $\ittau_i :=
f^{-1}(i)$, such that
\[
\mu(f)(\varepsilon,\itkappa) = \ittau,
\]
where $\mu$ is the structure map of the operad $\calO$.  Compositions of
morphisms are defined in the obvious manner.  The category $\ttO$ thus
constructed is clearly an operadic category such that the functor $F :
\ttO \to \ttP$ given by
\[
\mbox{$F(\ittau) := \itT$ for $\ittau \in \calO(\itT)$ and $F(\varepsilon,f)
  := f$}
\]
is a discrete operadic fibration. We call this construction the {\it
  operadic Grothendieck construction}.  It 
is a direct generalization of the classical Grothendieck
construction for presheaves~\cite[p.~44]{maclane-moerdijk} 
as the following proposition shows.

\begin{proposition}
  \label{sec:discr-fibr-betw}
The above construction establishes an equivalence between the category
$\OpSet \ttP$ of\/ $\ttP$-operads in  $\Set$ and the category of discrete
operadic fibrations over $\ttP$.
\end{proposition}

\begin{proof}
A direct verification.
\end{proof}

\begin{example}
As we saw in Example~\ref{zitra},
$\Set$-operads $\calP$ in the operadic category $\Cplusone$ of
Example~\ref{taut1} such that $\calP(1) = 1$ are presheaves on
$\ttO$. The restriction of the correspondence
of Proposition~\ref{sec:discr-fibr-betw} to operads $\calP$
with this property
is an equivalence between the category of presheaves on $\ttO$ and the
category of discrete 
fibrations $\ttP \to \ttC$ of
categories. Proposition~\ref{sec:discr-fibr-betw} therefore indeed
generalizes the discrete version of the Grothendieck construction.
\end{example}

\begin{proposition}
  \label{sec:nejdou_fonty}
  Assume that in the pullback~(\ref{eq:1}), the functor $\pi : \ttO \to \ttP$
  is a discrete operadic fibration. Then $\varpi : \ttR \to \ttQ$ is a discrete
  operadic fibration, too.
\end{proposition}

\begin{proof}
  We rely on the notation in the proof of
  Proposition~\ref{pullback-create-pullbacks}. Suppose we are given
  a~morphism $f: \itT \to \itS$ in $\ttQ$ with fibers $\itT_i :=
  f^{-1}(i)$, $i \in |\itS|$. Suppose we are also given objects
  $(\ittau_i,\widetilde \itT_i)$ and $(\itkappa,\widetilde\itS)$ of the
  category $\ttR$ such that $\itS = \varpi(\itkappa,\widetilde\itS)$ and $\itT_i
  = \varpi(\ittau_i,\widetilde\itT_i)$ for each $i \in |\itS|$.  We must find
  a unique $(\sigma,\widetilde f) : (\ittau, \widetilde \itT) \to
  (\itkappa,\widetilde \itS)$ in $\ttR$ such that $\varpi(\sigma,\widetilde f) =
  f$ and $(\sigma,\widetilde f )^{-1} (i) = (\ittau_i,\widetilde
  \itT_i)$ for each $i \in |\itS|$.

  It follows from definitions that $\itS = \varpi(\itkappa,\widetilde\itS)$
  implies $\widetilde \itS = \itS$, $\itT_i =
  \varpi(\ittau_i,\widetilde\itT_i)$ implies $\widetilde \itT_i = \itT_i$ and
  $\varpi(\sigma,\widetilde f) = f$ implies $\widetilde f = f$. Since
  $(\itkappa,\widetilde\itS)$ and $(\ittau_i,\widetilde\itT_i)$ are
  objects of $\ttR$ we see that $\pi(\itkappa) = p(\itS)$, and $\pi(\ittau_i) =
  p(\itT_i)$ for all $i \in |\itS|$. Similarly, we conclude that
  $\pi(\sigma) = p(f)$.

  We therefore need to prove the following statement.  Given $f: \itT
  \to \itS$ in $\ttQ$ with fibers $\itT_i := f^{-1}(i)$ and objects
  $\ittau_i, \itkappa$ of $\ttO$, $i \in |\itS|$, such that $\pi(\itkappa) =
  p(\itS)$, and $\pi(\ittau_i) = p(\itT_i)$, there exists a unique $\sigma
  :\ittau \to \itkappa$ in $\ttO$ such that $\sigma^{-1}(i) = \ittau_i$ and
  $\pi(\sigma) = p(f)$.  The above statement however follows from the
  lifting property in the discrete operadic fibration $\pi : \ttO \to \ttP$
  applied to the data $p(f) : p(\itT) \to p(\itS) \in \ttP$ and
  $\itkappa,\ittau_i \in \ttO$, $i \in |\itS|$. This finishes the proof.
\end{proof}

We close this section by proving that squares of adjoint functors
between the associated categories of operads induced by pullbacks of 
discrete operadic fibrations satisfy
the Beck-Chevalley property~\cite[p.~205]{maclane-moerdijk}. Therefore
operads over operadic categories behave similarly as presheaves over
small categories.

\begin{proposition}
  \label{Beck-Chevalley}
  If in the pullback of operadic categories~(\ref{eq:1}) the functor
  $\pi : \ttO \to \ttP$ is a discrete operadic fibration, then the induced functors
  between the associated categories of
  operads satisfy the Beck-Chevalley condition, meaning that there is a
  natural isomorphism
  \begin{equation}
    \label{eq:2}
    \varpi_!(r^*(\calP))\cong p^*(\pi_!(\calP))
  \end{equation}
  for any operad $\calP \in \OpV \ttO$.
\end{proposition}

\begin{proof}
  Notice that the functor $\varpi : \ttR \to \ttQ$ is a discrete operadic
  fibration by Proposition~\ref{sec:nejdou_fonty}. We will use the
  explicit description of the left adjoint along discrete fibrations
  provided by Theorem~\ref{Jaruska}.  For $\itT \in \ttP$ it gives
  \[
  \pi_!(\calP)(\itT) = \coprod_{\pi(\ittau) = \itT} \calP(\ittau),
  \]
  therefore, for $\itS \in \ttQ$,
  \[
  p_* \big(\pi_!(\calP)\big)(\itS) = \pi_!(\calP)(p(\itS)) \cong \coprod_{\pi(\ittau)
    = p(\itS)} \calP(\ittau).
  \]
  On the other hand,
  \[
  \varpi_!\big(r^*(\calP)\big)(\itS) = \coprod_{\varpi(\ittau,\widetilde \itS) =
    \itS}r^*(\calP)(\ittau,\widetilde \itS) \cong \coprod_{\varpi(\ittau,\widetilde
    \itS) = \itS}\calP\big(r(\ittau, \widetilde \itS)\big) =
  \coprod_{\pi(\ittau) = p(\itS)} \calP(\ittau),
  \]
  therefore indeed $p_* \big(\pi_!(\calP)\big)(\itS) \cong
  \varpi_!\big(r^*(\calP)\big)(\itS)$ for each $\itS \in \ttQ$. In the last
  display we used the fact that $\varpi(\ittau,\widetilde \itS) = \itS$
  implies $\widetilde \itS = \itS$ so, since $(\ittau,\widetilde \itS)
  \in \ttR$, we have the equality $\pi(\ittau) = p(\itS)$.
\end{proof}

\section{$\ttO$-multi(co)tensors and generalized algebras of $\ttO$-operads}
\label{Jaruska_ma_streptokoka}

  We start by showing how the standard notion of a
multitensor on a $\ttV$-category
$\ttC$, see~\cite[Def.~2.1]{batanin-weber} or \cite{batanin:AM08,DS2},  generalizes to the realm
of operadic categories.  Let $\End_{\tt{C}}$ be the endomorphism
$\sFSet$-operad of $\ttC$, so $\End_{\ttC}(n)$ is, for $n \geq 0$,
the category of $\ttV$-functors $\ttC^{\otimes n}\to \ttC.$
The restriction along the cardinality functor $|\dash|: \ttO
\to \sFSet$ gives a categorical $\ttO$-operad $|\End_{\ttC}|^*$.  Let
${\it 1}_\ttO$ be the terminal categorical $\ttO$-operad.  An
$\ttO$-multitensor on $\ttC$ is then a lax-morphism of categorical
operads ${\it 1}_\ttO\to |\End_{\ttC}|^*$. Unpacking this definition
we obtain:

\begin{definition} \label{multitensor}
An {\em $\ttO$-multitensor\/} on
a $\ttV$-category $\ttC$ is an $\ttO$-collection $E = \{E_T\}_{T \in \ttO}$  of $\ttV$-functors
\[
E_T: \underbrace{\tt{C}\otimes\cdots\otimes \tt{C}}_{|T|- times}
\rightarrow \tt{C},\    T \in \ttO,
\]
equipped with
\begin{itemize}
\item[(i)]
$\ttV$-natural transformations
\[
\mu_{f}:E_S(E_{T_1},\ldots,E_{T_k})\rightarrow E_{T}
\]
defined for any $f:T\to S$ in $\tt O$ with fibers $T_1,\ldots,T_k$, and  
\item[(ii)] 
$\ttV$-natural transformations (the units)
\[
\eta_c:
\id \rightarrow E_{U_c},\ c\in \pi_0(\tt{O}),
\]
\end{itemize}
satisfying the obvious  associativity and unitality  conditions.
\end{definition}

Multitensors create operads, as shown in the following lemma whose
simple proof we leave to the reader. 
For an object $T \in\ttO$, let
$\pi_0s(T)$ denote the set of connected components of the
fibers of the identity $\id : T\to T$, and $\pi_0(T)$ the connected
component of $T$, cf.~the proof of Proposition~\ref{Jarca}.

\begin{lemma}
\label{Opicka}
\label{coendo}
Let $E$ be an  $\ttO$-multitensor on
a $\ttV$-category $\ttC$ and $\delta = \{\delta(i)\}_{i \in
  \pi_0(\ttO)}$  an arbitrary collection in $\ttC$. Then
the collection
\[
\Coend^E_\delta = \big\{\Coend^E_\delta(T)\big\}_{T \in \ttO}
\]
with $\Coend^E_\delta(T)$ the enriched hom
\[
\ttC\big(\delta(i),E_T(\rada{\delta(i_1)}{\delta(i_k)})\big),\
 i = \pi_0(T),\ \{\Rada i1k\} = \pi_0s(T),
\]
is a natural $\ttO$-operad in $\ttV$.
\end{lemma}

Dually, we introduce $\ttO$-{\em multicotensors\/} on $\ttC$ as 
colax-morphisms of the categorical \hbox{$\ttO$-operads} ${\it
  1}_\ttO\to |\End_{\ttC}|^*$. An explicit definition can be obtained
by inverting arrows in Definition~\ref{multitensor}.
Multicotensors create operads in a similar way as multitensors do:

\begin{lemma} 
\label{endo}
Let $D$ be an  $\ttO$-multicotensor on
a $\ttV$-category $\ttC$ and  $X = \{X(i)\}_{i \in
  \pi_0(\ttO)}$  an arbitrary collection of objects of $\ttC$. Then
the collection
\[
\End^D_X = \big\{\End^D_X(T)\big\}_{T \in \ttO}
\]
with $\End^E_X(T) := \ttC\big(D_T(\rada{X(i_1)}{X(i_k)}),X(i) \big)$
and $\Rada i1k,i \in \pi_0(\ttO)$ having the meaning as in Lemma~\ref{Opicka},
is a natural $\ttO$-operad in $\ttV$.
\end{lemma}
  
\begin{example}
\label{one} 
Let $\tt 1$ be the terminal category of
Example~\ref{onefirst}.   
A $\tt 1$-(co)multitensor on $\ttC$ is the same as a
$V$-(co)monad on $\ttC$.

For an arbitrary operadic category $\ttO$ and $c\in \pi_0(\ttO)$,
there is an operadic functor ${\tt 1}\to \ttO$ which sends the unique
object of $\tt 1$ to $U_c$. Restricting along
this functor we verify that $E_{U_c}$ (resp.~$D_{U_c}$) is a $V$-monad
(resp.~$V$-comonad) on $\ttC$ for an arbitrary $\ttO$-multitensor $E$
(resp.~multicotensor~$D$).
\end{example}
  
\begin{example}\label{odot} 
  The tensor product $\odot$ of a symmetric monoidal $V$-category
  $\ttC$ gives rise to a $\sFSet$-multitensor $\bigodot$ on $\ttC$,
  which is simultaneously  a $\sFSet$-multicotensor on $\ttC$.
  Namely, for a finite set $S$ of cardinality $n$ and $\Rada X1n \in
  \ttC$ we put 
\[
\textstyle
{\bigodot}_T(X_1,\ldots,X_n) :=
\bigodot_{i\in S}X_i.
\]
For any operadic category $\ttO$ we then have an $\ttO$-multitensor
${\bigodot}^{\ttO}$ on $\ttC$
(which is also a $\ttO$-multicotensor) 
given by restricting $\bigodot$ along the cardinality functor:
\[
\textstyle
{\bigodot}^{\ttO}_T(X_1,\ldots,X_n): =
{\bigodot}_{|T|}(X_1,\ldots,X_k),\ T \in \ttO.
\]
The case  $\ttC =V$ of the above construction along with Lemmas~\ref{coendo}
and~\ref{endo} explains the standard fact that an object of a
symmetric monoidal category $V$ has both the (classical) endomorphism
and coendomorphism operads, see \cite{markl-shnider-stasheff:book},
Definitions~II.1.7 and II.1.9.
\end{example}

\begin{definition} 
\label{dnes_obed_s_Magdou}
  Let $D$ be a fixed $\ttO$-multicotensor and
  $\calP$  an $\ttO$-operad.  An {\em algebra of $\calP$ in~$\ttC$\/} is a
  $\pi_0(\ttO)$-collection  $A$ in $\ttC$ with a morphism of $\ttO$-operads
$$\calP\to \End^D_A.$$
\end{definition}

\begin{example}
  If $\ttC = \ttV$ and $D = \bigodot^\ttO$, then $\calP$-algebras in
  the sense of the above definition are the same as $\calP$-algebras
  of Definition~\ref{Jaruska_ma_chripecku}.
\end{example}

Assume that $D$ is  a $\ttO$-multicotensor and suppose that
$\ttC$ is 
cocomplete as a $V$-category. This means, in particular, that there is
a left action functor $V\otimes \ttC \to \ttC$ which we  denote
by $\otimes$, believing that it would not 
lead to confusion.  We have an
adjunction
\[
\ttC(B\otimes X,Y )\cong V(B,\ttC(X,Y))
\]
for any $B\in V$ and $X,Y\in \ttC$.

Since $D$ is a $V$-functor in each variable, there is for each $T\in
\ttO$, $|T|=n$, $\Rada X1n \in \ttC$, and $B \in \ttV$, a $V$-natural
transformation
\[
B\otimes D_T(X_1,\ldots,X_i,\ldots,X_n) \lra D_T(X_1,\ldots,B\otimes X_i,\ldots,
X_n),\  1\le i\le n,
\]
called the {\it strength} of $D$.  We will require that $D$ interacts
with the left $\ttV$-action in such a way that the following
conditions are satisfied:
\begin{itemize}\item[(i)] 
The diagram 
 \[
    \xymatrix@C = +1em@R = +1.28em{
    B\otimes D_{U_c}(X)      \ar[rr] \ar[dr]_{B\otimes \eta_c} 
& & D_{U_c}(B\otimes X) \ar[dl]^{\eta_c}
      \\
      &B\otimes X&
    }
\]
commutes for any $c\in \pi_0(\ttO)$, $X \in \ttC$ and $B \in \ttV$.
\item[(ii)]
Let  $f:T\to S$ be a morphism in $\ttO$ with fibers
$T_1,\ldots,T_k$. Let
$n = |T|$, and $1 \leq i \leq n$, $1 \leq s \leq k$ be such that $i \in
|T|$ belongs to $|T_s|$. 
Then the diagram
\[
\xymatrix@C = -12.2em@R=+1em{
B\otimes D_T(\Rada X1i,\ldots,X_n) \ar[rr] \ar[d]
&&B\otimes D_S\big(D_{T_1}(X_1,\ldots),\ldots,D_{T_k}(\ldots,X_n)\big)
\ar[ddl]^{}
\\
D_T(X_1,\ldots, B\otimes X_i,\ldots,X_n)
 \ar[rd]
 & &
\\
&
 D_S\big(D_{T_1}(X_1,\ldots),\ldots,D_{T_s}
(\ldots,B\otimes X_i,\ldots ),\ldots, D_{T_k}(\ldots,X_n)\big)&
}
\]
commutes  for any $B \in \ttV$ and $\Rada X1n \in \ttC$.
\end{itemize}

\begin{definition} 
\label{Jarka_ma_streptokoka}
  An $\ttO$-multicotensor $D$ on a cocomplete $V$-category $\ttC$
  satisfying properties (i) and (ii) above will be called {\em
    strong\/}.  We will call a $V$-category $\ttC$  equipped with a
  strong multitensor $D$ for which the comonad $D_{U_c}$ of
  Example~\ref{one} is the identity comonad for each 
$c \in \pi_0(\ttO)$ a {\em colax $\ttO$-monoidal $V$-category\/}.
\end{definition}

\begin{remark} 
  The terminology above has been adapted from the classical definition of a
  strong monad on a closed monoidal category \cite{AK}. Indeed, when
  $\ttO$ is the terminal category~$\tt 1$, 
the notion of a strong $\ttO$-multicotensor coincides
  with the notion of a strong comonad, cf.~Example \ref{one}.
\end{remark}

For a $\pi_0(\ttO)$-collection $X$ in $\ttC$ and $T\in \ttO$ denote 
\[
X^T: =
D_T\big(\rada{X(i_1)}{X(i_k)}\big), \  \{\Rada i1k\} = \pi_0s(T).
\]
The structure of a $\calP$-algebra on a $\pi_0(\ttO)$-collection  $A$ in
a colax monoidal $V$-category $\ttC$ can be expressed in terms of
an action defined as a collection of morphisms
$$\calP(T)\otimes A^T\to A(i), \ T \in \ttO, \ i = \pi_0(T),$$
satisfying the following conditions:
\begin{itemize}
\item[(i)] 
The square 
\[
\xymatrix@C = +3.2em{I\otimes A^{U_c}
  \ar[r]^{\cong}
\ar[d]
&A^{U_c}\ar[d]^{\eta_c}
\\
\calP(U_c)\otimes A^{U_c}
 \ar[r]
 & 
 A}
\]
in which $\eta_c$ is the counit of the comonad $D_{U_c}(-)= (-)^{U_c}$,
commutes for any $c\in \pi_0(\ttO).$

\item[(ii)]  For any morphism $f:T\to S$ in $\tt O$ with fibers
$T_1,\ldots,T_k$, the following diagram in which $i = \pi_0(T)$ commutes
\[
\xymatrix{
\calP(S)\!\otimes\! \calP(T_1)\!\otimes\cdots\ot \calP(T_k)\!\otimes\! A^T
\ar[r]\ar[d]
&
\calP(S)\!\otimes\! \calP(T_1)\!\otimes\cdots\ot \calP(T_k)\!\otimes\!
D_S(A^{T_1},\ldots,A^{T_k})\ar[d]
\\
\calP(T)\!\otimes\! A^T\ar[d] & \calP(S)\!\otimes\! D_S(\calP(T_1)\!\otimes\!
A^{T_1},\ldots,\calP(T_k)\!\otimes\! A^{T_k})\ar[d]
\\
A(i) & \calP(S)\!\otimes\! A^S \ar[l]
}
\]
\end{itemize}

\section{Convolution and condensation}

The condensation described in this section creates,
in a controlled manner, out of colored operads and their algebras,
non-colored ones. 
The main statement of this section is Proposition~\ref{v_patek_do_Prahy}.

Fix a finite set $\frC$ and consider the pullback $\ttO^\frC$
in~(\ref{eq:Jaruska}). Objects of $\ttO^\frC$ can be 
interpreted as objects of $\ttO$
colored by elements of $\frC$. A typical object of $\ttO^\frC$ will
therefore be denoted by $T(i_1,\ldots,i_k;i)$, 
where $k:= |T|$ and $\Rada
i1k,i \in \frC$.

The operadic category $\ttO^\frC$ contains a full operadic subcategory $\ttLO^\frC$\footnote{``{\tt L}'' abbreviating ``linear.''} whose objects are objects of $\ttO^\frC$ of the form 
 $U_c(i;j)$, $i,j \in \frC$, $c\in \pi_0(\ttO).$
It is easy to see that $\ttLO^\frC$-operads are precisely $\pi_0(\ttO)$-families of
$\ttV$-enriched categories with the set of objects $\frC$.
In other words,
$\OpV{\ttLO^\frC} \cong \ttCat_\frC(\ttV)^{\pi_0(\ttO)}$, where $\ttCat_\frC(\ttV)$
denotes the category of $\ttV$-categories with the set of objects $\frC$,
and the set $\pi_0(\ttO)$ is considered as a
discrete category.

The inclusion $\liota: {\ttLO^\frC}\hookrightarrow {\ttO^\frC}$ induces the restriction
functor $\liota^* : \OpV {\ttO^\frC} \to   \OpV{\ttLO^\frC}$ 
between the categories of operads. We define the functor $\ttU :
\OpV {\ttO^\frC} \to \ttCat(\ttV)$ as the composition
\[
\ttU : \xymatrix@C=1.5em{
 \OpV {\ttO^\frC} \ar[r]^(.33){\hbox{$\iota$}^*} & \OpV{\ttLO^\frC} \cong
\ttCat_\frC(\ttV)^{\pi_0(\ttO)} \ar[rr]^(.65){\coprod_{c \in
  \pi_0(\ttO)}} && \ttCat(\ttV)},
\] 
where the coproduct exists because $\ttV$ is cocomplete.

\begin{definition} 
We will call the $\ttV$-enriched category $\ttU(\calP)$ the {\em underlying
    category\/} of an $\ttO^\frC$-operad $\calP$ in $\ttV$.
\end{definition}  

Explicitly, the 
set of objects of $\ttU(\calP)$ is 
\hbox{$\pi_0(\ttO)\! \times\! \frC$}, while the enriched hom-set between
$(c,i) \in\hbox{$\pi_0(\ttO)\! \times\! \frC$}$ and $(d,j)
\in\hbox{$\pi_0(\ttO)\! \times\! \frC$}$ is
\[
\ttU(\calP)\big((c,i),(d,j)\big) 
=  
\cases{\hbox{the initial object $0$ of $\ttV$}}
{if $c\ne d$, and}{\calP\big(U_c(i;j)\big)}{otherwise.}
\]
Since $U_c(i;i)$ are the chosen terminal objects of $\ttO^\frC$, we have
the identity morphisms
\[
I\to \ttU(\calP)\big((c,i),(c,i)\big), 
\ (c,i) \in \hbox{$\pi_0(\ttO)\! \times\! \frC$},
\]
given by the operad units of $\calP$. The composition in $\ttU(\calP)$ is
induced by the (unique) maps between objects of the form $U_c(i;j)$,
$i,j \in \frC$, $c \in \pi_0(\ttO)$.

Let $\ttC$  be a colax $\ttO$-monoidal  $V$-category with a
multicotensor $D$ as in
Definition~\ref{Jarka_ma_streptokoka}.  
We are going to define the convolution product on 
the  category~$\ttC^{\ttU(\calP)}$. 
Assume first
for simplicity that $\ttO$ is connected so that the objects of $\ttU(\calP)$ are
elements of~$\frC$.  For any
$T\in \ttO$, $k = |T|$, the operad $\calP$ generates a $\ttV$-functor
\[
\mathbb{\calP}(T): \underbrace{\ttU(\calP)^{\it op}\otimes\cdots\otimes
  \ttU(\calP)^{\it op}}_{k-times}\otimes\, \ttU(\calP) \rightarrow \ttV
\]
defined on objects $i_1,\ldots,i_k,i \in (\ttU(\calP)^{\it op})^{\ot
  k} \ot \ttU(\calP)$ by
\begin{equation}
\label{Ralph}
\mathbb{\calP}(T)(i_1,\ldots,i_k;i) := \calP\big(T(i_1,\ldots,i_k;i)\big).
\end{equation}
To see how $\mathbb{\calP}$ acts on morphisms,
we observe that in $\tt O^\frC$ we have morphisms of the
form 
\[
T(i_1,\ldots,i_k;i)\lra T(j_1,\ldots, j_k;i),\
\Rada i1k,\Rada j1k,i \in \frC,
\] 
whose underlying morphism in $\tt O$ is the identity of $T$. The 
multiplication in the $\ttO^\frC$-operad $\calP$ with
respect to these morphisms induces the contravariant part of the functor
$\mathbb{\calP}(T)$. The covariant part is induced by the multiplication
in $\calP$
corresponding to the morphisms in $\ttO^\frC$ of the form
\[
T(i_1,\ldots,i_k;i) \to U_c(n;i),\ \Rada i1k, i ,n \in \frC.
\]
 
For $T$ as above and $\Rada X1k \in \ttC^{\ttU(\calP)}$,
the {\em convolution product\/} $\E D^\calP_T(X_1,\ldots,X_k) \in \ttC^{\ttU(\calP)}$  
is given by the  coend in $\ttC$:
\begin{equation}
\label{convolut}
\E D^\calP_T(X_1,\ldots,X_k)(i) = 
\int^{\ttU(\calP)^{\otimes n}}\hskip -.6em 
\mathbb{\calP}(T)(i_1,\ldots,i_k;i)\otimes
D_T\big(X_1(i_1),\ldots, X_k(i_k)\big), \  i \in \ttU(\calP).
\end{equation}
The above constructions easily extend to the case of an arbitrary
$\pi_0(\ttO)$, we just take the products of the corresponding
constructions for the individual connected components.

The convolution~(\ref{convolut}) generalizes the Day-Street
convolution product~\cite{DS2} of a substitude related to an ordinary
(that is $\sFSet$-) operad.  
If $\ttC$ is a symmetric monoidal $V$-category and $D=
\bigodot^{\ttO}$ is the  \hbox{$\ttO$-multicotensor} from
Example \ref{odot}, many standard facts about the convolution remain valid
in this generalized situation. For example we have

\begin{proposition} 
\label{Musim_na_prohlidku}
The convolution product determines an $\ttO$-multitensor $\E\odot^\calP 
=\{\E\odot^\calP_T\}_{T \in \ttO}$ on 
$\ttC^{\ttU(\calP)}$ whose unital part $\prod_{c \in \pi_0(\ttO)}
\E\odot^\calP_{U_c}$ 
is the identity monad.
\end{proposition}

\begin{remark}
The $\ttO$-multitensor structure of  $\E\odot^\calP$ is induced by the
$\ttO$-multitensor structure of the $\ttO$-multicotensor
$\bigodot^\ttO$. The fact that $\bigodot^\ttO$ is both a
multicotensor {\em and\/} a multitensor makes this situation very
special -- the convolution product~(\ref{convolut}) need not be a
$\ttO$-multitensor for general $D$. To simplify the notation, we will sometimes
write $E^\calP$ instead of  $\E\odot^\calP$. 
\end{remark}

An $\ttO$-multicotensor $D$ on $\ttC$ induces an
$\ttO^\frC$-multicotensor 
$D^\frC$ on $\ttC$ by `forgetting the colors'
\begin{equation}
\label{Zitra_s_chlapci}
D^\frC_{T(\Rada i1k;i)} := D_T, \ T(\Rada i1k;i) \in \ttO^\frC
\end{equation}
which in turn restricts to an  $\ttLO^\frC$-multicotensor $LD^\frC$. 
Since, by assumption, $D_{U_c}$ is the identity comonad for each $c
\in \pi_0(\tt0)$, the colax unitality of multicotensors gives natural morphisms
\begin{equation}
\label{Kdy_se_jarka_uzdravi}
LD^\frC_{U_c(i;j)}(X) = D_{U_c}(X) \stackrel\cong 
\lra X, \ U_c(i;j) \in \ttLO^\frC.
\end{equation}

Let now $A$ be an algebra of an $\ttO^\frC$-operad $\calP$ in $\ttC$
as in Definition~\ref{dnes_obed_s_Magdou}. 
The action $\alpha:\calP\to
\End_A^{D^\frC}$ induces, via the restriction along the
inclusion $\liota : \ttLO^\frC \hookrightarrow \ttO^\frC$ combined
with the canonical transformations~(\ref{Kdy_se_jarka_uzdravi}), an action
\[
\liota^*(\alpha):
\liota^*(\calP)\lra  \liota^*(\End_A^{D^\frC}) =  \End_A^{LD^\frC} 
\lra  \End_A^{\id},
\]
where $\id$ is the obvious identity $\ttLO^\frC$-multitensor on $\ttC$.
This gives rise to a functor 
\begin{equation}
\label{Jarka_slusne_prudi}
\liota^*: \Alg_\calP(\ttC)\to \ttC^{\ttU(\calP)}
\end{equation}
from the category of $\calP$-algebras in $\ttC$ to the functor
category $\ttC^{\ttU(\calP)}$.  If $\ttO$ is connected, the algebra
$A$ is a collection $\{A(i)\}_{i \in \frC}$ in $\ttV$, the functor
$\liota^*(\alpha)$ takes $i \in \frC$ to $A(i)$, and the map
\[
\liota^*(\alpha) : \calP\big(U(i;j)\big) \to \ttC\big(A(i),A(j)\big)
\] 
of the enriched hom-sets
is given by the
$\calP$-algebra structure of~$A$. The description of $\liota^*(\alpha)$ for
a general $\pi_0(\ttO)$ is similar.

In the particular case when
$\ttC$ is a symmetric monoidal $V$-category,
$E^\calP = \E\odot^\calP$ is a multitensor on $\ttC^{\ttU(\calP)}$ by
Proposition~\ref{Musim_na_prohlidku}. 
It turns out that $\liota^*(A)$ is an
algebra over  $E^\calP$ in the sense of the following definition.

\begin{definition}
An {\em algebra over an $\ttO$-multitensor\/} $E$ on a   $\ttV$-category $\ttC$
is an object $Z \in \ttC$ equipped with a family of morphisms
\[
\alpha_T : E_T(\rada ZZ) \to Z,\ T \in \ttO,
\]
such that
\begin{itemize}
\item[(i)]
the composition
$
\xymatrix@1{Z\ \ar[r]^(.34){\eta_c} & 
\ E_{U_c}(Z)\ \ar[r]^(.64){\alpha_{U_c}} &\ Z} 
$
is the identity for each $c \in \pi_0(\ttO)$, and
\item[(ii)]
the diagram
\[
\xymatrix@C = +3.2em{E_S\big(E_{T_1}(\rada ZZ),\ldots,E_{T_k}(\rada ZZ)\big)
  \ar[r]^(.7){\mu_\sigma} 
\ar[d]_{E_S(\rada{\alpha_{T_1}}{\alpha_{T_k}})}
& E_T(\rada ZZ)\ar[d]^{\alpha_T}
\\
E_S(\rada ZZ) \ar[r]^{\alpha_S} & Z
}
\]
commutes  for any morphism $\sigma:T\to S$ in $\tt O$ with fibers
$T_1,\ldots,T_k$. 
\end{itemize}
\end{definition}

As in the classical case~\cite[Proposition 1.8]{batanin-berger} we obtain

\begin{proposition} 
\label{problemy_tentokrat_s_okem}
The functor $\liota^*: \Alg_\calP(\ttC)\to
\ttC^{\ttU(\calP)}$ induces an isomorphism between the
category of $\calP$-algebras in $\ttC$ and the category of  $E^\calP$-algebras in
$\ttC^{\ttU(\calP)}$. 
\end{proposition}

Let us return to a general colax 
$\ttO$-monoidal $V$-category $\ttC$. Assume that, in addition, $\ttC$
is complete as a $V$-category, so we have cotensors
$Y^{\alpha} \in \ttC$ such that 
\[
\ttC(X,Y^\alpha) \cong \ttV\big(\alpha,\ttC(X,Y)\big)\ \mbox { for } \ X,Y \in
\ttC,\ \alpha \in \ttV.
\]
For an $\calP$-algebra $A$ in $\ttC$ and an object $\delta$ of the functor category
$\ttV^{\ttU (\calP)}$ we define the {\em
  $\delta$-totalization\/} of $A$ as the end
\begin{equation}
\label{totalization}
\Tot_{\delta}(A):= 
\int_{i\in \ttU(\calP)} \big(\iota^*(A)(i)\big)^{\delta(i)}.
\end{equation}
Since $E^\calP = \E\odot^\calP$ is a multitensor,
the  $\ttO$-collection $\Coend^\calP_{\delta}$ with
\[
\Coend^\calP_{\delta}(T) := 
\int_{i\in \ttU(\calP)}\ttV\big(\delta(i), E^\calP_T(\delta,\ldots,\delta)(i)\big)
\] 
is an $\ttO$-operad in $\ttV$ 
by  Lemma~\ref{Opicka}. 
The main statement of this section reads:

\begin{proposition}
\label{v_patek_do_Prahy}
Let $\calP$ be an $\ttO^\frC$-operad in $\ttV$ and $A$ a $\calP$-algebra  in $\ttC$.
Then the $\pi_0(\ttO)$-collection 
$\Tot_{\delta}(A)$
is a natural $\Coend^\calP_{\delta}$-algebra. 
\end{proposition}

\begin{definition}
We will call the $\ttO$-operad $\Coend^\calP_{\delta}$ the
{\em $\delta$-condensation\/} of $\calP$ and its algebra $\Tot_{\delta}(A)$ the {\it
  $\delta$-totalization\/} of $A.$
\end{definition}

\begin{remark}
  If $\ttC = \ttV$, the $\delta$-totalization~(\ref{totalization}) 
  can be given by a simplified formula
\[
\Tot_{\delta}(A):= 
\int_{i\in \ttU(\calP)} \ttV\big({\delta(i)} , \iota^*(A)(i)\big).
\] 
Proposition~\ref{v_patek_do_Prahy} in this case can be proved using
Proposition~\ref{problemy_tentokrat_s_okem} by a straightforward
generalization of~\cite[Proposition~1.5]{batanin-berger}, see also the
appendix to~\cite{batanin-berger-markl}.
\end{remark}

\begin{proof}[Proof of Proposition~\ref{v_patek_do_Prahy}]
We will assume for simplicity that $\ttO$ is connected, the general case can be
handled similarly. The action
\begin{equation}
\label{Dnes_setkani_reditelu}
\Coend^\calP_{\delta}(T)\otimes \Tot_{\delta}(A)^T \lra
\Tot_{\delta}(A),\ T\in \ttO,
\end{equation}
where $\Tot_{\delta}(A)^T = D_T\big(
\Tot_{\delta}(A),\ldots,\Tot_{\delta}(A)\big)$,
will be constructed  using a natural morphism
 \begin{equation}
\label{action}
\mho_T:  \Tot_{\delta}(A)^T\lra \int_{i \in\ttU(\calP)}
 \E D^{\calP}_T(\liota^*A,\ldots,\liota^* A)(i)^{E^{\calP}_T(\delta,\ldots,\delta)(i)}.
\end{equation}

To simplify the notation, we will implicitly assume that the symbols
$\Rada i1n,\ \Rada j1n$ and $i$ denote objects of the underlying
category $\ttU(\calP)$. We will also drop $\liota^*$ from the
notation, writing $A$ instead of $\liota^* A$.
Notice that the target of~(\ref{action}) is equal to 
\[ 
\int_{i} \int_{i_1,\ldots,i_n}
\E D^{\calP}_T(A,\ldots,A)(i)^{\calP(T)(i_1,\ldots,i_n;i)
\otimes\delta(i_1)\otimes\ldots\otimes
\delta(i_n)},\ n = |T|,
\]
so the morphism~(\ref{action}) is the same as a family of
morphisms
\[
\Tot_{\delta}(A)^T\to
\E D^{\calP}_T(A,\ldots,A)(i)^{\calP(T)(i_1,\ldots,i_n;i)
\otimes\delta(i_1)\otimes\ldots\otimes
  \delta(i_n)}, 
\] 
which satisfy some obvious naturality conditions. 
By adjunction, this amounts to a family of morphisms 
\begin{eqnarray*}
\lefteqn{
  {\calP(T)(i_1,\ldots,i_n;i)\otimes\delta(i_1)\otimes\cdots\otimes
    \delta(i_n)} \otimes   \Tot_{\delta}(A)^T} 
\\
&&\hskip 2em\lra
\calP(T)(i_1,\ldots,i_n;i)\otimes D_T
\Big(\delta(i_1)\otimes  \hskip -.3em \int_{j_1} \hskip -.1em
A(j_1)^{\delta(j_1)},\ldots,  \delta(i_n)\otimes \
\hskip -.3em
\int_{j_n}  \hskip -.1em A(j_n)^{\delta(j_n)}\Big)  
\\ 
 &&\hskip 2em \lra
 \int^{i_1,\ldots i_n} \calP(T)(i_1,\ldots,i_n;i)\otimes
 D_T\big(A(i_1),\ldots,A(i_n)\big).
\end{eqnarray*} 
We also have, for each $1\le k\le n$, the evaluation morphisms
 $$\delta(i_k)\otimes \int_{j_k} 
A(j_k)^{\delta(j_k)}\lra A(i_k)$$
which induce a map
\begin{eqnarray*}
\calP(T)(i_1,\ldots,i_n;i)\otimes D_T
\Big(\delta(i_1)\otimes  \hskip -.3em \int_{j_1} \hskip -.1em
A(j_1)^{\delta(j_1)},\ldots,  \delta(i_n)\otimes \
\hskip -.3em
\int_{j_n}  \hskip -.1em A(j_n)^{\delta(j_n)}\Big)  
\hskip 1em
\\
\hskip 1em \lra  \calP(T)(i_1,\ldots,i_n;i)\otimes
 D_T\big(A(i_1),\ldots,A(i_n)\big).
\end{eqnarray*}
Composing this map with the canonical coprojection to the coend
\begin{eqnarray*}
\lefteqn{
 \calP(T)(i_1,\ldots,i_n;i)\otimes
 D_T\big(A(i_1),\ldots,A(i_n)\big) \hskip 3em}
\\ 
&&
\hskip 3em \lra
 \int^{i_1,\ldots,i_n} \calP(T)(i_1,\ldots,i_n;i)\otimes
 D_T\big(A(i_1),\ldots,A(i_n)\big)
\end{eqnarray*}
we obtain the required morphism~(\ref{action}). 
The necessary naturality conditions for the above construction is obvious. 
 
Let us finally explain how the map $\mho_T$ in~(\ref{action})
leads to 
the action~(\ref{Dnes_setkani_reditelu}). 
To this end, 
observe that a $\calP$-algebra structure on $A$ generates a morphism
\begin{equation}
\label{Zitra_vedeni}
\E D^{\calP}_T(A,\ldots,A)(i)= 
\int^{i_1,\ldots,i_n}
\calP(T)(i_1,\ldots,i_n;i)\otimes 
D_T\big(A(i_1),\ldots,A(i_n)\big) \to A(i).
\end{equation} 
The action~(\ref{Dnes_setkani_reditelu} is then the composition 
\begin{eqnarray*} 
\lefteqn{
\Coend^\calP_{\delta}(T)\otimes \Tot_{\delta}(A)^T = 
 \Big(\int_{i}V\big(\delta(i),
 E^\calP_T(\delta,\ldots,\delta)(i)\big)\Big)\otimes
 \Tot_{\delta}(A)^T\stackrel{\id \ot \mho_T}\vlra}
\\
&&
 \hskip 5em   \int_{i}
V\big(\delta(i), E^\calP_T(\delta,\ldots,\delta)(i)\big)\otimes     
    \int_{i} \E
    D^{\calP}_T(A,\ldots,A)(i)^{E^{\calP}_T(\delta,\ldots,\delta)(i)}
 \lra 
\\
&&
 \hskip 5em  
\int_{i} \E D^{\calP}_T(A,\ldots,A)(i)^{\delta(i)} \lra \int_i
A(i)^{\delta(i)} = Tot_{\delta}(A)
\end{eqnarray*}
whose last arrow is generated by~(\ref{Zitra_vedeni}).
We leave a straightforward but tedious verification 
that this composition is, indeed, a $\calP$-algebra structure to the reader.  
\end{proof}

\part{Duoidal Deligne's conjecture}
\label{part2}

\section*{Reminders}

In this part we work with many  
particular examples of operadic categories
and their operads. We included
Table~\ref{za_tyden_z_Bonnu_do_Prahy} to simplify the reader's
navigation through them. 
\begin{table}[t]
  \centering
\begin{tabular}{llll}
\multicolumn{4}{c}
  {\bf Categories}
\\ [.3 em]
\hline
\hbox{Name:} & \hbox{type of objects:} &
\hbox{introduced in:} & \hbox{typical object:}
\\
\hline
\rule{0pt}{1.5em}%
$\ttOmega_k$ & \hbox {Batanin's $k$-trees} &
\hbox {Section~\ref{Omegan}} 
&  \hbox {$\EuT,\EuS,\ldots$}
\\
\hbox {$\OmegakN$} & \hbox {$\bbN$-colored $k$-trees} & 
\hbox {Section~\ref{OmegaN}}&  \hbox
{$\frT = (\EuT,c),\ \frS= (\EuS,c)$}
\\
\hbox {$\Ord_2$} & \hbox {$2$-ordinals} & \hbox {Section~\ref{nordinal}}
&  \hbox {$\EuO,\EuN$}
\\
\hbox {$\OrN$} & \hbox {$\bbN$-colored $2$-ordinals} & 
\hbox {Section~\ref{OmegaN}}&  
\hbox {$\frO = (\EuO,c), \ \frN = (\EuN,c)$}
\\
\hbox{$\LTr$} & \hbox{trees with levels}&\hbox{Section~\ref{LTr}} & 
\hbox{$\beta,\alpha,\ldots$}  
\\
\hbox {$\Tam^\bbN_2$} & $\OrN$-labelled trees  &\hbox
{Section~\ref{sec:tamark-tsyg-oper}} &\hbox {$(\frO,\delta)$, $(\frN,\gamma)$}
\\
\hbox {$\Tm^\bbN_2$} & $\OmegaN$-labelled trees &Section~\ref{sec:tamark-tsyg-oper}& 
$(\frT,\delta)$, $(\frS,\gamma)$
\\ 
$\pro$ & trees with labeled \bily-vertices&Proof of Lemma~\ref{tam-as-tree}&$\zeta$
\\
$\pr$ &  trees with labeled \bily-vertices&Proof of Lemma~\ref{tam-as-tree}&$\xi$
\\
\hline 
\multicolumn{4}{c}
  {\rule{0pt}{2em}\bf Operads:}
\\ 
 [.3 em]
\hline
\hbox{Name:} & \hbox{type of operad:} &
\hbox{introduced in:} & \hbox{typical element:}
\\
\hline
\rule{0pt}{1.5em}%
$\Lat^{(2)}$& \hbox{$\bbN$-colored $\Sigma$-operad} &
 \hbox  {Section~\ref{sec:tamark-tsyg-oper}} &  \hbox
{tree $\delta,\gamma,\ldots$}
\\
\hbox {$\opTm^\bbN_2$} & \hbox{$\bbN$-colored $2$-operad} & 
\hbox {Definition~\ref{musim_objet_Ryn}} &\hbox
{$\OmegaN$-labelled tree $(\frT,\delta)$}
\\
\hbox {$\opTam^\bbN_2$} & \hbox{pruned $\bbN$-colored $2$-operad} & 
\hbox {pullback~(\ref{Tam2})} &\hbox
{$\OrN$-labelled tree $(\frO,\delta)$}
\\
\hline \rule{0pt}{1em}
\end{tabular}

\caption{Notation.\label{za_tyden_z_Bonnu_do_Prahy}}
\end{table}
We also briefly recall, following~\cite{batanin-markl:Adv} closely, 
duoidal categories and the necessary related notions.

A {\em duoidal $\ttV$-category} is a pseudomonoid in the
2-category of monoidal $\ttV$-categories, lax-monoidal $\ttV$-functors and
their monoidal $\ttV$-transformations. 
Explicitly, a  duoidal $\ttV$-category is a quintuple $\duo = 
(\duo,\boxx_0,\boxx_1,e,v)$ such that
\begin{itemize}
\item [(i)]
$(\duo,\boxx_0,e)$ and $(\duo,\boxx_1,v)$ are monoidal $\ttV$-categories, equipped with 
\item[(ii)]
 a $\ttV$-natural interchange transformation
\begin{equation}
\label{dl}
(X \boxx_1 Y) \boxx_0 (Z \boxx_1 W) \to 
(X \boxx_0 Z) \boxx_1 (Y \boxx_0 W),
\end{equation}
\item[(iii)]
a map 
$e \to e\boxx_1 e, $
\item[(iv)] 
a map
$v\boxx_0 v \to v,$ and 
\item[(v)] 
a map 
$e\to v.$
\end{itemize}
The above data should enjoy the coherence properties listed e.g.\
in~\cite[p.~1816]{batanin-markl:Adv}. Moreover, we require that $v$ is
a monoid in $(\duo,\boxx_0,e)$ and $e$ a comonoid in
$(\duo,\boxx_1,v)$. 

A duoidal category $\duo$ is called 
{\em strict\/} if both monoidal categories
$(\duo,\boxx_0,e)$ and $(\duo,\boxx_1,v)$ are strict monoidal categories.
Since every duoidal category is equivalent to a strict one
by~\cite[Theorem~2.16]{batanin-markl:Adv}, we assume that all
duoidal categories in this article are strict.

A {\em $1$-operad\/} in a duoidal category $\duo =
(\duo,\boxx_0,\boxx_1,e,v)$ is a
collection $\scA = \{\scA(n)\}_{n \geq 0}$ of objects of $\duo$ such
that
\begin{itemize}
\item[(i)]
for each integers $n \geq 1$, $\Rada k1n
\geq 0$, one is given a structure morphism
\begin{equation}
\label{morph}
\gamma : \big( \scA(k_1) \boxx_1 \cdots \boxx_1 \scA(k_n)\big) \boxx_0 \scA(n) \to
\scA(k_1 + \cdots + k_n),
\end{equation}
\item[(ii)]
one is given a map $j : e \to \scA(1)$ (the unit) and
\item[(iii)]
a left $v$-module structure $v \boxx_0 \scA(0) \to \scA(0)$
\end{itemize}
such that appropriate axioms are satisfied, 
see~\cite[p.~1825]{batanin-markl:Adv} for
details. An example is the operad $\uAss$ with $\uAss(n) = v$ for each
$n \in \bbN$~\cite[Example~4.4]{batanin-markl:Adv}. 
Recall finally that a~{\em
  multiplicative $1$-operad\/} is a $1$-operad $\scA$ equipped with an
operad morphism 
\begin{equation}
\label{eq:19}
\alpha : \uAss \to \scA.
\end{equation}

For a duoidal category $\duo$ we denote by $\ttCat(\duo)$ the
$2$-category of $(\duo,\boxx_0,e)$-enriched categories. As observed by
Forcey \cite{Forcey}, $\ttCat(\duo)$ has a monoidal structure. The
tensor product $\timesred_1$ of two $\duo$-categories $\calK$ and
$\calL$ is given by the cartesian product on the objects level while
\[
(\calK\timesred_1\calL)\big((X,Y),(Z,W)\big) := \calK(X,Z)\boxx_1 \calL(Y,W),
\mbox { for $X,Z \in \calK$, $Y,W \in \calL$.}
\]
 The unit for this tensor product is the category $\mathbf 1_v$ which has one object $*$ and 
 ${\mathbf 1}_v(*,*) = v.$ 

A {\em monoidal $\duo$-category\/} $\calK = (\calK,\odot,\eta)$ is defined as 
a pseudomonoid in the monoidal $2$-category
$(\ttCat(\duo),\timesred_1,{\mathbf 1_v})$, see
\cite[pp.~1820--21]{batanin-markl:Adv} for a detailed description of
this structure. 
Each object $X \in \calK$ has its 
{\em endomorphism $1$-operad\/} $\End_X$ in $\duo$ with components
\[
\End_X(n) := \calK(\mbox{\large $\odot$}^n X,X),\ n \geq 1,
\]
see~\cite[Def.~4.7]{batanin-markl:Adv}.

A {\em monoid in $\calK$\/} is a lax monoidal functor
$\mathbf 1_v \to \calK$.
More explicitly, a monoid in $\calK$ is an object
$\ttM\in\calK$ together with:
\begin{itemize}
\item[(i)] 
a morphism (neutral element)
$i:\eta \to  \ttM,$
\item[(ii)] 
a morphism (multiplication)
$m: \ttM\odot \ttM \to \ttM$
and
\item[(iii)] 
a morphism (the {\em unit\/})
 $u:v\to \calK(\ttM,\ttM)$ in $\duo$.
\end{itemize}
These data should satisfy axioms listed in \cite[p.~1823]{batanin-markl:Adv}.
The endomorphism $1$-operad $\End_{\ttM}$ of a monoid $\ttM \in \calK$ is
multiplicative by \cite[Prop.~4.9]{batanin-markl:Adv}.

The {\em center\/} of a monoid $\ttM$ is the following equalizer in
$\duo$:
\begin{equation}
\label{equalizer}
Z(\ttM)\to {\calK}(\eta,\ttM)
\eqv
{\calK}(\ttM, \ttM).
\end{equation}

The center has a natural structure of a {\em duoid\/}  ({\em double monoid\/}
in the terminology of \cite{AM}) in $\duo$. This
is, by definition, an object $\scD \in \duo$  together with 
\begin{itemize} 
\item[(i)] a structure of a monoid 
$\scD\boxx_0\scD\to \scD \ , \  e\to \scD$
with respect to the first monoidal structure of $\duo$, and
\item[(ii)] a structure of a monoid
$\scD\boxx_1\scD \to \scD \ , \ v\to \scD$
with respect to the second monoidal structure of $\duo$
\end{itemize}
such that suitable axioms listed in
\cite[p.~1818]{batanin-markl:Adv} are satisfied.

\section{Operadic categories of trees and  ordinals.}
\label{sec:oper-categ-levell}

In this section we introduce several operadic categories of colored
trees and ordinals, and study various functors between them. We also
define endomorphism operads of collections in duoidal categories and
prove, in Theorem~\ref{etiopska_restaurace}, the existence of
canonical actions on multiplicative $1$-operads.

\subsection{The category of coloured $2$-trees} 
\label{OmegaN}
We will need the $\bbN$-colored version $\OmegakN$ of the
category $\ttOmega$ of $k$-trees recalled in Section~\ref{Omegan}.  
It is constructed by taking $\ttO = \ttOmega_k$ and $\frC =
\bbN$ in the pullback~(\ref{eq:Jaruska}). Explicitly, 
an {\em $\bbN$-colored $k$-tree\/} is a couple
$\bfT = (\EuT,c)$ consisting of a $k$-tree $\EuT \in \ttOmega_k$ 
and of a~coloring $c : \hbox{$\Leaf_k(\EuT)\! + \!1 \to
\bbN$}$ of its set of $k$-leaves plus one more color $c(1)$ interpreted
as the `output' color of~$\EuT$.  Morphisms $(\EuT,c') \to (\EuS,c'')$ are
morphisms $\EuT \to \EuS$ of the underlying $k$-trees if $c'(1)
= c''(1)$, while there are no morphisms if $c'(1) \not= c''(1)$. This
explicit description follows the ideas of
Tamarkin's~\cite{tamarkin:CM07}.

The category $\OmegakN$ with $|\bfT| := \Leaf_k(\EuT)$ is 
operadic. The fiber $\bfT_i$ of a map $f: \bfT = (\EuT,c')
\to \bfS = (\EuS,c'')$ over a $k$-leaf $i \in |\EuS|$ is the fiber as
in~(\ref{Eva_reditelem?}), with the coloring of its
$k$-leaves induced by $c'$ and the output color $c''(i)$.  One
has $\pi_0(\OmegakN) \cong \bbN$, the chosen terminal object for $n
\in \bbN$ being the 
terminal $k$-tree $\bfU_k^n$ with its unique $k$-leaf
colored by $n$ and the output color~$n$.

One analogously  defines
an $\bbN$-colored version $\OrdkN$ of the operadic category $\Ordk$ of 
$k$-ordinals of
Example~\ref{nordinal}. We leave its detailed description to the
reader. One has $\pi_0(\OrdkN)
\cong \bbN$, the chosen terminal object for $n \in \bbN$ being
the terminal $k$-ordinal $\bfjedna^n_k$ colored by $n$,  with the output
color~$n$.

In the rest of this article we will however need the categories
$\ttOmega_k$ and $\Ordk$, resp.~their colored versions  $\ttOmega_k^\bbN$
and $\OrdkN$, only for $k \leq 2$.

\begin{definition}
We will call $\ttOmega_2$-operads, 
resp.~$\ttOmega_2^\bbN$-operads, \hbox{\em $2$-operads\/}, 
resp.~{\em $\bbN$-colored  $2$-operads\/}. Similarly, $\Ord_2$-operads,
resp.~$\OrN$-operads, will be called {\em pruned $2$-operads\/},
resp. {\em pruned $\bbN$-colored $2$-operads\/}. 
\end{definition}

\begin{proposition} \label{omegaendos} There is a natural
  $\ttOmega_2$-multicotensor $\Boxx$ on any duoidal $V$-category $\duo
  = (\duo, \boxx_0,\boxx_1,e,v)$.
  If $\duo$ is a cocomplete $V$-category then the
  multicotensor $\Boxx$ is strong.

In particular,
each $\bbN$-colored collection $\scE = \coll \scE$ of
objects in $\duo$ admits its
$\bbN$-colored endomorphism $2$-operad $\End_\scE^{\OmegaN} := 
\End_\scE^{\boxx^\bbN} \in  \OpV
\OmegaN$. 
\end{proposition}

\begin{proof}
The proof is based on a simple
 modification of a construction given
in~\cite[p.~1853]{batanin-markl:Adv}.
Let $\EuT\in \ttOmega_2$ be a   $2$-tree. Let the $1$-truncation of  $\EuT$ be
$\{1,\ldots,t\}$, with its set of $2$-leaves over the $1$-vertex $d
\in \{1,\ldots,t\}$ being $\{v^d_1,\ldots,v^d_{q_d}\}$,
and let ${\mathcal X} = \{X_{c,i}\}_{1\leq d \leq t,\ 1\leq i\leq q_d} $ be a family of objects in $\duo.$

We then define, for $1 \leq d \leq t$,
\begin{equation}
  \label{Jarus1}
  {\mathcal X}^{\boxx_1^{q_d}}_\EuT := 
  \cases{
    X_{d,1} \boxx_1 \cdots  
    \boxx_1 X_{d,q_d}}{if $q_d > 1$}v{if
    $q_d =0$.} 
\end{equation}

With this notation,
\[
\Boxx_{\EuT}(X_{1,1},\ldots,X_{t,q_t}) := {\mathcal X}^{\boxx_1^{q_t}}_\EuT \boxx_0
{\mathcal X}^{\boxx_1^{q_{t-1}}}_\EuT  \boxx_0       \cdots\boxx_0
{\mathcal X}^{\boxx_1^{q_2}}_\EuT \boxx_0
{\mathcal X}^{\boxx_1^{q_1}}_\EuT.\footnote{Observe the reversed order of the
  factors.}
\]

The counit of the multicotensor $\Boxx$ is the identity.
The comultiplication 
\begin{equation}
\label{multD}
\mu_{\sigma}:
\xymatrix{
\Boxx_\EuT \ar[r]& \Boxx_{\EuS}(\Boxx_{\EuT_1},\ldots,\Boxx_{\EuT_k})
}
\end{equation}
corresponding to a morphism $\sigma: \EuT \to \EuS$ in $\ttOmega_2$
with fibers $\Rada \EuT1k$ can be described by induction. We do not
provide the details here but refer to the construction of
the morphism $X^{\sigma}$ in the proof of ~\cite[Lemma~11.13]{batanin-markl:Adv} which
is exactly the comultiplication $\mu_{\sigma}$ in the case when all
members of the family $\mathcal X$ are equal to the same object $X.$
The argument however does not depend on this difference. We will
repeat the same kind of construction in the description of the endomorphism
$2$-operad $\End_\scE^{\OmegaN}$ in the second part of this proof.

The proof that $\Boxx$ is a strong multitensor also goes by an induction following the construction of 
the comultiplication, using the fact that 
both $\boxx_0,\boxx_1$ are $V$-functors and the interchange
morphism~(\ref{dl}) 
is a $V$-natural transformation.  We leave the details to the reader.

As shown in~(\ref{Zitra_s_chlapci}), the $\ttOmega_2$-multicotensor
$\Boxx$ induces an $\ttOmega^\bbN_2$-multicotensor $\Boxx^\bbN$. The
$\bbN$-colored  $2$-operad 
$\End_\scE^{\OmegaN} \in \OpV \OmegaN$ of an $\bbN$-collection $\scE =
\coll \scE$ in $\duo$ is the endomorphism
operad related to this multicotensor as in Lemma~\ref{endo}.  We
describe it in detail because we will need this
description later.

We start by associating, to each $\bbN$-colored $2$-tree
$\frT = (\EuT,c) \in \OmegaN$, the {\em $\boxx$-power\/} $\scE^{\frT}$ of
$\scE$ as
follows.  Let the $1$-truncation of the underlying $2$-tree $\EuT$ be
$\{1,\ldots,t\}$, with its set of $2$-leaves over the $1$-vertex $d
\in \{1,\ldots,t\}$ being $\{v^d_1,\ldots,v^d_{q_d}\}$.

We then define, for $1 \leq d \leq t$,
\begin{equation}
  \label{Jarus}
  \scE^{\boxx_1^{q_d}}_\frT := 
  \cases{
    \scE\big(c(v^d_1)\big) \boxx_1 \cdots  
    \boxx_1 \scE\big(c(v^d_{q_d})\big)}{if $q_d > 1$}v{if
    $q_d =0$.} 
\end{equation}
With this notation,
\[
\scE^{\frT} := \scE^{\boxx_1^{q_t}}_\frT \boxx_0
\scE^{\boxx_1^{q_{t-1}}}_\frT  \boxx_0       \cdots\boxx_0
\scE^{\boxx_1^{q_2}}_\frT \boxx_0
\scE^{\boxx_1^{q_1}}_\frT.
\]
We believe that the portrait of $\scE^{\frT}$ in
Figure~\ref{symbol} borrowed from~\cite{batanin-markl:Adv} 
clarifies our definition.
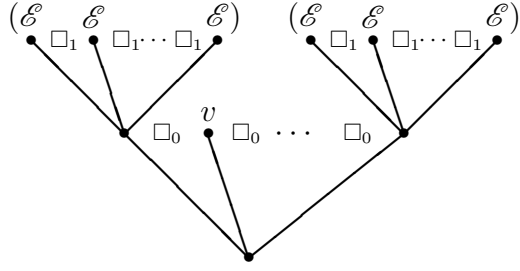
\begin{figure}
\begin{center}
{
\unitlength=.1em
\begin{picture}(150.00,80.00)(0.00,0.00)
\thicklines
\put(85.00,40.00){\makebox(0.00,0.00){$\cdots$}}
\put(130.00,70.00){\makebox(0.00,0.00){\scriptsize $\cdots$}}
\put(40.00,70.00){\makebox(0.00,0.00){\scriptsize $\cdots$}}
\put(110.00,40.00){\makebox(0.00,0.00)[r]{\scriptsize $\boxx_0$}}
\put(63.00,40.00){\makebox(0.00,0.00)[l]{\scriptsize $\boxx_0$}}
\put(43.00,40.00){\makebox(0.00,0.00){\scriptsize $\boxx_0$}}
\put(140.00,70.00){\makebox(0.00,0.00){\scriptsize $\boxx_1$}}
\put(120.00,70.00){\makebox(0.00,0.00){\scriptsize $\boxx_1$}}
\put(100.00,70.00){\makebox(0.00,0.00){\scriptsize $\boxx_1$}}
\put(50.00,70.00){\makebox(0.00,0.00){\scriptsize $\boxx_1$}}
\put(30.00,70.00){\makebox(0.00,0.00){\scriptsize $\boxx_1$}}
\put(10.00,70.00){\makebox(0.00,0.00){\scriptsize $\boxx_1$}}
\put(150.00,72.00){\makebox(0.00,0.00)[b]{$\hphantom{(}\scE)$}}
\put(110.00,74.00){\makebox(0.00,0.00)[b]{$\scE$}}
\put(90.00,72.00){\makebox(0.00,0.00)[b]{$(\scE \hphantom{)}$}}
\put(60.00,72.00){\makebox(0.00,0.00)[b]{$\hphantom{(}\scE)$}}
\put(20.00,74.00){\makebox(0.00,0.00)[b]{$\scE$}}
\put(0.00,72.00){\makebox(0.00,0.00)[b]{$(\scE \hphantom{)}$}}
\put(70.00,0.00){\makebox(0.00,0.00){\scriptsize $\bullet$}}
\put(30.00,40.00){\makebox(0.00,0.00){\scriptsize $\bullet$}}
\put(120.00,40.00){\makebox(0.00,0.00){\scriptsize $\bullet$}}
\put(150.00,70.00){\makebox(0.00,0.00){\scriptsize $\bullet$}}
\put(110.00,70.00){\makebox(0.00,0.00){\scriptsize $\bullet$}}
\put(90.00,70.00){\makebox(0.00,0.00){\scriptsize $\bullet$}}
\put(60.00,70.00){\makebox(0.00,0.00){\scriptsize $\bullet$}}
\put(20.00,70.00){\makebox(0.00,0.00){\scriptsize $\bullet$}}
\put(0.00,70.00){\makebox(0.00,0.00){\scriptsize $\bullet$}}
\put(57.0,40.00){\makebox(0.00,0.00){\scriptsize $\bullet$}}
\put(57.0,44.00){\makebox(0.00,0.00)[b]{$v$}}
\put(70.00,0.00){\line(5,4){50.00}}
\put(30.00,40.00){\line(1,-1){40.00}}
\put(120.00,40.00){\line(1,1){30.00}}
\put(120.00,40.00){\line(-1,3){10.00}}
\put(90.00,70.00){\line(1,-1){30.00}}
\put(30.00,40.00){\line(1,1){30.00}}
\put(20.00,70.00){\line(1,-3){10.00}}
\put(30.00,40.00){\line(-1,1){30.00}}
\put(70.00,0.00){\line(-1,3){13.00}}
\end{picture}}
\caption{\label{symbol}An `ideological' picture of $\scE^{\frT}$. Leaves
  of height $2$ (resp.~$1$) are decorated by $\scE = \coll \scE$ (resp.~$v$).  The
  decorations of vertices of height $2$ (resp.~$1$) are then
  multiplied by $\boxx_1$ (resp.~$\boxx_0$), with the
  $\boxx_1$-multiplication performed first. For brevity we did not
  show the colors.}
\end{center}
\end{figure}
Define finally
\[
\End^{\OmegaN}_\scE(\frT) := \duo\big(\scE^\frT,\scE(n)\big),
\]
where $n$ is the output color of $\frT$. Let us describe
the operad multiplication
\begin{equation}
\label{mult}
\mu(f):
\xymatrix{
\End^{\OmegaN}_\scE(\frT_1) \ot \cdots \ot \End^{\OmegaN}_\scE(\frT_k) 
\ot \End^{\OmegaN}_\scE(\frS) \ar[r]& \End^{\OmegaN}_\scE(\frT)
}
\end{equation}
corresponding to a morphism $f : \frT \to \frS$ in $\OmegaN$ 
with fibers $\Rada \frT1k$. For 
\[
\phi : \scE^\frS \to \scE(n) \in \End^{\OmegaN}_\scE(\frS) 
\ \mbox { and } \ \phi_i : \scE^{\frT_i} \to \scE(n_i) \in
\End^{\OmegaN}_\scE(\frT_i),\ i \in |\frS|,
\]
we  use an auxiliary natural morphism
\begin{equation}
\label{eq:13}
\Phi^f(\Rada \phi1k): \scE^{\frT} \lra  \scE^{\frS}
\end{equation}
described below, and define $\mu(f)$ by the formula
\[
\mu(f)(\phi_1 \ot \cdots \phi_k \ot \phi) := \phi \circ 
\Phi^f(\Rada \phi1k).
\]

Let us define~(\ref{eq:13}). Suppose first that 
$\tr_1(\frS) = (1)$, so $\frS$ is the suspension of the $1$-tree
$(1,\ldots,k)$. In this case the exchange rule~(\ref{dl}) in $\duo$ induces  
a natural map
\[
\scE^f : \scE^{\frT} \longrightarrow \scE^{\frT_1} \boxx_1 \cdots \boxx_1 \scE^{\frT_k}
\]
and $\Phi^f(\Rada \phi1k)$ is the composite
\[
\xymatrix@1@C = +1.6em{
\scE^{\frT} \ar[r]^{\scE^f \hskip 3em}& \scE^{\frT_1} \boxx_1 \cdots \boxx_1
\scE^{\frT_k} \ar[rrr]^{\phi_1\boxx_1 \cdots \boxx_1 \phi_k \hskip 2em}&&&
 \scE(n_1)\boxx_1 \cdots \boxx_1 \scE(n_k) = \scE^{\frS}. 
}
\]

To address the general case denote, for $\frT_1,\frT_2 \in \OmegaN$, by $\frT_1
\vee \frT_2$ the colored $2$-tree 
obtained by identifying the root of $\frT_1$ with the root of
$\frT_2$.\footnote{In~\cite{batanin-markl:Adv} we denoted this
  operation by  $\frT_1+\frT_2$.}  Observe that
\begin{equation}
  \label{eq:16}
  \scE^{\frT_1 \vee \frT_2} = \scE^{\frT_1} \boxx_0 \scE^{\frT_2}.
\end{equation}
A general $\frS$ uniquely decomposes into the product 
$\frS_1 \vee \cdots \vee
\frS_p$ of the suspensions of $1$-trees; $f :\frT \to \frS$ is then of
the form 
\begin{equation}
\label{Vlasta}
f = f_1 \vee \cdots \vee f_p : \frT = \frT_1 \vee \cdots \vee \frT_p \lra
\frS_1 \vee \cdots \vee \frS_p = \frS.
\end{equation}
Suppose the fibers of $f_a : \frT_a \to \frS_a$ are
$\frT^a_1,\ldots,\frT^a_{p_a}$, $1 \leq a \leq s$, 
then the fibers of $f$ are
\[
\frT_1,\ldots,\frT_k = 
\frT^1_1,\ldots,\frT^1_{p_1},\ldots,\frT^s_1,\ldots,\frT^s_{p_s}.
\]
For $\phi^a_b : \scE^{\frT^a_b} \to \scE(n^a_b)$ we define
$\Phi^f(\phi^1_1,\ldots,\phi^1_{p_1},\ldots,\phi^s_1,\ldots,\phi^s_{p_s}): 
\scE^{\frT} \to \scE^{\frS}$ by
\begin{equation}
\label{eq:18}
\Phi^f(\phi^1_1,\ldots,\phi^1_{p_1},\ldots,\phi^s_1,\ldots,\phi^s_{p_s}) :=
\Phi^{f_1}(\phi^1_1,\ldots,\phi^1_{p_1}) \boxx_0 \cdots \boxx_0
\Phi^{f_s}(\phi^s_1,\ldots,\phi^s_{p_s}),
\end{equation} 
where we tacitly used~(\ref{eq:16}).
This defines~(\ref{eq:13}) for arbitrary trees.
\end{proof}

\subsection{The category of levelled trees}
\label{LTr}
The category $\LTr$ has objects planar rooted trees with three
types of \hbox{vertices}: `white' vertices $\bily$, `vertical'
vertices $\tlusty$ and
`horizontal' vertices~$\cerny$.\footnote{The terminology will be
explained in Section~\ref{LTrendo}.}   These vertices may have arbitrary arities
$\geq 0$ and are lined up into levels of two types:
\begin{itemize}
\item[(i)]
levels consisting of white vertices \bily\ and/or vertical vertices \tlusty, and
\item[(ii)]
levels consisting only of horizontal vertices \cerny.
\end{itemize}
An example of a tree in $\LTr$ is given in Figure~\ref{fig:1} which
uses the convention that levelled trees are drawn horizontally, with the
root on the left.
\begin{figure}
\begin{center}
\psscalebox{.5 .5}
{
\begin{pspicture}(2,-5.75)(21.006176,4.9)
\psline[linecolor=black, linewidth=0.02](2.7289307,1.7506206)(2.7289307,1.7506206)
\psline[linecolor=black, linewidth=0.02](2.5974844,4.250776)(2.5974844,-5.75)(2.5974844,-5.6249905)
\pscustom[linecolor=black, linewidth=0.05]
{
\newpath
\moveto(8.9,1.15)
}
\pscustom[linecolor=black, linewidth=0.05]
{
\newpath
\moveto(21.5,2.45)
}

\pscustom[linecolor=black, linewidth=0.05]
{
\newpath
\moveto(2.597484,-1.2496119)
\lineto(2.912956,-1.549635)
\curveto(3.0706918,-1.6996472)(3.51761,-1.9996704)(3.8067925,-2.149682)
\curveto(4.095975,-2.2996936)(4.858365,-2.599717)(5.331572,-2.7497284)
\curveto(5.8047805,-2.89974)(7.1718235,-3.1997633)(8.06566,-3.3497753)
\curveto(8.959497,-3.499787)(10.668302,-3.6998022)(13.113208,-3.8498137)
}

\pscustom[linecolor=black, linewidth=0.05]
{
\newpath
\moveto(2.597484,-1.2496119)
\lineto(2.912956,-0.9495886)
\curveto(3.0706918,-0.799577)(3.51761,-0.49955383)(3.8067925,-0.34954223)
\curveto(4.095975,-0.19953065)(4.858365,0.10049255)(5.331572,0.25050414)
\curveto(5.8047805,0.40051574)(7.1718235,0.70053893)(8.06566,0.85055053)
\curveto(8.959497,1.0005622)(10.668302,1.200578)(13.113208,1.3505896)
}
\pscustom[linecolor=black, linewidth=0.05]
{
\newpath
\moveto(13.113208,1.3505896)
\lineto(13.22968,1.4976599)
\curveto(13.287915,1.5711951)(13.429346,1.7182654)(13.512539,1.7918005)
\curveto(13.595733,1.8653357)(13.895232,2.0124056)(14.111536,2.0859408)
\curveto(14.327839,2.159476)(14.843642,2.262425)(15.742138,2.3506672)
}
\pscustom[linecolor=black, linewidth=0.05]
{
\newpath
\moveto(13.113208,-3.849814)
\lineto(13.22968,-3.996884)
\curveto(13.287915,-4.0704193)(13.429346,-4.2174892)(13.512539,-4.291024)
\curveto(13.595733,-4.364559)(13.895232,-4.5116296)(14.111536,-4.585165)
\curveto(14.327839,-4.6587)(14.843642,-4.761649)(15.742138,-4.849891)
}
\pscustom[linecolor=black, linewidth=0.05]
{
\newpath
\moveto(13.113208,-3.849814)
\lineto(13.22968,-3.6880364)
\curveto(13.287915,-3.607148)(13.429346,-3.4453712)(13.512539,-3.3644824)
\curveto(13.595733,-3.2835937)(13.895232,-3.1218164)(14.111536,-3.0409276)
\curveto(14.327839,-2.9600391)(14.843642,-2.846795)(15.742138,-2.7497284)
}
\pscustom[linecolor=black, linewidth=0.05]
{
\newpath
\moveto(13.113208,1.3505896)
\lineto(13.22968,1.1888123)
\curveto(13.287915,1.1079236)(13.429346,0.94614655)(13.512539,0.86525786)
\curveto(13.595733,0.78436923)(13.895232,0.62259215)(14.111536,0.54170346)
\curveto(14.327839,0.46081483)(14.843642,0.3475708)(15.742138,0.25050429)
}
\psline[linecolor=black, linewidth=0.05](13.113208,-3.8498137)(13.113208,-3.8498137)
\psline[linecolor=black, linewidth=0.05](13.113208,-3.8498137)(18.371069,-3.8498137)(18.371069,-3.8498137)
\psline[linecolor=black, linewidth=0.05](15.742138,0.25050429)(18.371069,0.05048877)(18.502516,0.05048877)
\psline[linecolor=black, linewidth=0.05](15.742138,-2.7497284)(18.371069,-2.549713)(18.371069,-2.549713)
\psline[linecolor=black, linewidth=0.05](18.371069,-2.549713)(21.0,-2.3496974)(21.0,-2.3496974)
\pscustom[linecolor=black, linewidth=0.05]
{
\newpath
\moveto(18.371069,0.050488893)
\lineto(18.604012,0.3232373)
\curveto(18.720484,0.4596118)(19.252924,0.73235995)(19.668896,0.8687341)
\curveto(20.084866,1.0051084)(20.625626,1.1687573)(21.0,1.2505819)
}
\pscustom[linecolor=black, linewidth=0.05]
{
\newpath
\moveto(15.742139,2.350667)
\lineto(15.975081,2.5779574)
\curveto(16.091553,2.6916027)(16.623995,2.9188929)(17.039967,3.0325382)
\curveto(17.455935,3.1461835)(17.996696,3.2825577)(18.371069,3.3507447)
}
\pscustom[linecolor=black, linewidth=0.05]
{
\newpath
\moveto(18.371069,-3.849814)
\lineto(18.604012,-3.6907105)
\curveto(18.720484,-3.6111584)(19.252924,-3.4520557)(19.668896,-3.3725038)
\curveto(20.084866,-3.2929523)(20.625626,-3.1974902)(21.0,-3.1497595)
}
\pscustom[linecolor=black, linewidth=0.05]
{
\newpath
\moveto(18.502516,0.050488893)
\lineto(18.72381,0.1641339)
\curveto(18.83446,0.22095641)(19.340279,0.33460206)(19.73545,0.39142457)
\curveto(20.130623,0.44824708)(20.644346,0.5164343)(21.0,0.5505276)
}
\pscustom[linecolor=black, linewidth=0.05]
{
\newpath
\moveto(18.371069,2.450675)
\lineto(18.604012,2.33703)
\curveto(18.720484,2.2802076)(19.252924,2.1665618)(19.668896,2.1097393)
\curveto(20.084866,2.0529168)(20.625626,1.9847299)(21.0,1.9506361)
}
\pscustom[linecolor=black, linewidth=0.05]
{
\newpath
\moveto(18.371069,2.450675)
\lineto(18.604012,2.6325073)
\curveto(18.720484,2.7234235)(19.252924,2.9052558)(19.668896,2.996172)
\curveto(20.084866,3.087088)(20.625626,3.1961875)(21.0,3.250737)
}
\pscustom[linecolor=black, linewidth=0.05]
{
\newpath
\moveto(15.742139,2.350667)
\lineto(15.975081,2.168835)
\curveto(16.091553,2.0779188)(16.623995,1.8960865)(17.039967,1.8051703)
\curveto(17.455935,1.7142541)(17.996696,1.6051548)(18.371069,1.5506052)
}
\pscustom[linecolor=black, linewidth=0.05]
{
\newpath
\moveto(5.2264147,-2.6497204)
\lineto(5.4593577,-2.467888)
\curveto(5.57583,-2.376972)(6.1082716,-2.19514)(6.524242,-2.1042237)
\curveto(6.940212,-2.0133076)(7.480973,-1.9042084)(7.8553457,-1.8496586)
}
\pscustom[linecolor=black, linewidth=0.05]
{
\newpath
\moveto(18.371069,-3.849814)
\lineto(18.604012,-4.1452913)
\curveto(18.720484,-4.29303)(19.252924,-4.588507)(19.668896,-4.736246)
\curveto(20.084866,-4.8839846)(20.625626,-5.061272)(21.0,-5.1499147)
}
\pscustom[linecolor=black, linewidth=0.05]
{
\newpath
\moveto(18.371069,0.050488893)
\lineto(18.604012,-0.26771727)
\curveto(18.720484,-0.42682067)(19.252924,-0.7450275)(19.668896,-0.9041309)
\curveto(20.084866,-1.0632336)(20.625626,-1.2541577)(21.0,-1.3496199)
}
\pscustom[linecolor=black, linewidth=0.05]
{
\newpath
\moveto(18.502516,0.050488893)
\lineto(18.72381,-0.06315674)
\curveto(18.83446,-0.11997925)(19.340279,-0.23362426)(19.73545,-0.2904468)
\curveto(20.130623,-0.3472693)(20.644346,-0.41545653)(21.0,-0.44955003)
}
\psline[linecolor=black, linewidth=0.05](18.371069,-3.8498137)(21.0,-4.149837)(21.0,-4.149837)
\psline[linecolor=black, linewidth=0.05](15.742138,2.3506672)(18.371069,2.450675)(18.371069,2.450675)
\psline[linecolor=black, linewidth=0.02](5.226415,4.250776)(5.226415,-5.75)(5.226415,-5.6249905)
\psline[linecolor=black, linewidth=0.02](7.8553457,4.250776)(7.8553457,-5.75)(7.8553457,-5.6249905)
\psline[linecolor=black, linewidth=0.02](10.484277,4.250776)(10.484277,-5.75)(10.484277,-5.6249905)
\psline[linecolor=black, linewidth=0.02](13.113208,4.250776)(13.113208,-5.75)(13.113208,-5.6249905)
\psline[linecolor=black, linewidth=0.02](15.842138,4.250776)(15.842138,-5.75)(15.842138,-5.6249905)
\psline[linecolor=black, linewidth=0.02](18.371069,4.250776)(18.371069,-5.75)(18.371069,-5.6249905)
\psdots[linecolor=black, dotsize=0.49939746](5.2,-2.65)
\psdots[linecolor=black, dotsize=0.49939746](5.2,0.15)
\psdots[linecolor=black, dotsize=0.49939746](7.8,0.85)
\psdots[linecolor=black, dotsize=0.49939746](7.9,-1.85)
\psdots[linecolor=black, dotsize=0.49939746](7.9,-3.25)
\psdots[linecolor=black, fillstyle=solid, dotstyle=o, dotsize=0.49939746](10.5,1.15)
\psdots[linecolor=black, fillstyle=solid, dotstyle=o, dotsize=0.49939746](18.4,0.05)
\psdots[linecolor=black, fillstyle=solid, dotstyle=o, dotsize=0.49939746](18.4,3.35)
\psdots[linecolor=black, fillstyle=solid, dotstyle=o, dotsize=0.49939746](18.4,2.45)
\psdots[linecolor=black, fillstyle=solid, dotstyle=o, dotsize=0.49939746](18.4,1.55)
\psdots[linecolor=black, fillstyle=solid, dotstyle=o, dotsize=0.49939746](18.4,-3.85)
\psdots[linecolor=black, dotstyle=otimes, dotsize=0.49939746](18.4,1.55)
\psdots[linecolor=black, fillstyle=solid, dotstyle=o, dotsize=0.49939746](13.1,1.35)
\psdots[linecolor=black, fillstyle=solid, dotstyle=o, dotsize=0.49939746](15.8,0.25)
\psdots[linecolor=black, fillstyle=solid, dotstyle=o, dotsize=0.49939746](15.8,-4.85)
\psdots[linecolor=black, fillstyle=solid, dotstyle=o, dotsize=0.49939746](15.8,2.35)
\psdots[linecolor=black, fillstyle=solid, dotstyle=o, dotsize=0.49939746](18.4,-2.55)
\psdots[linecolor=black, fillstyle=solid, dotstyle=o, dotsize=0.49939746](15.8,-2.75)
\psdots[linecolor=black, dotstyle=otimes, dotsize=0.49939746](15.8,0.25)
\psdots[linecolor=black, dotstyle=otimes, dotsize=0.49939746](15.8,-2.75)
\psdots[linecolor=black, dotstyle=otimes, dotsize=0.49939746](15.8,-4.85)
\psdots[linecolor=black, dotstyle=otimes, dotsize=0.49939746](18.4,-2.55)
\psdots[linecolor=black, dotstyle=otimes, dotsize=0.49939746](15.8,2.35)
\psdots[linecolor=black, dotstyle=otimes, dotsize=0.49939746](15.8,-3.85)
\psdots[linecolor=black, fillstyle=solid, dotstyle=o, dotsize=0.49939746](13.1,-3.85)
\psline[linecolor=black, linewidth=0.05, arrowsize=0.05291666666666668cm 6.0,arrowlength=2.0,arrowinset=0.0]{<-}(0.1,-1.2496121)(2.5974844,-1.2496121)
\psdots[linecolor=black, fillstyle=solid, dotstyle=o, dotsize=0.49939746](10.5,-3.65)
\psdots[linecolor=black, dotstyle=otimes, dotsize=0.49939746](10.5,-3.65)
\psdots[linecolor=black, fillstyle=solid, dotstyle=o, dotsize=0.49939746](13.1,-2.25)
\psdots[linecolor=black, dotstyle=otimes, dotsize=0.3](2.6,-1.25)
\psdots[linecolor=black, fillstyle=solid, dotstyle=o, dotsize=0.49939746](2.6,-1.25)
\psdots[linecolor=black, dotstyle=otimes, dotsize=0.49939746](2.6,-1.25)
\psdots[linecolor=black, dotsize=0.49939746](18.4,3.35)
\psdots[linecolor=black, dotsize=0.49939746](18.4,2.45)
\psdots[linecolor=black, dotsize=0.49939746](18.4,1.55)
\psdots[linecolor=black, dotsize=0.49939746](18.4,0.05)
\psdots[linecolor=black, dotsize=0.49939746](18.4,-2.55)
\psdots[linecolor=black, dotsize=0.49939746](18.4,-3.85)
\pscustom[linecolor=black, linewidth=0.05]
{
\newpath
\moveto(10.542139,-3.449333)
\lineto(10.7485,-3.199315)
\curveto(10.85168,-3.0743067)(11.323361,-2.8242884)(11.691864,-2.6992798)
\curveto(12.060367,-2.5742712)(12.539417,-2.4242601)(12.871069,-2.3492553)
}
\pscustom[linecolor=black, linewidth=0.05]
{
\newpath
\moveto(5.142138,0.25066712)
\lineto(5.375082,0.06883484)
\curveto(5.4915533,-0.022081299)(6.023995,-0.20391357)(6.439965,-0.29482973)
\curveto(6.855935,-0.38574585)(7.3966956,-0.49484497)(7.771069,-0.54939485)
}
\psdots[linecolor=black, dotsize=0.49939746](7.8,-0.55)
\pscustom[linecolor=black, linewidth=0.05]
{
\newpath
\moveto(7.8710694,-0.5493249)
\lineto(8.1040125,-0.48112914)
\curveto(8.220484,-0.44703126)(8.752926,-0.37883544)(9.168896,-0.34473756)
\curveto(9.584866,-0.31063965)(10.125627,-0.2697217)(10.5,-0.24926297)
}
\pscustom[linecolor=black, linewidth=0.05]
{
\newpath
\moveto(7.8710694,-0.64932495)
\lineto(8.1040125,-0.8311523)
\curveto(8.220484,-0.92206544)(8.752926,-1.1038928)(9.168896,-1.1948059)
\curveto(9.584866,-1.2857196)(10.125627,-1.3948157)(10.5,-1.4493638)
}
\psdots[linecolor=black, fillstyle=solid, dotstyle=o, dotsize=0.49939746](10.5,-1.45)
\psline[linecolor=black, linewidth=0.05](10.542138,-0.24933279)(13.171069,-0.14932503)(13.171069,-0.14932503)
\psdots[linecolor=black, fillstyle=solid, dotstyle=o, dotsize=0.49939746](10.5,-0.25)
\psdots[linecolor=black, fillstyle=solid, dotstyle=o, dotsize=0.49939746](13.1,-0.15)
\psdots[linecolor=black, dotstyle=otimes, dotsize=0.49939746](10.5,1.15)
\psdots[linecolor=black, dotstyle=otimes, dotsize=0.49939746](10.5,-0.25)
\psdots[linecolor=black, dotstyle=otimes, dotsize=0.49939746](10.5,-1.45)
\psdots[linecolor=black, dotstyle=o, dotsize=0.49939746](15.8,-3.85)
\psdots[linecolor=black, dotstyle=o, dotsize=0.49939746](15.8,0.25)
\psdots[linecolor=black, dotstyle=o, dotsize=0.49939746](2.6,-1.25)
\rput[b](2.64,4.75){\Huge 1}
\rput[b](5.2,4.75){\Huge 2}
\rput[b](7.8,4.75){\Huge 3}
\rput[b](10.5,4.75){\Huge 4}
\rput[b](13.1,4.75){\Huge 5}
\rput[b](15.8,4.75){\Huge 6}
\rput[b](18.4,4.75){\Huge 7}
\rput{90.0}(16.6221,-27.22){\rput(22.01,-5.15){\Large 1}}
\rput{90.0}(17.745,-25.945){\rput(21.795,-4.05){{\Large 2}}}
\rput{90.0}(18.754349,-24.854347){\rput(21.804348,-3.05){{\Large 3}}}
\rput{90.0}(19.445,-24.145){\rput(21.795,-2.35){{\Large 4}}}
\rput{90.0}(20.454348,-23.154348){\rput(21.804348,-1.35){{\Large 5}}}
\rput{90.0}(21.354347,-22.254349){\rput(21.804348,-0.45){{\Large 6}}}
\rput{90.0}(22.354347,-21.254349){\rput(21.804348,0.55){{\Large 7}}}
\rput{90.0}(23.154348,-20.454348){\rput(21.804348,1.35){{\Large 8}}}
\rput{90.0}(23.754349,-19.854347){\rput(21.804348,1.95){{\Large 9}}}
\rput{90.0}(25.054348,-18.554348){\rput(21.804348,3.25){{\Large 10}}}
\end{pspicture}
}
\end{center}
\caption{\label{fig:1}
A tree in $\LTr$. It has $7$ levels numbered
  from the left to the right, $4$ of type (i), $3$ of type (ii). The
  vertices on the same 
  level and the input leaves are ordered from the bottom up.}
\end{figure}
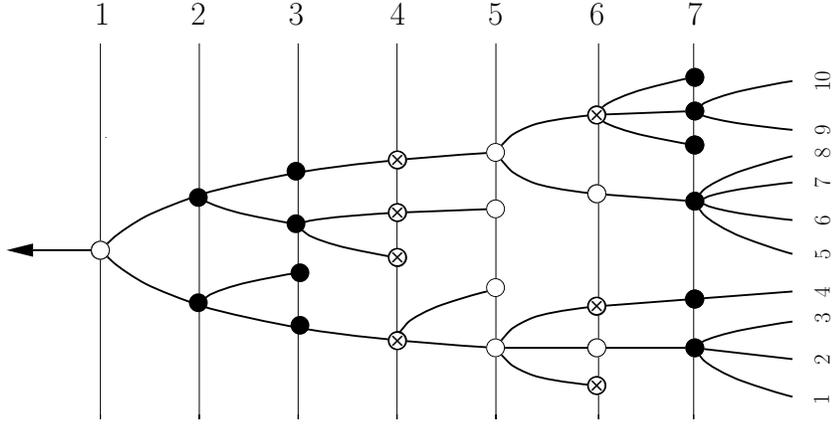
Morphisms in $\LTr$ are generated by three types of `elementary' morphisms:

\noindent 
{\it Type 1.}
Maps of trees $f : \beta \to \alpha$, where $\alpha$ is obtained from $\beta$ by
choosing two adjacent levels and contracting all edges connecting vertices
in these two chosen levels. This contraction results in a single
level with vertices determined by the following rules:
\begin{itemize}
\item[(i)]
contracting an edge adjacent to a \bily-vertex produces a \bily-vertex,
\item[(ii)]
contracting an edge connecting two \tlusty-vertices or a
\cerny-vertex with a \tlusty-vertex produces a \tlusty-vertex and
\item[(iii)]
contracting an edge connecting two \cerny-vertices produces a
\cerny-vertex.
\end{itemize}
If we `order' the types of vertices by 
\begin{equation}
\label{eq:6}
\cerny \prec  \tlusty \prec \bily
\end{equation}
then the above rules say that the `higher takes everything.'

\noindent 
{\it Type 2.}
Maps of trees $f : \beta \to \alpha$, where $\alpha$ is obtained from $\beta$ by
replacing a~\tlusty-vertex by a~\bily-vertex of the same arity, or
by replacing all \cerny-vertices in the same level by \tlusty-vertices.
Therefore only replacements that increase order~(\ref{eq:6})
are allowed.

\noindent 
{\it Type 3.}
Maps of trees $f : \beta \to \alpha$, where $\alpha$ is obtained from
$\beta$ by introducing a new level consisting only of \bily-vertices of
arity $1$.

We require the following relations between these elementary morphisms:
a contraction of two levels does not depend on the order, 
contractions commute with replacement of vertices, and 
introducing a level and then contracting it is an identity morphism.

The category $\LTr$ is operadic. 
The cardinality functor  $\LTr \to \FSet$ assigns to $\beta \in \LTr$
the set $\Wh(\beta)$ of its white vertices. The fibers
of a map $\sigma : \beta \to \alpha$ are preimages of white vertices of $\alpha$. 
Trees in $\LTr$ belong to the same connected component
if they have the same arity (= number of input edges), one therefore clearly has
$\pi_0(\LTr) \cong \bbN$. The (in this case unique) terminal objects are white
corollas $\bily^n$ of arity $n \in \bbN$.

There is an operadic functor 
\begin{equation}
\label{eq:21}
\qOmega:\LTr\to \OmegaN
\end{equation}
assigning to $\beta
\in \LTr$ an
$\bbN$-colored $2$-tree $\bfT = (\EuT,c)$ defined as follows.
Assume that the white and vertical vertices of $\beta$ are lined up, 
from the root down, as in the table
\begin{equation}
\label{eq:3}
\begin{array}{cccc}
 v^1_1 &  v^1_2& \cdots &  v^1_{q_1}
\\
 v^2_1 &  v^2_2& \cdots &  v^2_{q_2}
\\
\vdots & \vdots && \vdots
\\
 v^t_1 &  v^t_2& \cdots &  v^t_{q_t}
\end{array}
\end{equation}
(the table therefore does not show type~(ii) levels).
The $1$-truncation of $\EuT$ is the $1$-ordinal $\{1,\ldots,t\}$
represented by the corolla with $t$ leaves numbered from the left to the right.
The leaves of $\EuT$ are elements of the set
\begin{equation}
\label{Dnes_snad_Jarca_zalije_kyticky}
\big\{v^c_d,\ 1 \leq c \leq t,\ 1 \leq d \leq q_c,\ v^c_d  \mbox { is
  white} \big\}.\footnote{So we use the same symbols for both the white vertices
  of $\beta$ and for the leaves of $\EuT$.}
\end{equation}
The leaf $v^c_d$ is connected to the $c$th $1$-vertex $c$. Finally,
the color $c(v^c_d)$ of the leaf $v^c_d$ is the arity of the vertex
$v^c_d$, while the output color $c(1)$ of $\EuT$ is the arity
of $\beta$. We put $\qOmega(\beta) :=(\EuT,c)$. An example of this construction
is in Figure~\ref{tab:fig2}.  Observe that the functor $\qOmega$ in~(\ref{eq:21})
does not see horizontal vertices.

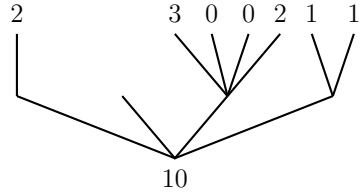
\begin{figure}
\begin{center}
\psscalebox{.7 0.7}
{
\begin{pspicture}(0,-1.24)(6.5653358,1.54)
\psline[linecolor=black, linewidth=0.04](2.1053357,0.0)(3.1053357,-1.1818181)(4.1053357,0.0)(4.1053357,0.0)
\psline[linecolor=black, linewidth=0.04](3.1053357,-1.1818181)(6.1053357,0.0)(6.1053357,0.0)
\psline[linecolor=black, linewidth=0.04](0.10533569,1.1818181)(0.10533569,0.0)(0.10533569,0.0)
\psline[linecolor=black, linewidth=0.04](4.1053357,0.0)(3.1053357,1.1818181)(3.1053357,1.1818181)
\psline[linecolor=black, linewidth=0.04](4.1053357,0.0)(5.1053357,1.1818181)(5.1053357,1.1818181)
\psline[linecolor=black, linewidth=0.04](4.1053357,0.0)(4.505336,1.1818181)(4.505336,1.1818181)
\psline[linecolor=black, linewidth=0.04](4.1053357,0.0)(3.8053358,1.1818181)(3.8053358,1.1818181)
\psline[linecolor=black, linewidth=0.04](6.1053357,0.0)(5.7053356,1.1818181)(5.7053356,1.1818181)
\psline[linecolor=black, linewidth=0.04](6.1053357,0.0)(6.505336,1.1818181)(6.505336,1.1818181)
\rput[b](0.10533569,1.4){\large $2$}
\rput[b](3.1053357,1.4){\large$3$}
\rput[b](3.8053358,1.4){\large$0$}
\rput[b](4.505336,1.4){\large$0$}
\rput[b](5.1053357,1.4){\large$2$}
\rput[b](5.7053356,1.4){\large$1$}
\rput[b](6.505336,1.4){\large$1$}
\rput[t](3.1053357,-1.4){\large$10$}
\psline[linecolor=black, linewidth=0.04](3.1053357,-1.1818181)(0.10533569,0.0)(0.10533569,0.0)
\end{pspicture}
}
\end{center}
\caption{An $\bbN$-colored  $2$-tree corresponding to the
  leveled tree of Figure~\ref{fig:1}.}
  \label{tab:fig2}
\end{figure}

\subsection{Endomorphism object for levelled trees}
\label{LTrendo}
By Proposition~\ref{omegaendos},
each collection
$\scE = \coll \scE$ of objects of a~duoidal category $\duo = (\duo,
\boxx_0,\boxx_1,e,v)$ has its $\bbN$-colored endomorphism $2$-operad
$\End_\scE^{\OmegaN}$. Likewise, one can define
a $\LTr$-collection $\End_\scE^{\LTr} =
\big\{\End_\scE^{\LTr}(\beta)\big\}_{\beta \in \LTr}$ which becomes, under the
mild assumption of the commutativity of diagram~(\ref{eq:32b}) below, 
an operad in $\OpV \LTr$.

Its construction is similar to the one in the proof of 
Proposition~\ref{omegaendos}.  Assume that all (not only \bily- and
\tlusty-) vertices of a leveled tree $\beta \in \LTr$ are organized as
in the table
\begin{equation}
\label{eq:3all}
\begin{array}{cccc}
u^1_1 &  u^1_2& \cdots &  u^1_{p_1}
\\
u^2_1 &  u^2_2& \cdots &  u^2_{p_2}
\\
\vdots & \vdots && \vdots
\\
u^\ell_1 &  u^\ell_2& \cdots &  u^\ell_{p_\ell}
\end{array}
\end{equation}
Notice that necessarily $p_1 = 1$.
Then  we define, for $1 \leq b \leq \ell$, 
\[
\scE^{\boxx_1^{p_b}}_\beta := 
 \scE(u^b_1) \boxx_1 \cdots  
    \boxx_1 \scE(u^b_{p_b}),
\]
where, for $1 \leq c \leq p_b$,
\begin{equation}
\label{eq:14}
\scE(u^b_c) 
:= \tricases{\scE(n^b_c)}{if $u^b_c$ is a white vertex of arity $n^b_c$,}
v{if  $u^b_c$ is a vertical vertex, and}e{if $u^b_c$ is a horizontal vertex.}
\end{equation}
Vertical vertices \tlusty\ are therefore represented by the vertical unit
$v$ and the horizontal vertices~\cerny\ by the horizontal unit $e$, 
which explains the terminology. Finally we put
\[
\scE^{\beta} := \scE^{\boxx_1^{p_\ell}}_\beta\boxx_0\cdots\boxx_0 
\scE^{\boxx_1^{p_1}}_\beta
\]
or, in the expanded form,
\begin{equation}
\label{eq:15}
\scE^{\beta} =
\big(\scE(u^\ell_1) \boxx_1 \cdots  
\boxx_1 \scE(u^\ell_{p_\ell}) \big) \boxx_0 \cdots \boxx_0  
\big(\scE(u^1_1) \boxx_1 \cdots  
    \boxx_1 \scE(u^1_{p_1})\big),
\end{equation}
and define $\End_\scE^\LTr(\beta) := \duo\big(\scE^\beta,\scE(n)\big)$ with $n$
the arity of $\beta$.

One has a natural morphism  of
$\LTr$-collections
\begin{equation}
\label{eq:12}
\Lambda:
\End_\scE^\LTr \to  \qOmega^* \End_\scE^\OmegaN
\end{equation}
whose components
\[
\Lambda_\beta:
\End_\scE^\LTr(\beta) \to  \End_\scE^\OmegaN(\frT), \ \beta \in \LTr, \
\frT := \qOmega(\beta),
\]
are induced by a natural map 
\begin{equation}
\label{eq:tttt}
\theta_\beta : \scE^\frT \to \scE^\beta  
\end{equation}
defined as follows.
Let us introduce first an auxiliary reduced $\boxx$-power
\[
\overline \scE^{\beta} := \overline \scE^{\boxx_1^{p_\ell}}_\beta\boxx_0 
\scE^{\boxx_1^{p_{\ell-1}}}_\beta 
\boxx_0\cdots\boxx_0 \overline \scE^{\boxx_1^{p_2}}_\beta\boxx_0
\overline \scE^{\boxx_1^{p_1}}_\beta
\]
whose factors are, for $1 \leq b \leq \ell$, given as
\[
\overline \scE^{\boxx_1^{p_b}} :=
\cases{\scE^{\boxx_1^{p_b}}}{if the $b$th level of $\beta$ is of type (i),
and}e{if the $b$th level of $\beta$ is of type (ii).}
\]

Clearly, $\overline \scE^{\beta}$ is obtained from $\scE^\beta$ by replacing
multiple $\boxx_1$-powers of $e$ by a single instance of~$e$. 
The comonoid structure 
\begin{equation}
\label{eq:34}
e \to e \boxx _1 e
\end{equation}
therefore gives rise to a canonical map $\varpi : \overline \scE^\beta \to
\scE^\beta$. More precisely, 
\[
\varpi =  \varpi^\ell \boxx_0 \cdots \boxx_0 \varpi^1 :  
\overline \scE^\beta =\overline  \scE^{\boxx_1^{p_\ell}}_\beta\boxx_0\cdots\boxx_0 
\overline \scE^{\boxx_1^{p_1}}_\beta  \longrightarrow 
 \scE^{\beta} =  \scE^{\boxx_1^{p_\ell}}_\beta\boxx_0\cdots\boxx_0 
\scE^{\boxx_1^{p_1}}_\beta,
\]
where $\varpi^b : \overline \scE^{\boxx_1^{p_b}}_\frT \to 
\scE^{\boxx_1^{p_b}}_\beta$, $1 \leq b \leq \ell$, are the following
canonical morphisms.

If the $b$th level of $\beta$ is of type (i), then
$\varpi^b$ is the identity
$\overline \scE^{\boxx_1^{p_b}}_\beta =    \scE^{\boxx_1^{p_b}}_\beta$.
When the $b$th level is of type (ii), then by definition
$\overline \scE^{\boxx_1^{p_b}}_\beta = e$ while
$\scE^{\boxx_1^{p_b}}_\beta = e \boxx_1 \cdots \boxx_1 e$
($p_b$-times). In this case we take as $\varpi_b$ the map
given by iterating the comonoid structure of $e$.

Comparing the definitions of $\overline \scE^\beta$ and 
$\scE^\frT$ we see that $\overline \scE^\beta$ differs from $\scE^\frT$ by
(possibly multiple) $\boxx_0$-products with $e$ and/or
$\boxx_1$-products with $v$. Recalling that $e$ (resp.~$v$) 
is the unit for the horizontal (resp.~vertical) 
multiplication in $\duo$, we conclude that $\overline \scE^\beta$ and $\scE^\frT$ are
canonically isomorphic. The map $\Lambda_\beta$ in~(\ref{eq:tttt}) is  the
composition
$\scE^\frT \cong \overline \scE^\beta \stackrel\varpi\longrightarrow  \scE^\beta$.

\begin{remark}
One can check that if the diagram
\begin{equation}
\label{eq:32b}
\xymatrix@R = +1.6em{e \boxx_1 e \ar[r] \ar[d] & e \boxx_1 v\ar[d]
\cr
v \boxx_1 e \ar[r] & e
}
\end{equation}
induced by the canonical map $e \to v$ and by the unit constraint for $v$
commutes, the construction of Section~\ref{LTrendo} gives rise to
an $\LTr$-multicotensor that induces a natural $\LTr$-operad
structure on $\End^\LTr_\scE$.

This in particular happens when the 
comultiplication~(\ref{eq:34}) is an isomorphism, in which case 
$\varpi$ is an isomorphism, too, and
$\Lambda$ is an isomorphism of $\LTr$-operads. An important instance
of such a situation is
when $\duo$ is the duoidal category $\funny(\ttC,\ttV)$ of
span $\ttV$-objects over a small category $\ttC$ 
\cite[Def.~6.3]{batanin-markl:Adv}.
\end{remark}

\begin{theorem}
\label{etiopska_restaurace}
For each multiplicative $1$-operad $\scA$ in a duoidal category $\duo$
there exists  a~natural  map of\/ $\LTr$-collections
\begin{equation}
\label{ltraction}
\Psi: \bfone^\LTr \to \End^\LTr_\scA 
\end{equation}
 such that the composite
\begin{equation}
\label{eq:33}
\Xi:
\bfone^\LTr \stackrel\Psi\longrightarrow \End^\LTr_\scA 
\stackrel\Lambda\longrightarrow \qOmega^* \End_\scA^\OmegaN
\end{equation}
is a morphism of\/ $\LTr$-operads. If the diagram~(\ref{eq:32b})
commutes, all maps
in~(\ref{eq:33}) are operad morphisms.
\end{theorem}

\begin{proof}
The map~(\ref{eq:19}) gives rise to a map $\alpha_v : v \to
\scA(n)$ 
and, when precomposed with  
the canonical map $e\to v$, to a  map $\alpha_e : e \to
\scA(n)$ for each $n \geq 0$.
By the definition of $ \bfone^\LTr$, morphism~(\ref{ltraction}) is
determined by specifying a map $\Psi_\beta : \scA^\beta \to \scA(n)$ for
each $\beta\in \LTr$, where $n$ is the arity of $\beta$. 
Assume again that the vertices of $\beta$ are as in
table~(\ref{eq:3all}) and denote, only for
purposes of this proof,
\begin{equation}
\label{eq:17}
\scA(\beta): =
\big(\scA(n^\ell_1) \boxx_1 \cdots  
\boxx_1 \scA(n^\ell_{p_\ell}) \big) \boxx_0 \cdots \boxx_0  
\big(\scA(n^1_1) \boxx_1 \cdots  
    \boxx_1 \scA(n^1_{p_1})\big)
\end{equation}
where $n^b_c$ is the arity of the vertex $v^b_c$, $1\leq b \leq \ell$,
$1 \leq c \leq p_b$. It is clear that the structure morphisms~(\ref{morph}) of the
$1$-operad $\scA$ give rise to a map $\gamma_\beta :   \scA(\beta) \to \scA(n)$.

For $\scA(u^b_c)$ as in~(\ref{eq:14}) define $\omega^b_c : \scA(u^b_c) \to
\scA(n^b_c)$ by
\[
\omega^b_c := \tricases{\mbox {the identity $\id : \scA(n^b_c) \to
    \scA(n^b_c)$}}{if $u^b_c$ is white}
          {\mbox {the map $\alpha_v: v \to
    \scA(n^b_c)$}}{if $u^b_c$ is vertical, and}
      {\mbox {the map $\alpha_e : e \to
    \scA(n^b_c)$}}{if $u^b_c$ is horizontal.}
\] 
Comparing~(\ref{eq:17}) and~(\ref{eq:15}), we see that the above maps
assemble to a morphism
\begin{equation}
\label{eq:20}
\omega_\beta : \scA^\beta \to \scA(\beta)
\end{equation}
We finally define $\Psi_\beta: \scA^\beta \to \scA(n)$ as the composite
\[
\Psi_\beta:
\scA^\beta  \stackrel{\omega_\beta}\lra   \scA(\beta) \stackrel{\gamma_\beta}\lra \scA(n).
\]

The category $\LTr$ was designed to model 
the `pasting schemes' for multiplicative
$1$-operads and all related constructions were `tautological.'
This makes the desired properties of the above objects obvious. 
\end{proof}

\section{The Tamarkin operad}
\label{sec:tamark-tsyg-oper}

We are going to give an alternative definition of Tamarkin's
operad $\seq$ acting on dg-categories and study various related categories
and operads. The second part of this section is devoted to the
construction of the functor~(\ref{eq:functoru}) needed in
Section~\ref{sec:acti-tamark-tsyg}.  
We consider as in the introduction 
to~\cite{batanin-berger} the pruned $\bbN$-colored $2$-operad
$\opTam_2^\bbN$ given by the pullback
\begin{equation}
\label{Tam2}  
   \xymatrix@C = +3em{
    \opTam_2^\bbN \ar[d]_{} 
    \ar[r]^{} 
    &\opK^{(2)} /a_2\ar[d]
    \\ 
   Des_2(\Lat^{(2)})
    \ar[r]^{Des_2(c)}
    &\, Des_2(\opK^{(2)}) 
  }
\end{equation}
where $\Lat^{(2)}$ is the second filtration of the lattice-path
operad,
 $\opK^{(2)}$ is the second filtration of the complete graph operad of
Berger \cite{berger:CM202}, $c: \Lat^{(2)} \to \opK^{(2)}$ 
is the complexity index functor and
$a_2$ is the canonical internal $2$-operad in $\opK^{(2)}$ consisting of
$2$-ordinals. 
The prominent r\^ole of  $\opTam_2^\bbN$ is given by the fact that 
the operad $\seq$ of~\cite[\S5.6.1]{tamarkin:CM07} acting on dg-categories
equals its restriction along the pruning
functor~(\ref{pruning}). To keep the notation compatible with the rest
of the paper, we denote the Tamarkin operad by $\opTm^\bbN_2$ and take
the above observations as its definition. 

\begin{definition}
\label{musim_objet_Ryn}
The {\em Tamarkin operad\/} $\opTm^\bbN_2$ is the restriction
$p^*\opTam_2^\bbN$ of the pullback~(\ref{Tam2}) along the pruning
functor $p:\OmegaN \to \OrN$.
\end{definition}

An alternative description of $\opTm^\bbN_2$ 
is given in Remark~\ref{pristi_tyden_posledni}.
Without going into details of all objects above, we give below  
an explicit description of  $\opTam_2^\bbN$ and the related~objects. 

It was shown in \cite[Prop.~4.10]{batanin-berger-markl}
or~\cite[Prop.~2.14]{batanin-berger} that
$\Lat^{(2)}$ is isomorphic to the Tamarkin-Tsygan operad
\cite{tamarkin-tsygan:LMP01}. Our exposition will use this description of $\Lat^{(2)}$.
The operad $\Lat^{(2)}$ is an ordinary symmetric
$\bbN$-colored operad in the monoidal category $\Set$ of sets. Its component 
$\Lat^{(2)}(n_1,\ldots,n_k;m)$ is the set of planar rooted trees $\delta$ with
$m$ input leaves and two types of vertices:
\begin{itemize} 
\item[(i)]
white vertices \bily\ labelled $\rada 1k$ such that the
$i$th vertex has arity $n_i$, and
\item[(ii)]
black vertices \cerny\ of arbitrary arities different from~1.
\end{itemize}
We moreover require that $\delta$ has no internal edge connecting
two black vertices. An example is given in Figure~\ref{tam-tree}.
\begin{figure}[t]
  \centering
\psscalebox{.5 .5}
{
\begin{pspicture}(7,-5.462137)(18.44407,3.2)
\definecolor{colour0}{rgb}{0.03137255,0.003921569,0.003921569}
\psline[linecolor=black, linewidth=0.02](2.597499,1.4627581)(2.597499,1.4627581)
\pscustom[linecolor=black, linewidth=0.05]
{
\newpath
\moveto(8.9,0.86213744)
}
\pscustom[linecolor=black, linewidth=0.05]
{
\newpath
\moveto(21.5,2.1621375)
}

\pscustom[linecolor=black, linewidth=0.05]
{
\newpath
\moveto(13.113208,-4.1376762)
\lineto(13.185376,-4.2847466)
\curveto(13.22146,-4.358282)(13.309092,-4.5053515)(13.36064,-4.5788865)
\curveto(13.412188,-4.652422)(13.597764,-4.7994924)(13.73179,-4.8730273)
\curveto(13.865815,-4.9465623)(14.185415,-5.049512)(14.742138,-5.137754)
}
\pscustom[linecolor=black, linewidth=0.05]
{
\newpath
\moveto(13.113208,-4.1376762)
\lineto(13.349299,-3.9023695)
\curveto(13.467346,-3.7847161)(13.754028,-3.5494099)(13.922666,-3.4317565)
\curveto(14.091303,-3.314103)(14.698396,-3.0787964)(15.136852,-2.9611425)
\curveto(15.575308,-2.8434894)(16.620857,-2.678775)(18.442139,-2.537591)
}
\pscustom[linecolor=black, linewidth=0.05]
{
\newpath
\moveto(10.913208,0.66272706)
\lineto(11.127149,0.55977356)
\curveto(11.234118,0.5082965)(11.493901,0.40534303)(11.646716,0.35386598)
\curveto(11.799531,0.3023889)(12.349662,0.19943604)(12.746979,0.14795898)
\curveto(13.144296,0.096481934)(14.091743,0.024414062)(15.742138,-0.03735826)
}
\psline[linecolor=black, linewidth=0.05](13.113208,-4.1376762)(13.113208,-4.1376762)
\psline[linecolor=black, linewidth=0.05](13.113208,-4.1376762)(15.571069,-4.1376762)(15.571069,-4.1376762)
\pscustom[linecolor=black, linewidth=0.05]
{
\newpath
\moveto(16.871069,-0.1296814)
\lineto(17.006544,0.16404602)
\curveto(17.074282,0.31090942)(17.383938,0.60463715)(17.625858,0.75150084)
\curveto(17.867779,0.89836454)(18.182274,1.0746012)(18.4,1.1627194)
}
\pscustom[linecolor=black, linewidth=0.05]
{
\newpath
\moveto(15.7006035,-4.1376762)
\lineto(15.93979,-3.978573)
\curveto(16.059383,-3.899021)(16.606096,-3.7399182)(17.033216,-3.6603663)
\curveto(17.460337,-3.5808148)(18.015593,-3.4853528)(18.4,-3.437622)
}
\pscustom[linecolor=black, linewidth=0.05]
{
\newpath
\moveto(16.947515,-0.1296814)
\lineto(17.076218,-0.0072949217)
\curveto(17.140568,0.053898316)(17.434742,0.17628479)(17.664566,0.23747803)
\curveto(17.894388,0.29867128)(18.19316,0.37210265)(18.4,0.40881887)
}
\pscustom[linecolor=black, linewidth=0.05]
{
\newpath
\moveto(14.542139,1.8628045)
\lineto(14.881412,1.7946087)
\curveto(15.051047,1.7605109)(15.826527,1.6923151)(16.432371,1.6582172)
\curveto(17.038214,1.6241193)(17.825811,1.5832013)(18.371069,1.5627426)
}
\pscustom[linecolor=black, linewidth=0.05]
{
\newpath
\moveto(9.263346,-0.637583)
\lineto(9.377826,-0.3193872)
\curveto(9.435066,-0.1602893)(9.696737,0.15790649)(9.901168,0.31700438)
\curveto(10.105598,0.4761023)(10.371359,0.66702026)(10.555346,0.7624788)
}
\pscustom[linecolor=black, linewidth=0.05]
{
\newpath
\moveto(15.7006035,-4.1376762)
\lineto(15.93979,-4.4331536)
\curveto(16.059383,-4.5808926)(16.606096,-4.8763695)(17.033216,-5.024109)
\curveto(17.460337,-5.171847)(18.015593,-5.3491344)(18.4,-5.437777)
}
\pscustom[linecolor=black, linewidth=0.05]
{
\newpath
\moveto(16.871069,-0.1296814)
\lineto(17.006544,-0.4723633)
\curveto(17.074282,-0.64370483)(17.383938,-0.9863867)(17.625858,-1.1577277)
\curveto(17.867779,-1.3290687)(18.182274,-1.5346777)(18.4,-1.6374824)
}
\pscustom[linecolor=black, linewidth=0.05]
{
\newpath
\moveto(16.947515,-0.1296814)
\lineto(17.076218,-0.25206786)
\curveto(17.140568,-0.31326112)(17.434742,-0.43564758)(17.664566,-0.49684083)
\curveto(17.894388,-0.55803406)(18.19316,-0.6314661)(18.4,-0.6681818)
}
\psline[linecolor=black, linewidth=0.05](15.700603,-4.1376762)(18.4,-4.4377)(18.4,-4.4377)
\psdots[linecolor=black, fillstyle=solid, dotstyle=o, dotsize=0.010382021](10.5,0.86213744)
\psdots[linecolor=black, fillstyle=solid, dotstyle=o, dotsize=0.49939746](13.1,-4.1378627)
\psline[linecolor=black, linewidth=0.05, arrowsize=0.05291666666666672cm 6.0,arrowlength=2.0,arrowinset=0.0]{<-}(5.6982164,-1.5374746)(7.725617,-1.5374746)
\psdots[linecolor=black, fillstyle=solid, dotstyle=o, dotsize=0.49939746](12.0,-2.8378625)
\psdots[linecolor=black, dotsize=0.49939746](14.4,1.8621374)
\psdots[linecolor=black, dotsize=0.49939746](17.0,-0.13786255)
\pscustom[linecolor=black, linewidth=0.05]
{
\newpath
\moveto(10.542139,-3.7371955)
\lineto(10.651031,-3.5553596)
\curveto(10.705476,-3.4644415)(10.954375,-3.2826061)(11.148826,-3.1916883)
\curveto(11.343276,-3.1007702)(11.596065,-2.9916687)(11.77107,-2.9371178)
}
\pscustom[linecolor=black, linewidth=0.05]
{
\newpath
\moveto(9.236145,-0.8371875)
\lineto(9.350626,-0.7689917)
\curveto(9.4078665,-0.7348938)(9.669538,-0.666698)(9.873968,-0.63260007)
\curveto(10.078399,-0.5985022)(10.344159,-0.5575842)(10.528146,-0.5371255)
}
\psdots[linecolor=black, fillstyle=solid, dotstyle=o, dotsize=0.49939746](10.5,-0.53786254)
\psdots[linecolor=black, dotstyle=otimes, dotsize=0.010382021](10.5,0.86213744)
\psdots[linecolor=black, fillstyle=solid, dotstyle=o, dotsize=0.49939746](14.4,-4.1378627)
\psdots[linecolor=black, fillstyle=solid, dotstyle=o, dotsize=0.49939746](15.8,-0.03786255)
\pscustom[linecolor=black, linewidth=0.05]
{
\newpath
\moveto(7.860071,-1.7371875)
\lineto(8.087773,-2.237196)
\curveto(8.201625,-2.487201)(8.722088,-2.9872096)(9.128698,-3.2372143)
\curveto(9.53531,-3.4872186)(10.063905,-3.787224)(10.429855,-3.9372263)
}
\pscustom[linecolor=black, linewidth=0.05]
{
\newpath
\moveto(7.887272,-1.3375831)
\lineto(8.001753,-1.2239331)
\curveto(8.058993,-1.1671082)(8.320664,-1.053457)(8.525095,-0.9966321)
\curveto(8.729526,-0.9398071)(8.995284,-0.8716162)(9.179273,-0.8375212)
}
\pscustom[linecolor=black, linewidth=0.05]
{
\newpath
\moveto(10.671069,-4.0371876)
\lineto(10.8862915,-4.0599236)
\curveto(10.993901,-4.0712914)(11.485837,-4.0940275)(11.870161,-4.1053953)
\curveto(12.254486,-4.116764)(12.754108,-4.1304054)(13.1,-4.1372266)
}
\psdots[linecolor=black, fillstyle=solid, dotstyle=o, dotsize=0.49939746](13.1,-4.1378627)
\pscustom[linecolor=black, linewidth=0.05]
{
\newpath
\moveto(10.842138,0.96280456)
\lineto(11.154829,1.1673676)
\curveto(11.311174,1.269649)(12.025893,1.4742126)(12.584269,1.5764941)
\curveto(13.142644,1.6787757)(13.868533,1.8015133)(14.371069,1.8628823)
}
\pscustom[linecolor=black, linewidth=0.05]
{
\newpath
\moveto(14.542139,1.9628046)
\lineto(14.881412,2.0991857)
\curveto(15.051047,2.1673768)(15.826527,2.303758)(16.432371,2.3719485)
\curveto(17.038214,2.440139)(17.825811,2.521968)(18.371069,2.5628822)
}
\psdots[linecolor=black, dotsize=0.49939746](9.3,-0.83786255)
\psdots[linecolor=black, fillstyle=solid, dotstyle=o, dotsize=0.49939746](7.8,-1.5378625)
\psdots[linecolor=black, fillstyle=solid, dotstyle=o, dotsize=0.49939746](10.7,0.7621375)
\psline[linecolor=colour0, linewidth=0.05](16.0,-0.03786255)(17.1,-0.13786255)
\psdots[linecolor=black, dotsize=0.49939746](10.5,-3.9378626)
\psdots[linecolor=black, dotsize=0.49939746](14.9,-5.1378627)
\psdots[linecolor=black, dotsize=0.49939746](15.8,-4.1378627)
\rput(7.6,-1.0178625){\textcolor{colour0}{\Large 1}}
\rput(12.8,-3.6378624){\textcolor{colour0}{\Large 2}}
\rput(12.2,-2.2378625){\textcolor{colour0}{\Large 3}}
\rput(10.7,-0.083786255){\textcolor{colour0}{\Large 4}}
\rput(10.4,1.2621374){\textcolor{colour0}{\Large 5}}
\rput[bl](15.8,0.4213746){\Large 7}
\rput(14.7,-3.678626){\textcolor{colour0}{\Large 6}}
\end{pspicture}
}
\caption{A tree $\delta \in \Lat^{(2)}(2,3,0,0,2,1,1;10)$.}
  \label{tam-tree}
\end{figure}
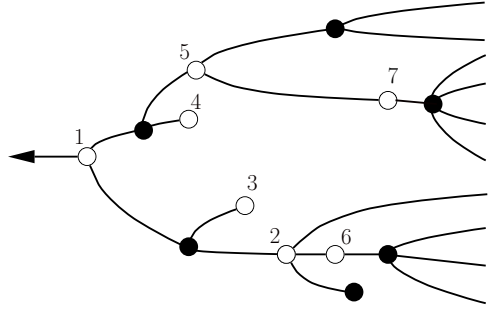

The symmetric groups permute the labels of white vertices. The
$\circ_i$-operation
\[
\Lat^{(2)}(n_1,\ldots,n_k;m)\! \times\! \Lat^{(2)}(m_1,\ldots,m_l;n_i) 
\to \Lat^{(2)}(n_1,\ldots,n_{i-1}, m_1,\ldots,m_l,n_{i+1},\ldots,n_k;m)
\]
inserts the tree $\gamma \in\Lat^{(2)}(m_1,\ldots,m_l;n_i) $ to the
$i$th vertex of $\delta \in \Lat^{(2)}(n_1,\ldots,n_k;m)$ and
contracts the edges connecting two black vertices if necessary.  The
units are represented by planar corollas whose unique vertex is white.

There is a canonical {\em complementary order\/}~\cite[Def.~2.2]{batanin:conf}
on the set $\Wh(\delta)$ of white vertices of a
tree $\delta\in \Lat^{(2)}$ given as follows. For $i,j \in
\Wh(\delta)$,
\begin{itemize} 
\item[(i)] 
$i \triangleleft_0 j$ \ if and only if
  there is a directed path in $\delta$ from  $i$ to $j$, and
\item[(ii)]
$i \triangleleft_1 j$ \ if and only if the edge path connecting
$i$ with the root lies on the left from the path connecting
$j$ to the root, where `left' refers to the planar structure of $\delta$.
\end{itemize}

The following definition will be useful in the sequel; recall that
$\bbN$-colored $2$-ordinals were reviewed in Section~\ref{OmegaN}.

\begin{definition}
An {\em $\OrN$\/-labelled tree} is a couple $(\frO,\delta)$, where  $\delta \in
  \Lat^{(2)}$ and $\frO \in \OrN$ is an $\bbN$-colored $2$-ordinal whose
  underlying set is
  $\Wh(\delta)$ and the coloring given by  the
  arity of the corresponding vertex of $\delta$.
\end{definition}

Let $\triangleleft$ be some complementary order on the underlying
set of a $2$-ordinal $\EuO$. We say
that $\EuO$ {\em dominates\/} $\triangleleft$ if, for all $i,j \in \EuO$, 
\begin{equation}
\label{eq:32}
i  \not\triangleright_0\ j \ \mbox { if }\ i <_0 j,\  \mbox { and }\
i \triangleleft_1 j \ \mbox { if } \ i <_1 j.
\end{equation} 

The pullback (\ref{Tam2}) represents the pruned  $2$-operad
$\opTam_2^\bbN$ as a suboperad of the desymmetrisation of the operad
$\Lat^{(2)}$ whose arity- $\frO$ operations
are  $\OrN$-labelled trees  $(\frO,\delta)$ such that 
$\frO$ dominates the canonical complementary
order on the set $\Wh(\delta)$ of white vertices of~$\delta$.

Let us define the operadic category $\Tam_2^\bbN$ 
as the Grothendieck construction
of Proposition~\ref{sec:discr-fibr-betw} applied to the pruned
$2$-operad $\opTam_2^\bbN$. By definition, the
objects of $\Tam_2^\bbN$ are pairs $(\frO,\delta) \in
\opTam_2^\bbN(\frO)$ 
of $\OrN$-labeled trees such that $\frO$ dominates the canonical
complementary order on $\Wh(\delta)$.
The rest of this section will be devoted to the definition of a functor
\begin{equation}
\label{eq:functoru}
u : \LTr \to \Tam_2^\bbN,
\end{equation}
where $\LTr$ is the category of leveled trees introduced in
Section~\ref{LTr}. This functor will play a key r\^ole in
Section~\ref{sec:acti-tamark-tsyg}. 

We start by noticing that the set $\Wh(\beta)$ of \bily-vertices of
$\beta \in \LTr$ has a natural lexicographic order defined by saying
that $u < v$ if the level of $u$ is closer to the root than the level
of $v$; if $u$ lies on the same level as $v$ then $u < v$ if and only
if $u$ is on the left from $v$ in the sense of the planar structure of
$\beta$.  Given a leveled tree $\beta\in \LTr$, we consider an
$\OrN$-labelled tree $u(\beta) = (\frO,\bar{\beta})$ with $\frO :=
p\big(\Omega(\beta)\big)$, where $\Omega : \LTr \to \OmegaN$ is the functor
in~(\ref{eq:21}), $p: \OmegaN \to \OrN$ the pruning functor, and
$\bar{\beta}$ produced from $\beta$ in four steps:
\begin{itemize}
\item[(i)] 
labelling the white vertices of $\beta$ by
$\rada 1k$ using the above
lexicographical order,   
\item[(ii)] 
forgetting the level structure of $\beta$,
\item[(iii)] 
converting  all \tlusty-vertices of $\beta$ to \cerny-vertices and
\item[(iv)] 
contracting edges connecting two \cerny-vertices 
and erasing all \cerny-vertices of arity  one.
\end{itemize}
As an exercise, we recommend to check that 
applying the above steps to the tree $\beta$ in
Figure~\ref{fig:1} produces the tree in Figure~\ref{tam-tree}. To
show that indeed $u(\beta) = (\frO,\bar{\beta}) \in \Tam^\bbN_2$ means
to verify that the $2$-ordinal $\frO$ dominates the complementary
order generated by~$\bar{\beta}$.

Let $i<_0 j$ in $\frO$ and let $w_i, w_j$ be the corresponding white
vertices from $\bar{\beta}$. The relation $i<_0 j$ means that, in
$\beta$, the vertex $w_j$ lies on the level closer to the root than the
level of $w_i$.  There are two possibilities: either $w_i$ and $w_j$
are connected by a directed path and then $w_i\triangleleft_0 w_j$
in $\Wh(\bar{\beta})$, or there is no such a directed path, in which
case either $w_i\triangleleft_1 w_j$ or $w_j \triangleleft_1 w_1$ in $\bar{\beta}$.  The
domination condition~(\ref{eq:32}) clearly holds for the pair $i,j$ in this
case.

Assume that $i<_1 j$ in $\frO$. This means that $w_i$ precedes 
$w_j$ in the lexicographical order and also that there is no directed
path connecting $w_i$ and $w_j$ in $\beta$. Hence $w_i \triangleleft_1 w_j$ in
$\Wh(\bar{\beta})$ which finishes the verification that $\frO$
dominates $\bar{\beta}$. The idea is indicated in Figure~\ref{ide}.

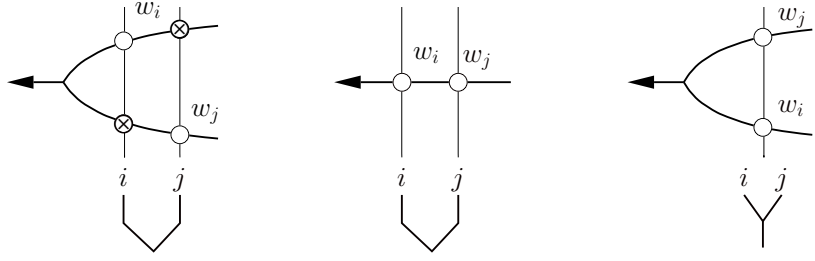
\begin{figure}[t]
  \centering 
\psscalebox{.5 .5} 
{
\begin{pspicture}(-4,-5.2676873)(26.71,2)
\definecolor{colour0}{rgb}{0.03137255,0.003921569,0.003921569}
\psline[linecolor=black, linewidth=0.02](5.997484,1.2684633)(5.997484,-2.6270506)(6.0,-2.7323127)
\pscustom[linecolor=black, linewidth=0.05]
{
\newpath
\moveto(8.9,0.66768736)
}
\pscustom[linecolor=black, linewidth=0.05]
{
\newpath
\moveto(21.5,1.9676874)
}

\pscustom[linecolor=black, linewidth=0.05]
{
\newpath
\moveto(2.897484,-0.7319244)
\lineto(3.0209558,-0.905025)
\curveto(3.0826917,-0.9915753)(3.25761,-1.1646746)(3.3707926,-1.251225)
\curveto(3.483975,-1.3377752)(3.7823648,-1.5108758)(3.9675722,-1.5974262)
\curveto(4.15278,-1.6839764)(4.6878242,-1.8570758)(5.0376606,-1.943626)
\curveto(5.387497,-2.0301764)(6.056302,-2.1455762)(7.0132074,-2.2321265)
}

\pscustom[linecolor=black, linewidth=0.05]
{
\newpath
\moveto(2.897484,-0.7319244)
\lineto(3.0209558,-0.5588245)
\curveto(3.0826917,-0.47227478)(3.25761,-0.2991742)(3.3707926,-0.2126239)
\curveto(3.483975,-0.12607361)(3.7823648,0.047025755)(3.9675722,0.13357605)
\curveto(4.15278,0.22012635)(4.6878242,0.39322633)(5.0376606,0.47977662)
\curveto(5.387497,0.5663269)(6.056302,0.6817267)(7.0132074,0.768277)
}
\psline[linecolor=black, linewidth=0.02](4.526415,1.2684633)(4.526415,-2.7323127)(4.526415,-2.682303)
\psdots[linecolor=black, fillstyle=solid, dotstyle=o, dotsize=0.49939746](6.0,-2.1323125)
\psdots[linecolor=black, fillstyle=solid, dotstyle=o, dotsize=0.49939746](6.0,0.66768736)
\psdots[linecolor=black, dotstyle=otimes, dotsize=0.49939746](6.0,0.66768736)
\psdots[linecolor=black, fillstyle=solid, dotstyle=o, dotsize=0.49939746](4.5,0.36768737)
\rput{90.0}(31.677687,-41.742313){\rput[bl](36.71,-5.0323124){1}}
\psline[linecolor=black, linewidth=0.05, arrowsize=0.05291666666666672cm 6.0,arrowlength=2.0,arrowinset=0.0]{<-}(1.4,-0.7319247)(2.8974843,-0.7319247)
\psdots[linecolor=black, dotstyle=o, dotsize=0.49939746](4.5,-1.8323126)
\psdots[linecolor=colour0, dotstyle=otimes, dotsize=0.5](4.5,-1.8323126)
\psline[linecolor=colour0, linewidth=0.05](4.5,-3.7323127)(4.5,-4.4823127)(5.3,-5.2323127)(6.0,-4.4823127)(6.0,-3.7323127)(6.0,-3.7323127)
\psline[linecolor=black, linewidth=0.05, arrowsize=0.05291666666666672cm 6.0,arrowlength=2.0,arrowinset=0.0]{<-}(10.1,-0.7319247)(14.797484,-0.7319247)
\psline[linecolor=black, linewidth=0.02](13.397484,1.2684633)(13.397484,-2.6270506)(13.4,-2.7323127)
\psline[linecolor=black, linewidth=0.02](11.897484,1.2684633)(11.897484,-2.6270506)(11.9,-2.7323127)
\psdots[linecolor=colour0, dotstyle=o, dotsize=0.5](11.9,-0.7323126)
\psdots[linecolor=colour0, dotstyle=o, dotsize=0.5](13.4,-0.7323126)
\psline[linecolor=colour0, linewidth=0.05](11.9,-3.7323127)(11.9,-4.4823127)(12.7,-5.2323127)(13.4,-4.4823127)(13.4,-3.7323127)(13.4,-3.7323127)
\pscustom[linecolor=black, linewidth=0.05]
{
\newpath
\moveto(19.397484,-0.7319244)
\lineto(19.499956,-0.8934863)
\curveto(19.551191,-0.97426695)(19.69636,-1.1358289)(19.790293,-1.2166095)
\curveto(19.884224,-1.2973907)(20.131865,-1.4589521)(20.285572,-1.5397333)
\curveto(20.43928,-1.6205145)(20.883326,-1.7820758)(21.173662,-1.8628571)
\curveto(21.463999,-1.9436377)(22.019053,-2.051345)(22.813208,-2.1321263)
}
\pscustom[linecolor=black, linewidth=0.05]
{
\newpath
\moveto(19.397484,-0.7319244)
\lineto(19.499956,-0.57036316)
\curveto(19.551191,-0.4895819)(19.69636,-0.32802063)(19.790293,-0.24723938)
\curveto(19.884224,-0.16645813)(20.131865,-0.004896851)(20.285572,0.0758844)
\curveto(20.43928,0.15666504)(20.883326,0.31822664)(21.173662,0.39900756)
\curveto(21.463999,0.4797882)(22.019053,0.58749634)(22.813208,0.668277)
}
\psline[linecolor=black, linewidth=0.02](21.526415,1.2684633)(21.526415,-2.7323127)(21.526415,-2.682303)
\psdots[linecolor=black, dotstyle=o, dotsize=0.49939746](21.5,0.46768737)
\psline[linecolor=black, linewidth=0.05, arrowsize=0.05291666666666672cm 6.0,arrowlength=2.0,arrowinset=0.0]{<-}(17.9,-0.7319247)(19.397484,-0.7319247)
\psdots[linecolor=black, dotstyle=o, dotsize=0.49939746](21.5,-1.9323126)
\psline[linecolor=colour0, linewidth=0.05](21.0,-3.7323127)(21.5,-4.4323125)(22.0,-3.7323127)(22.0,-3.7323127)
\psline[linecolor=colour0, linewidth=0.05](21.5,-4.4323125)(21.5,-5.132313)(21.5,-5.132313)
\rput[bl](4.8,0.96768737){\LARGE $w_i$}
\rput[bl](6.3,-1.8323126){\LARGE $w_j$}
\rput[bl](12.2,-0.23231262){\LARGE $w_i$}
\rput(13.9,-0.13231263){\textcolor{colour0}{\LARGE $w_j$}}
\rput[l](21.9,-1.4323126){\textcolor{colour0}{\LARGE $w_i$}}
\rput[l](21.9,0.9676874){\textcolor{colour0}{\LARGE $w_j$}}
\rput(4.5,-3.2623127){\LARGE $i$}
\rput(6.0,-3.3323126){\textcolor{colour0}{\LARGE $j$}}
\rput(11.9,-3.2623127){\textcolor{colour0}{\LARGE $i$}}
\rput(13.4,-3.3323126){\textcolor{colour0}{\LARGE $j$}}
\rput(21.0,-3.2623127){\textcolor{colour0}{\LARGE $i$}}
\rput(22.0,-3.3323126){\textcolor{colour0}{\LARGE $j$}}
\end{pspicture}
}
\caption{\label{ide}The cases $i <_0 j$ (two left pictures) and $i <_1 j$
  (rightmost picture).}  
\end{figure}

A morphism $\phi : (\frO,\delta) \to (\frN, \gamma)$ in $\Tam_2^\bbN$ is,
by definition of the Grothendieck construction, an $|\frN|\!+\!1$-tuple
\begin{equation}
\label{eq:4}
\big(\sigma,(\frO_1,\delta_1),\ldots,(\frO_k,\delta_k)\big)
\end{equation}
where $\sigma:\frO\to \frN$ is a
morphism of $2$-ordinals, $\frO_i =  \sigma^{-1}(i)$ and $(\frO_i,\delta_i)\in
\opTam^\bbN_2(\frO_i)$, $i \in |\frN|$, 
are such that 
\begin{equation}
\label{eq:38}
m\big((\frO_1,\delta_1),\ldots,(\frO_k,\delta_k); (\frN, \gamma)\big)
=  (\frO,\delta),
\end{equation}
where $m$ is the multiplication in the pruned $2$-operad
$\opTam^\bbN_2$.
  
For any morphism $f : \beta \to \alpha$ in $\LTr$, the 
functor $\Omega: \LTr \to \OmegaN$ induces a map of $2$-trees $\Omega(f): \Omega(\beta)\to
\Omega(\alpha)$ so, denoting $\frO: =  p\big(\Omega(\beta)\big)$ and $\frN: =
p\big(\Omega(\alpha)\big)$,  we have the induced map
\begin{equation}
\label{eq:5}
\sigma: \frO \lra \frN , \ \sigma: =  p(\Omega(f))
\end{equation}
of $\bbN$-colored $2$-ordinals.

To define the functor $u: \LTr \to \Tam^\bbN_2$ on morphisms, it
suffices to specify it on elementary morphisms listed in
Section~\ref{LTr}.  Let $f : \beta \to \alpha$ be an elementary
morphism of Type~1 collapsing two consecutive levels.    Each \bily-vertex $w$ of
$\alpha$ has the fiber $f^{-1}(w)\in \LTr$. Such a $w$
also corresponds to a unique element $i = i(w)$ of the $2$-ordinal
presented $\frN =p(\Omega(\alpha))$. So, we associate to $w$ the 
element $(\frO_i,\delta_i) :=
\big(\sigma^{-1}(i),\overline{f^{-1}(w)}\big)
\in \Tam^\bbN_2\big(\sigma^{-1}(i)\big)$. It is easy to
check that then~(\ref{eq:38}) with  $(\frO_i,\delta_i)$ as above,
$(\frO,\delta) := u(\beta)$ and $(\frN,\gamma) := u(\alpha)$ is satisfied,
so we constructed a morphism $u(f) : u(\beta) \to u(\alpha)$ in $ \Tam^\bbN_2$.

If $f: \beta\to \alpha$ is an elementary morphism of Type~2 that
replaces a level of \cerny-vertices by a level of \tlusty-vertices,
then clearly $u(\beta) = u(\alpha)$ and we define $u(f)$ to be the
identity morphism.

Assume that $f: \beta\to \alpha$ replaces a 
\tlusty-vertex of $\beta$  by a \bily-vertex $w$ and denote by $e =
e(w)$ 
the corresponding `exceptional' element of the $2$-ordinal $\frN$. To describe the
fibers of $f: \beta\to \alpha$ and of the induced map~(\ref{eq:5}) denote,
for $c\in \bbN$, by ${\boldsymbol 1}^c \in \OrN$ the terminal $2$-ordinal 
whose input and output
colors equal $c$, ${\boldsymbol 0}^c_2 \in \OrN$ the initial $2$-ordinal whose
output color is $c$, and by $\bily^c$  (resp.~$\tlusty^c$)
the \bily-  (resp.~\tlusty-) corolla of arity $c$. With this notation,
\[
f^{-1}(u) = \cases{\bily^{c_v}}{if $v \not= w$, and}{\tlusty^{c_v}}{if $v = w$,} 
\]
where $c_v$ is the arity of a \bily-vertex $v$ of $\alpha$. The
induced map~(\ref{eq:5}) is an inclusion and
\[
\sigma^{-1}(u) = 
\cases{{\boldsymbol 1}_2^{c_i}}{if $i \not= e$, and}{{\boldsymbol 0}_2^{c_i}}{if $i = e$,} 
\]
where $c_i \in \bbN$ denotes the color of $i \in \EuS$. We define
$u(f) : u(\beta) \to u(\alpha)$ by taking in~(\ref{eq:4})
\[
(\frO_i,\delta_i) :=
\cases{({\boldsymbol 1}^{c_i}_2\, , \hskip -.2em \bily^{c_i})}
{if $i \not= e$, and}
{ ({\boldsymbol 0}^{c_i}_2\, , \hskip -.2em\tlusty^{c_i})}{if  $i = e$.}
\]

The discussion
of Type~3 elementary morphisms $f: \beta\to \alpha$ is analogous. The
exceptional fibers of $f$ are now the exceptional trees $\except$ with
no vertices, while the exceptional fibers of
the induced map $\sigma : \frO \to \frN$ are the initial of
$2$-ordinals ${\boldsymbol1}_2$ with the output color $1$
This completes the definition of the functor $u$.

\section{Action of the Tamarkin operad}
\label{sec:acti-tamark-tsyg}

The aim of this section is to prove the following statement in which
$\opTm^\bbN_2$ is the Tamarkin operad recalled in
Definition~\ref{musim_objet_Ryn}.  

\begin{theorem}
\label{lode_na_Rynu}
For any multiplicative $1$-operad $\scA$ in a duoidal category one has a
natural action
\begin{equation}
\label{eq:23}
\opTm^\bbN_2 \to \End_\scA^{\Omega_2^\bbN}.
\end{equation}
\end{theorem}

Proposition~\ref{sec:discr-fibr-betw} associates to the operad ${\opTam^\bbN_2}$
the operadic category $\Tam^\bbN_2$, together
with a discrete operadic fibration
\[
s:\Tam^{\bbN}_2\to \OrN.
\]
By~(\ref{eq:22}), ${\opTam^\bbN_2} = s_!(\bfone^{\Tam^\bbN_2})$,  therefore
$\opTm^\bbN_2 = p^*({\opTam^\bbN_2}) = p^*\big(s_!(\bfone^{\Tam^\bbN_2})\big)$, 
so the map in~(\ref{eq:23}) is the same as a morphism
\begin{equation}
\label{eq:25}
p^*\big(s_!(\bfone^{\Tam^\bbN_2})\big) \to \End_\scA^{\Omega_2^\bbN}.
\end{equation}

On the other hand, define the operadic category $\Tm^\bbN_2$ 
using Proposition~\ref{pullback-create-pullbacks}  as the 
pullback of operadic categories:
\begin{equation}
  \label{eq:11}
  \xymatrix@C = +3em{
    \Tm_2^\bbN \ar[d]_{t} 
    \ar[r]^{r} 
    &\Tam_2^\bbN \ar[d]^s
    \\ 
    \Omega_2^\bbN
    \ar[r]^{p}
    &\, \OrN.
  }
\end{equation}
The induced functors
of the associated categories of operads enjoy the Beck-Chevalley property by
Proposition~\ref{Beck-Chevalley}, so
\begin{equation}
  \label{eq:patek}
p^*\big(s_!(\bfone^{\Tam^\bbN_2})\big)
\cong t_!\big(r^*(\bfone^{\Tam^\bbN_2})\big).
\end{equation}
The map in~(\ref{eq:25}) is thus the
same as a morphism
\[
\label{eq:29}
t_!\big(r^*(\bfone^{\Tam^\bbN_2})\big) \to \End_\scA^{\Omega_2^\bbN}
\]
which is, by adjunction, the same as an operad morphism
\[
r^*(\bfone^{\Tam^\bbN_2})\to t^*(\End_\scA^{\OmegaN}).
\]
Since clearly
$r^*(\bfone^{\Tam^\bbN_2}) = \bfone^{\Tm^\bbN_2}$, 
Theorem~\ref{lode_na_Rynu} will be proved if 
we construct a natural morphism.
\begin{equation}
\label{eq:26}
\bfone^{\Tm^\bbN_2}\to t^*(\End_\scA^{\OmegaN}).
\end{equation}

\begin{remark}
\label{pristi_tyden_posledni}
Notice that $r^*(\bfone^{\Tam^\bbN_2}) = \bfone^{\Tm^\bbN_2}$ together
with~(\ref{eq:patek}) implies that $\opTm^\bbN_2 = t_!(
\bfone^{\Tm^\bbN_2})$, so  $\opTm^\bbN_2$ is the result of the
application of the functor inverse to the
Grothendieck construction to the discrete fibration  $\Tm^\bbN_2\to \OmegaN$.   
\end{remark}

It is obvious that the diagram
\[
\mbox { \xymatrix@C = +3em{ \LTr \ar[d]_{\Omega} \ar[r]^{u} &\ar[d]^s
    \Tam_2^\bbN
    \\
    \OmegaN \ar[r]^{p} &\OrN } }
\]
in which $\Omega: \LTr \to \OmegaN$ is the functor~(\ref{eq:21}) and $u :
\LTr \to \Tam_2^\bbN$ the functor~(\ref{eq:functoru}) commutes, so we
have the induced operadic functor \hbox{$w: \LTr\to \Tm_2^\bbN$} to
the pullback of~(\ref{eq:11}) as in the commutative diagram
\begin{equation}
\label{eq:27}
\xymatrix@C = +3em@R = 1em{\LTr  \ar[dr]^w  \ar@/_1em/[ddr]^{\Omega} \ar@/^1em/[rrd]^{u}& &
\\
& \Tm_2^\bbN  \ar[d]_t \ar[r]^r  &\ar[d]^s
    \Tam_2^\bbN
    \\
  &  \OmegaN \ar[r]^{p} &\ \OrN .}
\end{equation}
Recall that we already constructed a canonical 
morphism~(\ref{eq:33}) of $\LTr$-operads 
\[
\Xi :\bfone^{\sLTr}\to
\Omega^*(\End_\scA^{\OmegaN}).
\]
By~(\ref{eq:27}), $tw = \Omega$, so
$\Omega^*(\End_\scA^{\OmegaN}) = w^*(t^*(\End_\scA^{\OmegaN}))$. Noticing the
isomorphism  $\bfone^{\sLTr} = w^*(\bfone^{\Tm_2^\bbN})$, we
see that we therefore also have a natural morphism of $\LTr$-operads
\begin{equation}
\label{eq:28}
\Upsilon :   w^*(\bfone^{\Tm_2^\bbN})
\longrightarrow w^*\big(t^*(\End_\scA^{\OmegaN})\big).
\end{equation}
The requisite map in~(\ref{eq:26}) will be constructed by `inverting
$w^*$' in~(\ref{eq:28}), using:  

\begin{lemma}
\label{pelmene}
Let $F : \ttO \to \ttP$ be an operadic functor and $\calP,\calO \in \OpV
  \ttP$. 
Assume that $F$ is surjective on objects. 
Suppose we are also given a morphism $\varsigma : F^*(\calP) \to F^*(\calO)$ of
 $\ttO$-operads such that, for arbitrary $Q',Q'' \in \ttO$ such that
 $F(Q') = F(Q'')$,
\begin{equation}
\label{eq:30}
\varsigma_{Q'} : F^*\calP(Q') \to F^*\calO(Q') \ \mbox { equals }\
\varsigma_{Q''} : F^*\calP(Q'') \to F^*\calO(Q''). 
\end{equation}
Then there exist a unique morphism $\rho : \calP \to \calO$ of\/ $\ttP$-operads
satisfying $\varsigma = F^*(\rho)$.
\end{lemma}

\begin{proof}[Proof Lemma~\ref{pelmene}]
  For $T \in \ttP$ choose $Q \in \ttO$ such that $T = F(Q)$. Since
  $F^*\calP(Q) = \calP(T)$ and $F^*\calO(Q) = \calO(T)$, we may 
define \hbox{$\rho_T :
    \calP(T) \to \calO(T)$} by $\rho_T := \varsigma_Q$. We leave as an
  exercise to verify that the collection $\rho = \{\rho_T\}$
  is a well-defined morphism of operads.
\end{proof}

We wish to apply the lemma to the situation when $\ttO = \LTr$, $\ttP =
\Tm_2^\bbN$, $F$ is the functor $w$ in~(\ref{eq:27}), $\calP =
\bfone^{\Tm^\bbN_2}$, $\calO = t^*(\End_\scA^{\OmegaN})$ and $\varsigma$ the
morphism $\varepsilon$ of~(\ref{eq:28}). Since $w$ is clearly
surjective on objects, we only need
to verify~(\ref{eq:30}). In this particular case it means that, 
given $\beta',\beta'' \in \LTr$ such that $w(\beta')
= w(\beta'')$, $\Upsilon_{\beta'} = \Upsilon_{\beta''}$. 
Recalling again that   $\bfone^{\sLTr} = w^*(\bfone^{\Tm_2^\bbN})$ and
$tw = \Omega$, we easily see that it is enough to prove:

\begin{lemma}
\label{golovokruzenie}
Let $\Xi = \{\Xi_\beta\} : \bfone^{\sLTr}\to
\Omega^*(\End_\scA^{\OmegaN})$ be the composite~(\ref{eq:33}).
If $\beta',\beta'' \in \LTr$ are such that $w(\beta') = w(\beta'')$, then
$\Psi_{\beta'} =  \Psi_{\beta''}$.
\end{lemma}

As the first step in proving Lemma~\ref{golovokruzenie} we
characterize, in Lemma~\ref{elementary-moves} below,  
pairs $\beta',\beta'' \in \LTr$ having the same $w$-image in
$\Tm_2^\bbN$. For this, the following 
alternative description of objects of the categories $
\Tam^\bbN_2$ and $\Tm_2^\bbN$ will be useful.

\begin{lemma}
\label{tam-as-tree}
Objects of\/ $\Tam^\bbN_2$ can be described as the
isomorphism classes of planar rooted trees $\zeta$
with white vertices \bily, vertical vertices \tlusty\ and horizontal
vertices \cerny. While
\bily-vertices have arbitrary arities $\geq 0$, \cerny-vertices have
either arity $\geq 2$ or $0$, and all \tlusty-vertices
are of arity~$1$.  

We moreover require that $\zeta$ has no internal edge connecting two
\cerny-vertices and no internal edge starting from a \cerny-vertex and
ending in a \tlusty-vertex, i.e.~the following edges
\begin{equation}
\label{eq:31}
\psscalebox{.5 .5}
{
\begin{pspicture}(0,-0.6230486)(4.146097,0.6230486)
\pscustom[linecolor=black, linewidth=0.05]
{
\newpath
\moveto(8.9,-3.9769514)
}
\pscustom[linecolor=black, linewidth=0.05]
{
\newpath
\moveto(21.5,-2.6769514)
}

\psdots[linecolor=black, dotsize=0.49939746](3.9,-0.37695137)
\psdots[linecolor=black, dotsize=0.49939746](1.1,-0.37695137)
\psline[linecolor=black, linewidth=0.05, arrowsize=0.05291666666666672cm 6.0,arrowlength=2.0,arrowinset=0.0]{<-}(1.2,-0.37656346)(3.6974843,-0.37656346)
\end{pspicture}
}
\ \ \mbox { or } 
\psscalebox{.5 .5}
{
\begin{pspicture}(0,-0.6257164)(4.146097,0.6257164)
\pscustom[linecolor=black, linewidth=0.05]
{
\newpath
\moveto(8.9,-3.9742837)
}
\pscustom[linecolor=black, linewidth=0.05]
{
\newpath
\moveto(21.5,-2.6742837)
}

\psdots[linecolor=black, dotsize=0.49939746](3.9,-0.37428364)
\psline[linecolor=black, linewidth=0.05, arrowsize=0.05291666666666672cm 6.0,arrowlength=2.0,arrowinset=0.0]{<-}(1.2,-0.37389573)(3.6974843,-0.37389573)
\psdots[linecolor=black, fillstyle=solid, dotstyle=o, dotsize=0.49939746](1.0,-0.37428364)
\psdots[linecolor=black, dotstyle=otimes, dotsize=0.49939746](1.0,-0.37428364)
\end{pspicture}
}
\end{equation}
are not allowed. Finally, \bily- and \tlusty-vertices
are  lined up in levels such that each level contains at least one
\bily-vertex. 

Objects of\/ $\Tm^\bbN_2$ can similarly be identified with the
isomorphisms classes of trees as above, but this time allowing also
levels consisting solely of \tlusty-vertices.
\end{lemma}

\begin{proof}
Let us denote
provisionally by $\pro$ the set of isomorphism classes of trees which,
according to the lemma, should describe objects of $\Tam^\bbN_2$, and let
$\pr$ be similarly related to~$\Tm^\bbN_2$. We are going to construct
two couples of mutually inverse maps,
\[
{\vbox to 1.6em{\vss \hbox to
        20em{
$\xymatrix{\pro \ar@/^.8em/[r]^\phi &  \ar@/^.8em/[l]_\psi
  \Tam^\bbN_2}
\ \mbox { and }
\xymatrix{\pr \ar@/^.8em/[r]^\varrho &  \ar@/^.8em/[l]_\varsigma
  \Tm^\bbN_2}.$}}}
\]
While the definitions of $\phi : \pro \to \Tam^\bbN_2$ and $\varrho : 
\pr \to \Tm^\bbN_2$ are very simple, 
our constructions of their
inverses will involve intuitive geometric arguments. A formal
combinatorial construction should use a straightforward but
lengthy induction on the number of vertices of the trees involved. We
leave it to the interested reader.

Before we begin, notice that each planar rooted tree
$\omega$ with levels of \bily- and \tlusty-vertices determines
an $\bbN$-colored $2$-tree $\Omega(\omega) \in \OmegaN$ by same
procedure as described for the leveled trees of $\LTr$ in
Section~\ref{LTr}. That is, we organize \bily- and
\tlusty-vertices of $\omega$ to table~(\ref{eq:3}); the $1$-truncation
of $\Omega(\omega)$ will then be the set $\{\rada 1t\}$ and its
$2$-leaves the same as in~(\ref{Dnes_snad_Jarca_zalije_kyticky}). The
$\bbN$-coloring of the $2$-leaves of $\Omega(\omega)$ is given, as
always in this article, by the arities of the corresponding
\bily-vertices. It is clear that $\Omega(\omega)$ is pruned if and
only if $\omega$ does not have levels consisting solely of
\tlusty-vertices.

Let us describe $\phi : \pro \to \Tam^\bbN_2$. 
As in Section~\ref{sec:tamark-tsyg-oper},
the level structure of $\zeta \in \pro$ induces the lexicographic order on the
set $\Wh(\zeta)$ of its white vertices. We label the \bily-vertices
of $\zeta$ by $\{\rada 1k\}$ accordingly, remove the levels and denote
the resulting tree by $\overline \zeta$.

Since, by assumption, all levels of $\zeta$ contain at
least one white vertex, the $2$-tree $\Omega(\zeta)$ is pruned and
can therefore be interpreted as a $2$-ordinal. Then, by construction, 
$\phi(\zeta) := \big(\Omega(\zeta),\bar \zeta\big)$ is an $\OrN$-labelled tree.
It is moreover clear that the $2$-ordinal $\Omega(\zeta)$ dominates the
canonical complementary order on $\Wh(\overline \zeta)$, so in fact
$\phi(\zeta) \in \Tam^\bbN_2$.
    
On the other hand, take $(\frO,\delta) \in \Tam^\bbN_2$ and define
$\zeta = \psi(\frO,\delta) \in \pro$ as follows.  First, organize the
\bily-vertices of $\delta$ to levels such that $\Omega(\delta)
=\frO$. 
The domination condition for $(\frO,\delta)$ guarantees that
it is possible. Then move the \cerny-vertices of $\delta$ so close
to the root that none of the edges \cernybily\ 
intersect a level line and that none  of the \cerny-vertices lies on a
level line.
Finally, introduce unary \tlusty-vertices at the intersection points
of the level lines with the edges of $\delta$. 
Then $\phi :
\pro \to \Tam^\bbN_2$ and $\psi : \Tam^\bbN_2 \to \pro$ are obviously
mutual inverses.

For instance, if $\frO$ is the $\bbN$-colored $2$-ordinal represented
by the pruned  $\bbN$-colored $2$-tree in the right side of 
Figure~\ref{prun-unprun}
\begin{figure}[t]
  \centering
\psscalebox{.5 .5} 
{
\begin{pspicture}(0,-1)(6.1622066,1.57)
\psline[linecolor=black, linewidth=0.05](3.1975353,-1.07)(5.2410135,-0.07)(4.62797,0.93)(4.62797,0.93)
\psline[linecolor=black, linewidth=0.05](5.2410135,-0.07)(5.854057,0.93)(5.854057,0.93)
\psline[linecolor=black, linewidth=0.05](1.154057,0.93)(1.154057,-0.07)(3.1975353,-1.07)(3.1975353,-1.07)
\psline[linecolor=black, linewidth=0.05](3.1975353,-0.07)(3.1975353,-1.07)(3.1975353,-1.07)
\psline[linecolor=black, linewidth=0.05](4.2192745,0.93)(3.1975353,-0.07)(2.175796,0.93)(2.175796,0.93)
\psline[linecolor=black, linewidth=0.05](3.504057,0.93)(3.1975353,-0.07)(2.8910136,0.93)(2.8910136,0.93)
\rput[b](1.154057,1.23){\Large 2}
\rput[b](2.175796,1.23){\Large 3}
\rput[b](2.8910136,1.23){\Large 0}
\rput[b](3.504057,1.23){\Large 0}
\rput[b](4.2192745,1.23){\Large 2}
\rput[b](4.62797,1.23){\Large 1}
\rput[b](5.854057,1.23){\Large 1}
\rput[b](3.2236223,-1.67){\Large 10}
\psline[linecolor=black, linewidth=0.05](3.2192085,-1.07)(2.0086024,-0.07)(2.0086024,-0.07)
\psline[linecolor=black, linewidth=0.05](3.2192085,-1.07)(2.5070872,-0.07)(2.5070872,-0.07)
\psline[linecolor=black, linewidth=0.05](3.2192085,-1.07)(3.7889054,-0.07)(3.7889054,-0.07)
\psline[linecolor=black, linewidth=0.05](3.254057,-1.07)(6.154057,-0.07)(6.154057,-0.07)
\psline[linecolor=black, linewidth=0.05](3.154057,-1.07)(0.15405701,-0.07)(0.15405701,-0.07)
\end{pspicture}
}
\raisebox{1.5em}{
\hskip 2em \xymatrix@1@C=5em{\ar@{|->}[r]^{\mbox {\scriptsize pruning}}&}\hskip 2em}
\psscalebox{.5 .5} 
{
\begin{pspicture}(0,-1)(4.79,1.57)
\psline[linecolor=black, linewidth=0.05](2.105,-1.07)(4.105,-0.07)(3.505,0.93)(3.505,0.93)
\psline[linecolor=black, linewidth=0.05](4.105,-0.07)(4.705,0.93)(4.705,0.93)
\psline[linecolor=black, linewidth=0.05](0.105,0.93)(0.105,-0.07)(2.105,-1.07)(2.105,-1.07)
\psline[linecolor=black, linewidth=0.05](2.105,-0.07)(2.105,-1.07)(2.105,-1.07)
\psline[linecolor=black, linewidth=0.05](3.105,0.93)(2.105,-0.07)(1.105,0.93)(1.105,0.93)
\psline[linecolor=black, linewidth=0.05](2.405,0.93)(2.105,-0.07)(1.805,0.93)(1.805,0.93)
\rput[b](0.105,1.23){\Large 2}
\rput[b](1.105,1.23){\Large 3}
\rput[b](1.805,1.23){\Large 0}
\rput[b](2.405,1.23){\Large 0}
\rput[b](3.105,1.23){\Large 2}
\rput[b](3.505,1.23){\Large 1}
\rput[b](4.705,1.23){\Large 1}
\rput[b](2.105,-1.67){\Large 10}
\end{pspicture}
}
\caption{A $2$-tree $\frT \in \OmegaN$ and its pruning $\frO =p(\frT) \in \OrN$.}
\label{prun-unprun}
\end{figure}
and $\delta$ the tree in Figure~\ref{tam-tree}, then $(\frO,\delta)
\in \Tam^\bbN_2$, 
and $\psi(\frO,\delta) \in \pro$ is the tree in Figure~\ref{fig:6}.

\begin{figure}[t]
  \centering
\psscalebox{.5 .5}
{
\begin{pspicture}(5.5,-5.555)(20.005621,3)
\definecolor{colour0}{rgb}{0.03137255,0.003921569,0.003921569}
\psline[linecolor=black, linewidth=0.02](2.597499,1.5556207)(2.597499,1.5556207)
\pscustom[linecolor=black, linewidth=0.05]
{
\newpath
\moveto(8.9,0.955)
}
\pscustom[linecolor=black, linewidth=0.05]
{
\newpath
\moveto(21.5,2.255)
}

\pscustom[linecolor=black, linewidth=0.05]
{
\newpath
\moveto(12.013208,-4.0448136)
\lineto(12.098667,-4.191884)
\curveto(12.1413965,-4.2654195)(12.245169,-4.412489)(12.30621,-4.4860244)
\curveto(12.367251,-4.5595593)(12.587003,-4.7066298)(12.745713,-4.7801647)
\curveto(12.904423,-4.8537)(13.282884,-4.9566493)(13.942139,-5.0448914)
}
\pscustom[linecolor=black, linewidth=0.05]
{
\newpath
\moveto(12.013208,-4.0448136)
\lineto(12.36449,-3.794801)
\curveto(12.540131,-3.669795)(12.966687,-3.4197826)(13.217603,-3.2947767)
\curveto(13.468518,-3.1697705)(14.371814,-2.9197571)(15.024195,-2.794751)
\curveto(15.676575,-2.669745)(17.23225,-2.4947364)(19.942139,-2.3447285)
}
\pscustom[linecolor=black, linewidth=0.05]
{
\newpath
\moveto(12.113208,0.7555896)
\lineto(12.273983,0.6526361)
\curveto(12.35437,0.60115904)(12.549597,0.49820557)(12.664437,0.44672853)
\curveto(12.779278,0.39525145)(13.192699,0.2922986)(13.4912815,0.24082154)
\curveto(13.7898655,0.18934448)(14.50187,0.11727661)(15.742138,0.055504285)
}
\psline[linecolor=black, linewidth=0.05](13.113208,-4.0448136)(13.113208,-4.0448136)
\psline[linecolor=black, linewidth=0.05](13.112857,-4.0448136)(16.771069,-4.0448136)(13.1,-4.045)
\pscustom[linecolor=black, linewidth=0.05]
{
\newpath
\moveto(16.871069,-0.036818847)
\lineto(17.148315,0.25690857)
\curveto(17.286938,0.40377197)(17.920647,0.6974997)(18.415731,0.8443634)
\curveto(18.910816,0.9912271)(19.554424,1.1674637)(20.0,1.2555819)
}
\pscustom[linecolor=black, linewidth=0.05]
{
\newpath
\moveto(17.000603,-4.0448136)
\lineto(17.266373,-3.8857105)
\curveto(17.399256,-3.8061585)(18.00673,-3.6470556)(18.481318,-3.5675037)
\curveto(18.955906,-3.4879522)(19.57287,-3.3924901)(20.0,-3.3447595)
}
\pscustom[linecolor=black, linewidth=0.05]
{
\newpath
\moveto(16.947515,-0.036818847)
\lineto(17.217989,0.08556762)
\curveto(17.353226,0.14676087)(17.97145,0.26914734)(18.45444,0.33034056)
\curveto(18.937428,0.39153382)(19.56531,0.46496522)(20.0,0.5016814)
}
\pscustom[linecolor=black, linewidth=0.05]
{
\newpath
\moveto(14.542139,1.8556671)
\lineto(15.023184,1.7647439)
\curveto(15.263706,1.7192823)(16.363235,1.6283593)(17.222244,1.5828977)
\curveto(18.081253,1.5374359)(19.197962,1.482882)(19.97107,1.4556051)
}
\pscustom[linecolor=black, linewidth=0.05]
{
\newpath
\moveto(8.563345,-0.64472044)
\lineto(8.84618,-0.3037976)
\curveto(8.987598,-0.13333619)(9.63408,0.2075879)(10.139143,0.3780493)
\curveto(10.644207,0.54851073)(11.300789,0.7530646)(11.755346,0.8553414)
}
\pscustom[linecolor=black, linewidth=0.05]
{
\newpath
\moveto(17.000603,-4.0448136)
\lineto(17.266373,-4.340291)
\curveto(17.399256,-4.48803)(18.00673,-4.783507)(18.481318,-4.9312463)
\curveto(18.955906,-5.0789843)(19.57287,-5.256272)(20.0,-5.3449144)
}
\pscustom[linecolor=black, linewidth=0.05]
{
\newpath
\moveto(16.871069,-0.036818847)
\lineto(17.148315,-0.37950075)
\curveto(17.286938,-0.5508423)(17.920647,-0.89352417)(18.415731,-1.0648651)
\curveto(18.910816,-1.236206)(19.554424,-1.4418151)(20.0,-1.5446198)
}
\pscustom[linecolor=black, linewidth=0.05]
{
\newpath
\moveto(16.947515,-0.036818847)
\lineto(17.217989,-0.15920532)
\curveto(17.353226,-0.22039856)(17.97145,-0.34278503)(18.45444,-0.40397826)
\curveto(18.937428,-0.46517152)(19.56531,-0.53860354)(20.0,-0.57531923)
}
\psline[linecolor=black, linewidth=0.05](17.000603,-4.0448136)(18.4,-4.344837)(20.0,-4.545)
\psline[linecolor=black, linewidth=0.05, arrowsize=0.05291666666666672cm 6.0,arrowlength=2.0,arrowinset=0.0]{<-}(3.6982164,-1.9446121)(5.725617,-1.9446121)
\psdots[linecolor=black, fillstyle=solid, dotstyle=o, dotsize=0.49939746](12.0,-2.745)
\psdots[linecolor=black, dotsize=0.49939746](14.4,1.955)
\psdots[linecolor=black, dotsize=0.49939746](17.0,-0.045)
\pscustom[linecolor=black, linewidth=0.05]
{
\newpath
\moveto(9.842138,-3.644333)
\lineto(10.013056,-3.462497)
\curveto(10.098515,-3.371579)(10.489184,-3.1897438)(10.794395,-3.0988257)
\curveto(11.0996065,-3.0079076)(11.49638,-2.898806)(11.77107,-2.8442552)
}
\pscustom[linecolor=black, linewidth=0.05]
{
\newpath
\moveto(8.736145,-0.7443249)
\lineto(9.010119,-0.67612916)
\curveto(9.147107,-0.64203125)(9.7733345,-0.57383543)(10.262576,-0.5397375)
\curveto(10.751817,-0.5056397)(11.38783,-0.46472168)(11.828146,-0.44426295)
}
\psdots[linecolor=black, fillstyle=solid, dotstyle=o, dotsize=0.49939746](12.0,-0.445)
\psdots[linecolor=black, fillstyle=solid, dotstyle=o, dotsize=0.49939746](15.8,0.055)
\pscustom[linecolor=black, linewidth=0.05]
{
\newpath
\moveto(5.9600716,-1.944325)
\lineto(6.285242,-2.3761523)
\curveto(6.447828,-2.5920653)(7.1910753,-3.0238929)(7.771737,-3.239806)
\curveto(8.352399,-3.4557197)(9.10726,-3.7148156)(9.629855,-3.844364)
}
\pscustom[linecolor=black, linewidth=0.05]
{
\newpath
\moveto(5.987272,-1.8447205)
\lineto(6.216943,-1.5947064)
\curveto(6.3317785,-1.4696997)(6.856741,-1.2196851)(7.266867,-1.0946784)
\curveto(7.6769934,-0.9696716)(8.210158,-0.8196631)(8.579272,-0.74465865)
}
\pscustom[linecolor=black, linewidth=0.05]
{
\newpath
\moveto(9.971069,-3.844325)
\lineto(10.248316,-3.8896449)
\curveto(10.386938,-3.9123046)(11.020647,-3.9576244)(11.515731,-3.9802845)
\curveto(12.010816,-4.0029445)(12.654424,-4.030136)(13.1,-4.0437317)
}
\pscustom[linecolor=black, linewidth=0.05]
{
\newpath
\moveto(12.1421385,1.0556672)
\lineto(12.339639,1.2602301)
\curveto(12.438389,1.3625116)(12.889817,1.5670753)(13.2424965,1.6693567)
\curveto(13.595176,1.7716382)(14.053658,1.8943759)(14.371069,1.9557447)
}
\pscustom[linecolor=black, linewidth=0.05]
{
\newpath
\moveto(14.542139,2.0556672)
\lineto(15.023184,2.2147758)
\curveto(15.263706,2.29433)(16.363235,2.4534388)(17.222244,2.5329928)
\curveto(18.081253,2.612547)(19.197962,2.7080123)(19.97107,2.7557447)
}
\psdots[linecolor=black, dotsize=0.49939746](8.6,-0.745)
\psdots[linecolor=black, fillstyle=solid, dotstyle=o, dotsize=0.49939746](12.0,0.955)
\psline[linecolor=colour0, linewidth=0.05](16.0,0.055)(17.1,-0.045)
\psdots[linecolor=black, dotsize=0.49939746](9.8,-3.845)
\psdots[linecolor=black, dotsize=0.49939746](13.8,-4.945)
\psdots[linecolor=black, dotsize=0.49939746](17.0,-4.045)
\psdots[linecolor=black, fillstyle=solid, dotstyle=o, dotsize=0.49939746](12.0,-4.045)
\psdots[linecolor=black, fillstyle=solid, dotstyle=o, dotsize=0.49939746](15.8,2.355)
\psline[linecolor=black, linewidth=0.02](5.9,2.755)(5.9,-5.545)(5.9,-5.545)
\psline[linecolor=black, linewidth=0.02](12.0,2.755)(12.0,-5.545)(12.0,-5.545)
\psline[linecolor=black, linewidth=0.02](15.8,2.755)(15.8,-5.545)(15.8,-5.545)
\psdots[linecolor=black, fillstyle=solid, dotstyle=o, dotsize=0.49939746](12.0,0.955)
\psdots[linecolor=black, fillstyle=solid, dotstyle=o, dotsize=0.49939746](12.0,-2.745)
\psdots[linecolor=black, fillstyle=solid, dotstyle=o, dotsize=0.49939746](15.8,-4.045)
\psdots[linecolor=black, fillstyle=solid, dotstyle=o, dotsize=0.49939746](12.0,-4.045)
\psdots[linecolor=black, fillstyle=solid, dotstyle=o, dotsize=0.49939746](12.0,-0.445)
\psdots[linecolor=black, fillstyle=solid, dotstyle=o, dotsize=0.49939746](15.8,0.055)
\psdots[linecolor=black, fillstyle=solid, dotstyle=o, dotsize=0.49939746](15.8,-2.645)
\psdots[linecolor=black, fillstyle=solid, dotstyle=o, dotsize=0.49939746](15.8,2.355)
\psdots[linecolor=black, fillstyle=solid, dotstyle=o, dotsize=0.49939746](15.8,1.655)
\psdots[linecolor=black, dotstyle=otimes, dotsize=0.49939746](15.8,-2.645)
\psdots[linecolor=black, dotstyle=otimes, dotsize=0.49939746](15.8,2.355)
\psdots[linecolor=black, dotstyle=otimes, dotsize=0.49939746](15.8,1.655)
\psdots[linecolor=black, fillstyle=solid, dotstyle=o, dotsize=0.49939746](5.9,-1.945)
\end{pspicture}
}
\caption{\label{fig:6}
The tree $\zeta = \psi(\frO,\delta) \in \pro$.}
\end{figure}

Let us describe $\varrho : 
\pr \to \Tm^\bbN_2$. Notice that elements of $\Tm^\bbN_2$ are, by the definition
via the pullback in~(\ref{eq:11}),  couples
$(\frT,\delta)$ such that $\frT \in \OmegaN$ and $\big(p(\frT),\delta)\big) \in
\Tam^\bbN_2$, where $p(\frT) \in \OrN$ is the pruning 
of the $\bbN$-colored $2$-tree $\frT$. We can call couples
$(\frT,\delta) \in \Tm^\bbN_2$ {\em $\OmegaN$-labelled trees\/}.

For $\xi \in \pr$ denote by $p(\xi)$ be the tree obtained from $\xi$ by
removing levels all consisting only of \tlusty-vertices. Clearly
$p(\xi)\in \pro$, so it makes sense to put
$\overline{\xi} := \phi\big(p(\xi)\big)$. It is then clear that the rule $\xi
\mapsto \big(\Omega(\xi),\phi(p(\xi))\big)$ defines a map $\varrho :\pr \to
\Tm^\bbN_2$.

Let us construct its inverse $\varsigma : \Tm^\bbN_2 \to \pr$.
Suppose that $(\frT,\delta) \in \Tm^\bbN_2$. As we already observed,
$(p(\frT),\delta) \in \Tam^\bbN_2$, so we may use the previous
construction and consider, as an intermediate step, the 
tree $\zeta :
=\psi\big(p(\frT),\delta\big) \in \pro$.  The tree $\xi =\varsigma(\frT,\delta)$
will be constructed by adding additional levels of
\tlusty-vertices of arity $1$ to $\zeta$ as follows.

Let $\tr_1(\frT) = \{\rada 1u\}$ and $\tr_1\big(p(\frT)\big) = \{\rada
1t\}$. If $t=u$, 
there is nothing to do as $\frT$ is pruned; in this case we
take $\xi := \zeta$. Assume therefore that $t <
u$.

Since $\tr_1\big(p(\frT)\big)$ is a  subset of $\tr_1(\frT)$, 
we have an inclusion $\liota : \{\rada
1t\} \hookrightarrow \{\rada 1u\}$. The complement
$\{\rada 1u\}\setminus {\rm Im}(\liota)$ is the disjoint union $S_1
\cup \cdots \cup S_k$ of non-empty intervals.
 For instance, for $\frT$
as Figure~\ref{prun-unprun},   $\tr_1(\frT) = \{\rada 18\}$,
$\tr_1\big(p(\frT)\big) = \{1,2,3\}$, $ {\rm Im}(\liota) = \{2,5,7\}$,~so
 \[
\{\rada 18\}\setminus {\rm Im}(\liota) = (1) \cup (3,4)\cup (6) \cup
(8).
\]

For $i$, $1\leq i \leq k$, such that $t\not\in S_i$, let $r_i := \liota^{-1}(\max(S_i) +
1)$. In the example above, $r_1 = 1$, $r_2 = 2$ and
$r_3 = 3$. For each such an $i$ we add to $\zeta$ ${\rm card}(S_i)$ new
levels consisting of~\tlusty-vertices of arity $1$ placed above the
$r_i$th level of $\zeta$ so close to it that all vertices of $\zeta$ above
this level are also above these newly introduced levels.
If $t \in S_i$\footnote{This may obviously happen only when
$i=k$.} we introduce ${\rm card}(S_i)$ new
levels of \tlusty-vertices of arity $1$ intersecting the input leaves of
$\zeta$.

We denote the resulting tree by $\xi$ and define
$\varsigma(\frT,\delta) := \xi$. We believe that
Figure~\ref{fig:9} makes the construction of $\beta$ out of
$\zeta$ obvious. 
\begin{figure}[t]
  \centering
\psscalebox{.5 .5}
{
\begin{pspicture}(5,-5.605)(20.025,4)
\definecolor{colour0}{rgb}{0.03137255,0.003921569,0.003921569}
\psline[linecolor=black, linewidth=0.02](2.597499,1.6056206)(2.597499,1.6056206)
\pscustom[linecolor=black, linewidth=0.05]
{
\newpath
\moveto(8.9,1.005)
}
\pscustom[linecolor=black, linewidth=0.05]
{
\newpath
\moveto(21.5,2.305)
}

\pscustom[linecolor=black, linewidth=0.05]
{
\newpath
\moveto(12.013208,-3.994814)
\lineto(12.094235,-4.1418843)
\curveto(12.134751,-4.2154193)(12.233142,-4.362489)(12.29102,-4.436024)
\curveto(12.348898,-4.509559)(12.557256,-4.6566296)(12.707738,-4.730165)
\curveto(12.85822,-4.8037)(13.21706,-4.906649)(13.842138,-4.994891)
}
\pscustom[linecolor=black, linewidth=0.05]
{
\newpath
\moveto(12.013208,-3.994814)
\lineto(12.36449,-3.759507)
\curveto(12.540131,-3.6418536)(12.966687,-3.4065473)(13.217603,-3.288894)
\curveto(13.468518,-3.1712403)(14.371814,-2.9359338)(15.024195,-2.81828)
\curveto(15.676575,-2.7006269)(17.23225,-2.5359125)(19.942139,-2.3947284)
}
\pscustom[linecolor=black, linewidth=0.05]
{
\newpath
\moveto(12.113208,0.8055896)
\lineto(12.273983,0.7026361)
\curveto(12.35437,0.65115905)(12.549597,0.54820555)(12.664437,0.4967285)
\curveto(12.779278,0.44525146)(13.192699,0.3422986)(13.4912815,0.29082152)
\curveto(13.7898655,0.23934448)(14.50187,0.1672766)(15.742138,0.10550429)
}
\psline[linecolor=black, linewidth=0.05](13.113208,-3.994814)(13.113208,-3.994814)
\psline[linecolor=black, linewidth=0.05](13.112857,-3.994814)(16.771069,-3.994814)(13.1,-3.995)
\pscustom[linecolor=black, linewidth=0.05]
{
\newpath
\moveto(16.871069,0.013181153)
\lineto(17.148315,0.2841815)
\curveto(17.286938,0.4196814)(17.920647,0.69068176)(18.415731,0.82618165)
\curveto(18.910816,0.96168154)(19.554424,1.1242816)(20.0,1.2055819)
}
\pscustom[linecolor=black, linewidth=0.05]
{
\newpath
\moveto(17.000603,-3.994814)
\lineto(17.266373,-3.8357105)
\curveto(17.399256,-3.7561584)(18.00673,-3.5970557)(18.481318,-3.5175037)
\curveto(18.955906,-3.4379523)(19.57287,-3.3424902)(20.0,-3.2947595)
}
\pscustom[linecolor=black, linewidth=0.05]
{
\newpath
\moveto(16.941778,0.029192505)
\lineto(17.2167,0.14872314)
\curveto(17.354162,0.20848876)(17.975323,0.3243866)(18.459023,0.3805188)
\curveto(18.942722,0.4366504)(19.57125,0.50350434)(20.005737,0.5356698)
}
\pscustom[linecolor=black, linewidth=0.05]
{
\newpath
\moveto(14.542139,1.9056671)
\lineto(15.023184,1.8374714)
\curveto(15.263706,1.8033735)(16.363235,1.7351776)(17.222244,1.7010797)
\curveto(18.081253,1.6669818)(19.197962,1.6260638)(19.97107,1.6056051)
}
\pscustom[linecolor=black, linewidth=0.05]
{
\newpath
\moveto(8.763346,-0.49472046)
\lineto(9.03732,-0.17652465)
\curveto(9.174307,-0.017426757)(9.800534,0.30076903)(10.289776,0.45986694)
\curveto(10.7790165,0.61896485)(11.41503,0.8098828)(11.855346,0.9053414)
}
\pscustom[linecolor=black, linewidth=0.05]
{
\newpath
\moveto(17.000603,-3.994814)
\lineto(17.266373,-4.2902913)
\curveto(17.399256,-4.43803)(18.00673,-4.733507)(18.481318,-4.8812466)
\curveto(18.955906,-5.0289845)(19.57287,-5.206272)(20.0,-5.2949147)
}
\pscustom[linecolor=black, linewidth=0.05]
{
\newpath
\moveto(16.871069,0.013181153)
\lineto(17.148315,-0.32950073)
\curveto(17.286938,-0.5008423)(17.920647,-0.84352416)(18.415731,-1.0148652)
\curveto(18.910816,-1.1862061)(19.554424,-1.3918152)(20.0,-1.4946198)
}
\pscustom[linecolor=black, linewidth=0.05]
{
\newpath
\moveto(16.947515,0.013181153)
\lineto(17.217989,-0.10920532)
\curveto(17.353226,-0.17039856)(17.97145,-0.29278505)(18.45444,-0.35397828)
\curveto(18.937428,-0.4151715)(19.56531,-0.4886035)(20.0,-0.5253193)
}
\psline[linecolor=black, linewidth=0.05](17.000603,-3.994814)(20.0,-4.394783)(20.0,-4.395)
\psline[linecolor=black, linewidth=0.05, arrowsize=0.05291666666666672cm 6.0,arrowlength=2.0,arrowinset=0.0]{<-}(3.3982165,-1.8946121)(6.025617,-1.8946121)
\psdots[linecolor=black, fillstyle=solid, dotstyle=o, dotsize=0.49939746](12.0,-2.695)
\psdots[linecolor=black, dotsize=0.49939746](14.4,2.005)
\psdots[linecolor=black, dotsize=0.49939746](17.0,0.005)
\pscustom[linecolor=black, linewidth=0.05]
{
\newpath
\moveto(9.542138,-3.594333)
\lineto(9.739638,-3.412497)
\curveto(9.838388,-3.321579)(10.289817,-3.1397436)(10.642496,-3.0488257)
\curveto(10.995175,-2.9579077)(11.453657,-2.8488061)(11.77107,-2.7942553)
}
\pscustom[linecolor=black, linewidth=0.05]
{
\newpath
\moveto(8.836145,-0.694325)
\lineto(9.101259,-0.62612915)
\curveto(9.233816,-0.59203124)(9.83979,-0.5238354)(10.313208,-0.48973754)
\curveto(10.786626,-0.45563966)(11.402069,-0.41472167)(11.828146,-0.39426297)
}
\psdots[linecolor=black, fillstyle=solid, dotstyle=o, dotsize=0.49939746](12.0,-0.395)
\psdots[linecolor=black, fillstyle=solid, dotstyle=o, dotsize=0.49939746](15.8,0.105)
\pscustom[linecolor=black, linewidth=0.05]
{
\newpath
\moveto(6.1600714,-1.9943249)
\lineto(6.449799,-2.335243)
\curveto(6.594663,-2.505702)(7.2568984,-2.8466198)(7.7742686,-3.0170789)
\curveto(8.291639,-3.187538)(8.964221,-3.3920887)(9.429855,-3.4943638)
}
\pscustom[linecolor=black, linewidth=0.05]
{
\newpath
\moveto(6.1872716,-1.7947204)
\lineto(6.3992214,-1.5674341)
\curveto(6.505196,-1.4537903)(6.989652,-1.2265038)(7.368133,-1.1128601)
\curveto(7.746614,-0.9992163)(8.23864,-0.8628442)(8.579272,-0.79465866)
}
\pscustom[linecolor=black, linewidth=0.05]
{
\newpath
\moveto(9.671069,-3.594325)
\lineto(9.974898,-3.684812)
\curveto(10.126813,-3.7300556)(10.8212805,-3.8205426)(11.363833,-3.865786)
\curveto(11.906384,-3.9110296)(12.611702,-3.9653223)(13.1,-3.992468)
}
\pscustom[linecolor=black, linewidth=0.05]
{
\newpath
\moveto(12.1421385,1.1056671)
\lineto(12.339639,1.3102301)
\curveto(12.438389,1.4125116)(12.889817,1.6170752)(13.2424965,1.7193567)
\curveto(13.595176,1.8216382)(14.053658,1.9443759)(14.371069,2.0057447)
}
\pscustom[linecolor=black, linewidth=0.05]
{
\newpath
\moveto(14.442139,2.105667)
\lineto(14.923183,2.287503)
\curveto(15.163706,2.378421)(16.263235,2.560257)(17.122244,2.651175)
\curveto(17.981253,2.7420928)(19.097961,2.8511941)(19.871069,2.9057448)
}
\psdots[linecolor=black, dotsize=0.49939746](8.7,-0.695)
\psdots[linecolor=black, fillstyle=solid, dotstyle=o, dotsize=0.49939746](12.0,1.005)
\psline[linecolor=colour0, linewidth=0.05](16.0,0.105)(17.1,0.005)
\psdots[linecolor=black, dotsize=0.49939746](9.6,-3.495)
\psdots[linecolor=black, dotsize=0.49939746](13.9,-4.995)
\psdots[linecolor=black, dotsize=0.49939746](17.0,-3.995)
\psdots[linecolor=black, fillstyle=solid, dotstyle=o, dotsize=0.49939746](12.0,-3.995)
\psdots[linecolor=black, fillstyle=solid, dotstyle=o, dotsize=0.49939746](15.8,2.405)
\psline[linecolor=black, linewidth=0.02](6.0,3.505)(6.0,-5.595)(6.0,-5.234604)
\psline[linecolor=black, linewidth=0.02](12.0,3.505)(12.0,-5.595)(12.0,-5.595)
\psline[linecolor=black, linewidth=0.02](15.909255,3.5037165)(15.8,3.505)(15.8,-5.595)
\psdots[linecolor=black, fillstyle=solid, dotstyle=o, dotsize=0.49939746](12.0,1.005)
\psdots[linecolor=black, fillstyle=solid, dotstyle=o, dotsize=0.49939746](12.0,-2.695)
\psdots[linecolor=black, fillstyle=solid, dotstyle=o, dotsize=0.49939746](15.8,-3.995)
\psdots[linecolor=black, fillstyle=solid, dotstyle=o, dotsize=0.49939746](12.0,-3.995)
\psdots[linecolor=black, fillstyle=solid, dotstyle=o, dotsize=0.49939746](12.0,-0.395)
\psdots[linecolor=black, fillstyle=solid, dotstyle=o, dotsize=0.49939746](15.8,0.105)
\psdots[linecolor=black, fillstyle=solid, dotstyle=o, dotsize=0.49939746](15.8,-2.595)
\psdots[linecolor=black, fillstyle=solid, dotstyle=o, dotsize=0.49939746](15.8,2.405)
\psdots[linecolor=black, fillstyle=solid, dotstyle=o, dotsize=0.49939746](15.8,1.705)
\psdots[linecolor=black, dotstyle=otimes, dotsize=0.49939746](15.8,-2.595)
\psdots[linecolor=black, dotstyle=otimes, dotsize=0.49939746](15.8,2.405)
\psdots[linecolor=black, dotstyle=otimes, dotsize=0.49939746](15.8,1.705)
\psdots[linecolor=black, fillstyle=solid, dotstyle=o, dotsize=0.49939746](6.0,-1.895)
\psline[linecolor=black, linewidth=0.02](10.4,3.505)(10.4,-5.595)(10.4,-5.595)
\psline[linecolor=black, linewidth=0.02](11.2,3.505)(11.2,-5.595)(11.2,-4.776)
\psdots[linecolor=black, fillstyle=solid, dotstyle=o, dotsize=0.49939746](10.4,0.505)
\psdots[linecolor=black, fillstyle=solid, dotstyle=o, dotsize=0.475374](11.2,0.705)
\psdots[linecolor=black, fillstyle=solid, dotstyle=o, dotsize=0.49939746](11.2,-0.395)
\psdots[linecolor=black, fillstyle=solid, dotstyle=o, dotsize=0.49939746](10.4,-0.495)
\psdots[linecolor=black, fillstyle=solid, dotstyle=o, dotsize=0.49939746](11.2,-2.895)
\psdots[linecolor=black, fillstyle=solid, dotstyle=o, dotsize=0.49939746](10.4,-3.095)
\psdots[linecolor=black, fillstyle=solid, dotstyle=o, dotsize=0.49939746](11.2,-3.795)
\psdots[linecolor=black, fillstyle=solid, dotstyle=o, dotsize=0.49939746](10.4,-3.795)
\psdots[linecolor=black, dotstyle=otimes, dotsize=0.49939746](11.2,0.705)
\psdots[linecolor=black, dotstyle=otimes, dotsize=0.49939746](11.2,-0.395)
\psdots[linecolor=black, dotstyle=otimes, dotsize=0.49939746](11.2,-2.895)
\psdots[linecolor=black, dotstyle=otimes, dotsize=0.49939746](11.2,-3.795)
\psdots[linecolor=black, dotstyle=otimes, dotsize=0.49939746](10.4,0.505)
\psdots[linecolor=black, dotstyle=otimes, dotsize=0.49939746](10.4,-0.495)
\psdots[linecolor=black, dotstyle=otimes, dotsize=0.49939746](10.4,-3.095)
\psdots[linecolor=black, dotstyle=otimes, dotsize=0.49939746](10.4,-3.795)
\psline[linecolor=black, linewidth=0.02](15.1,3.505)(15.1,3.505)(15.1,-5.595)
\psdots[linecolor=black, fillstyle=solid, dotstyle=o, dotsize=0.49939746](15.1,2.305)
\psdots[linecolor=black, fillstyle=solid, dotstyle=o, dotsize=0.49939746](15.1,1.805)
\psdots[linecolor=black, fillstyle=solid, dotstyle=o, dotsize=0.49939746](15.1,0.105)
\psdots[linecolor=black, fillstyle=solid, dotstyle=o, dotsize=0.49939746](15.1,-2.695)
\psdots[linecolor=black, fillstyle=solid, dotstyle=o, dotsize=0.49939746](15.1,-3.995)
\psdots[linecolor=black, dotstyle=otimes, dotsize=0.49939746](15.1,1.805)
\psdots[linecolor=black, dotstyle=otimes, dotsize=0.49939746](15.1,0.105)
\psdots[linecolor=black, dotstyle=otimes, dotsize=0.49939746](15.1,-2.695)
\psdots[linecolor=black, dotstyle=otimes, dotsize=0.49939746](15.1,-3.995)
\psdots[linecolor=black, dotstyle=otimes, dotsize=0.49939746](15.1,2.305)
\psline[linecolor=black, linewidth=0.02](19.300177,3.505)(19.3,3.505)(19.299824,-5.595)
\psdots[linecolor=black, fillstyle=solid, dotstyle=o, dotsize=0.49939746](19.3,-1.295)
\psdots[linecolor=black, fillstyle=solid, dotstyle=o, dotsize=0.49939746](19.3,-0.495)
\psdots[linecolor=black, fillstyle=solid, dotstyle=o, dotsize=0.49939746](19.3,0.405)
\psdots[linecolor=black, fillstyle=solid, dotstyle=o, dotsize=0.49939746](19.3,1.005)
\psdots[linecolor=black, fillstyle=solid, dotstyle=o, dotsize=0.49939746](19.3,1.605)
\psdots[linecolor=black, fillstyle=solid, dotstyle=o, dotsize=0.49939746](19.3,2.905)
\psdots[linecolor=black, fillstyle=solid, dotstyle=o, dotsize=0.49939746](19.3,-5.095)
\psdots[linecolor=black, fillstyle=solid, dotstyle=o, dotsize=0.49939746](19.3,-4.295)
\psdots[linecolor=black, fillstyle=solid, dotstyle=o, dotsize=0.49939746](19.3,-3.395)
\psdots[linecolor=black, fillstyle=solid, dotstyle=o, dotsize=0.49939746](19.3,-2.395)
\psdots[linecolor=black, dotstyle=otimes, dotsize=0.49939746](19.3,2.905)
\psdots[linecolor=black, dotstyle=otimes, dotsize=0.49939746](19.3,1.605)
\psdots[linecolor=black, dotstyle=otimes, dotsize=0.49939746](19.3,1.005)
\psdots[linecolor=black, dotstyle=otimes, dotsize=0.49939746](19.3,0.405)
\psdots[linecolor=black, dotstyle=otimes, dotsize=0.49939746](19.3,-0.495)
\psdots[linecolor=black, dotstyle=otimes, dotsize=0.49939746](19.3,-1.295)
\psdots[linecolor=black, dotstyle=otimes, dotsize=0.49939746](19.3,-2.395)
\psdots[linecolor=black, dotstyle=otimes, dotsize=0.49939746](19.3,-3.395)
\psdots[linecolor=black, dotstyle=otimes, dotsize=0.49939746](19.3,-4.295)
\psdots[linecolor=black, dotstyle=otimes, dotsize=0.49939746](19.3,-5.095)
\psline[linecolor=black, linewidth=0.02](5.0,3.505)(5.0,-5.595)(5.0,-5.595)
\psdots[linecolor=black, fillstyle=solid, dotstyle=o, dotsize=0.50171536](5.0,-1.895)
\psdots[linecolor=black, dotstyle=otimes, dotsize=0.49939746](5.0,-1.895)
\end{pspicture}
}
\caption{\label{fig:9}The tree $\xi = \varsigma(\frT,\delta) \in \pr$.}
\end{figure}
It is also clear that the maps $\varrho : \pr \to \Tm^\bbN_2$ and
$\varsigma : \Tm^\bbN_2 \to \pr$ are inverse to each other, showing that
$ \pr \cong \Tm^\bbN_2$. This finishes the proof. 
\end{proof}

In the description of Lemma~\ref{tam-as-tree}, the $w$-image of a leveled tree 
$\beta \in \LTr$ in $\Tm^\bbN_2$ can be obtained in four steps. 
First we forget all type-(ii)
levels of $\beta$, so the \cerny-vertices are allowed to move
freely. In the second step we split all \tlusty-vertices of arity $n > 1$ into
one  \cerny-vertex of arity $1$ followed by $n$  \tlusty-vertices
of arities $1$, graphically:  
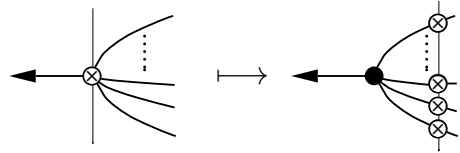
\begin{center}
\psscalebox{.5 .5}
{
\begin{pspicture}(1,-4.0)(14.735774,0)
\definecolor{colour0}{rgb}{0.03137255,0.003921569,0.003921569}
\pscustom[linecolor=black, linewidth=0.05]
{
\newpath
\moveto(8.9,-0.6)
}
\pscustom[linecolor=black, linewidth=0.05]
{
\newpath
\moveto(21.5,0.7)
}

\pscustom[linecolor=black, linewidth=0.05]
{
\newpath
\moveto(3.5633454,-1.7997204)
\lineto(3.7487125,-1.4815247)
\curveto(3.8413959,-1.3224268)(4.2650924,-1.004231)(4.596105,-0.84513307)
\curveto(4.927118,-0.68603516)(5.3574343,-0.4951172)(5.655346,-0.39965865)
}
\psline[linecolor=black, linewidth=0.02](3.5239174,-0.19922417)(3.5239174,-3.8)(3.5239174,-3.7549903)
\psline[linecolor=black, linewidth=0.05, arrowsize=0.05291666666666672cm 6.0,arrowlength=2.0,arrowinset=0.0]{<-}(1.2982163,-1.9996121)(3.325617,-1.9996121)
\pscustom[linecolor=black, linewidth=0.05]
{
\newpath
\moveto(3.6361454,-2.1993248)
\lineto(3.8215125,-2.517516)
\curveto(3.9141963,-2.6766114)(4.3378925,-2.994801)(4.6689053,-3.1538966)
\curveto(4.999918,-3.3129919)(5.4302344,-3.5039062)(5.728146,-3.5993638)
}
\psdots[linecolor=black, fillstyle=solid, dotstyle=o, dotsize=0.49939746](12.7,-2.2)
\psdots[linecolor=black, dotstyle=otimes, dotsize=0.49939746](12.7,-0.6)
\pscustom[linecolor=black, linewidth=0.05]
{
\newpath
\moveto(3.6734257,-2.060248)
\lineto(3.8524334,-2.0899072)
\curveto(3.9419374,-2.104737)(4.3518853,-2.143313)(4.672329,-2.1670594)
\curveto(4.992773,-2.1908057)(5.4093804,-2.2205396)(5.6978717,-2.2385025)
}
\psdots[linecolor=black, dotstyle=otimes, dotsize=0.49833098](12.7,-2.2)
\psdots[linecolor=black, fillstyle=solid, dotstyle=o,dotsize=0.49939746](3.5,-2.0)
\psdots[linecolor=black, fillstyle=solid, dotstyle=otimes,dotsize=0.49939746](3.5,-2.0)

\psdots[linecolor=black, dotsize=0.49939746](11.0,-2.0)
\psline[linecolor=colour0, linewidth=0.08, linestyle=dotted, dotsep=0.10583334cm](4.9,-0.9)(4.9,-1.9)(4.9,-1.8)
\pscustom[linecolor=black, linewidth=0.05]
{
\newpath
\moveto(3.6734257,-2.1602478)
\lineto(3.8524334,-2.2731006)
\curveto(3.9419374,-2.3295276)(4.3518853,-2.4763086)(4.672329,-2.5666625)
\curveto(4.992773,-2.6570165)(5.4093804,-2.7701538)(5.6978717,-2.8385026)
}
\pscustom[linecolor=black, linewidth=0.05]
{
\newpath
\moveto(11.063345,-1.7997204)
\lineto(11.248713,-1.4815247)
\curveto(11.341396,-1.3224268)(11.765093,-1.004231)(12.096106,-0.84513307)
\curveto(12.427119,-0.68603516)(12.857434,-0.4951172)(13.155346,-0.39965865)
}
\psline[linecolor=black, linewidth=0.02](12.723917,0.0)(12.723917,-4.0)(12.723917,-3.9499903)
\psline[linecolor=black, linewidth=0.05, arrowsize=0.05291666666666672cm 6.0,arrowlength=2.0,arrowinset=0.0]{<-}(8.798217,-1.9996121)(10.825617,-1.9996121)
\pscustom[linecolor=black, linewidth=0.05]
{
\newpath
\moveto(11.136145,-2.1993248)
\lineto(11.321512,-2.517516)
\curveto(11.414197,-2.6766114)(11.8378935,-2.994801)(12.168905,-3.1538966)
\curveto(12.499917,-3.3129919)(12.930234,-3.5039062)(13.228146,-3.5993638)
}
\pscustom[linecolor=black, linewidth=0.05]
{
\newpath
\moveto(11.173426,-2.060248)
\lineto(11.290537,-2.0899072)
\curveto(11.349093,-2.104737)(11.617292,-2.143313)(11.826935,-2.1670594)
\curveto(12.036577,-2.1908057)(12.309134,-2.2205396)(12.497871,-2.2385025)
}
\psline[linecolor=colour0, linewidth=0.08, linestyle=dotted, dotsep=0.10583334cm](12.4,-0.9)(12.4,-1.9)(12.4,-1.8)
\pscustom[linecolor=black, linewidth=0.05]
{
\newpath
\moveto(11.173426,-2.1602478)
\lineto(11.352434,-2.28974)
\curveto(11.441938,-2.354486)(11.851885,-2.5229077)(12.172329,-2.6265833)
\curveto(12.492773,-2.7302587)(12.90938,-2.8600762)(13.197871,-2.9385026)
}
\psdots[linecolor=black, fillstyle=solid, dotstyle=o, dotsize=0.49939746](12.7,-0.6)
\psdots[linecolor=black, fillstyle=solid, dotstyle=o, dotsize=0.49939746](12.7,-3.4)
\psdots[linecolor=black, fillstyle=solid, dotstyle=o, dotsize=0.49939746](12.7,-2.8)
\psdots[linecolor=black, dotstyle=otimes, dotsize=0.49939746](12.7,-0.6)
\psdots[linecolor=black, dotstyle=otimes, dotsize=0.49939746](12.7,-2.8)
\psdots[linecolor=black, dotstyle=otimes, dotsize=0.49939746](12.7,-3.4)
\psline[linecolor=colour0, linewidth=0.05](12.9,-2.2)(13.2,-2.2)(13.2,-2.2)
\psdots[linecolor=black, fillstyle=solid, dotstyle=o, dotsize=0.49939746](12.7,-2.2)
\psdots[linecolor=black, dotstyle=otimes, dotsize=0.49939746](12.7,-2.2)
\rput(7.5,-2.05){\textcolor{colour0}{\Huge $\longmapsto$}}
\end{pspicture}
}
\end{center}
The third step removes all internal edges starting at a \cerny-vertex and
ending at a \tlusty-vertex by allowing \cerny-vertices to penetrate
through \tlusty-vertices as in:
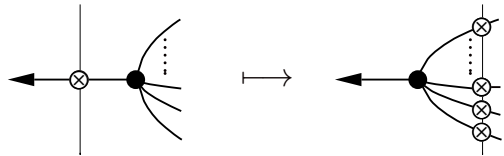
\[
\raisebox{-2em}{}
\psscalebox{.5 .5}
{
\begin{pspicture}(5,-2.2)(14.735774,0)
\definecolor{colour0}{rgb}{0.03137255,0.003921569,0.003921569}
\pscustom[linecolor=black, linewidth=0.05]
{
\newpath
\moveto(8.9,-0.6)
}
\pscustom[linecolor=black, linewidth=0.05]
{
\newpath
\moveto(21.5,0.7)
}

\pscustom[linecolor=black, linewidth=0.05]
{
\newpath
\moveto(6.5633454,-1.7997204)
\lineto(6.6630845,-1.4815247)
\curveto(6.712954,-1.3224268)(6.9409304,-1.004231)(7.1190357,-0.84513307)
\curveto(7.2971416,-0.68603516)(7.5286794,-0.4951172)(7.688975,-0.39965865)
}
\psline[linecolor=black, linewidth=0.02](5.023917,0.0)(5.023917,-4.0)(5.023917,-3.9499903)
\psline[linecolor=black, linewidth=0.05, arrowsize=0.05291666666666672cm 6.0,arrowlength=2.0,arrowinset=0.0]{<-}(3.0982163,-1.9996121)(6.325617,-1.9996121)
\pscustom[linecolor=black, linewidth=0.05]
{
\newpath
\moveto(6.6025167,-2.1993248)
\lineto(6.7022557,-2.517516)
\curveto(6.7521253,-2.6766114)(6.9801006,-2.994801)(7.158207,-3.1538966)
\curveto(7.336313,-3.3129919)(7.56785,-3.5039062)(7.728146,-3.5993638)
}
\psdots[linecolor=black, fillstyle=solid, dotstyle=o, dotsize=0.49939746](15.7,-2.2)
\psdots[linecolor=black, dotstyle=otimes, dotsize=0.49939746](15.7,-0.6)
\pscustom[linecolor=black, linewidth=0.05]
{
\newpath
\moveto(6.6225758,-2.060248)
\lineto(6.7188935,-2.0899072)
\curveto(6.767052,-2.104737)(6.987631,-2.143313)(7.1600494,-2.1670594)
\curveto(7.332468,-2.1908057)(7.5566297,-2.2205396)(7.7118564,-2.2385025)
}
\psdots[linecolor=black, dotstyle=otimes, dotsize=0.49833098](15.7,-2.2)
\psdots[linecolor=black, dotsize=0.49939746](14.0,-2.0)
\psline[linecolor=colour0, linewidth=0.08, linestyle=dotted, dotsep=0.10583334cm](7.282551,-0.9)(7.282551,-1.9)(7.282551,-1.8)
\pscustom[linecolor=black, linewidth=0.05]
{
\newpath
\moveto(6.6225758,-2.1602478)
\lineto(6.7188935,-2.2731006)
\curveto(6.767052,-2.3295276)(6.987631,-2.4763086)(7.1600494,-2.5666625)
\curveto(7.332468,-2.6570165)(7.5566297,-2.7701538)(7.7118564,-2.8385026)
}
\pscustom[linecolor=black, linewidth=0.05]
{
\newpath
\moveto(14.063345,-1.7997204)
\lineto(14.248713,-1.4815247)
\curveto(14.341396,-1.3224268)(14.765093,-1.004231)(15.096106,-0.84513307)
\curveto(15.427119,-0.68603516)(15.857434,-0.4951172)(16.155346,-0.39965865)
}
\psline[linecolor=black, linewidth=0.02](15.723917,0.0)(15.723917,-4.0)(15.723917,-3.9499903)
\psline[linecolor=black, linewidth=0.05, arrowsize=0.05291666666666672cm 6.0,arrowlength=2.0,arrowinset=0.0]{<-}(11.798217,-1.9996121)(13.825617,-1.9996121)
\pscustom[linecolor=black, linewidth=0.05]
{
\newpath
\moveto(14.136145,-2.1993248)
\lineto(14.321512,-2.517516)
\curveto(14.414197,-2.6766114)(14.8378935,-2.994801)(15.168905,-3.1538966)
\curveto(15.499917,-3.3129919)(15.930234,-3.5039062)(16.228146,-3.5993638)
}
\pscustom[linecolor=black, linewidth=0.05]
{
\newpath
\moveto(14.173426,-2.060248)
\lineto(14.290537,-2.0899072)
\curveto(14.349093,-2.104737)(14.617292,-2.143313)(14.826935,-2.1670594)
\curveto(15.036577,-2.1908057)(15.309134,-2.2205396)(15.497871,-2.2385025)
}
\psline[linecolor=colour0, linewidth=0.08, linestyle=dotted, dotsep=0.10583334cm](15.4,-0.9)(15.4,-1.9)(15.4,-1.8)
\pscustom[linecolor=black, linewidth=0.05]
{
\newpath
\moveto(14.173426,-2.1602478)
\lineto(14.352434,-2.28974)
\curveto(14.441938,-2.354486)(14.851885,-2.5229077)(15.172329,-2.6265833)
\curveto(15.492773,-2.7302587)(15.90938,-2.8600762)(16.197872,-2.9385026)
}
\psdots[linecolor=black, fillstyle=solid, dotstyle=o, dotsize=0.49939746](15.7,-0.6)
\psdots[linecolor=black, fillstyle=solid, dotstyle=o, dotsize=0.49939746](15.7,-3.4)
\psdots[linecolor=black, fillstyle=solid, dotstyle=o, dotsize=0.49939746](15.7,-2.8)
\psdots[linecolor=black, dotstyle=otimes, dotsize=0.49939746](15.7,-0.6)
\psdots[linecolor=black, dotstyle=otimes, dotsize=0.49939746](15.7,-2.8)
\psdots[linecolor=black, dotstyle=otimes, dotsize=0.49939746](15.7,-3.4)
\psline[linecolor=colour0, linewidth=0.05](15.9,-2.2)(16.2,-2.2)(16.2,-2.2)
\psdots[linecolor=black, fillstyle=solid, dotstyle=o, dotsize=0.49939746](15.7,-2.2)
\psdots[linecolor=black, dotstyle=otimes, dotsize=0.49939746](15.7,-2.2)
\psdots[linecolor=black, dotstyle=otimes, dotsize=0.49939746](5.0,-2.0)
\psdots[linecolor=colour0, fillstyle=solid, dotstyle=o, dotsize=0.5064466](5.0,-2.0)
\psdots[linecolor=black, dotstyle=otimes, dotsize=0.49939746](5.0,-2.0)
\psdots[linecolor=black, dotsize=0.49939746](6.5,-2.0)
\rput(10.0,-2.0){\textcolor{colour0}{\Huge $\longmapsto$}}
\end{pspicture}
}
\]
In the final step we contract all edges connecting two horizontal
vertices and remove horizontal vertices of arity $1$. The result
is the image $w(\beta)$. 

We leave as an exercise to show that if we apply the above
steps to the tree $\beta$ in
Figure~\ref{fig:1},  we get the tree $w(\beta)$ in Figure~\ref{w-redukce}.
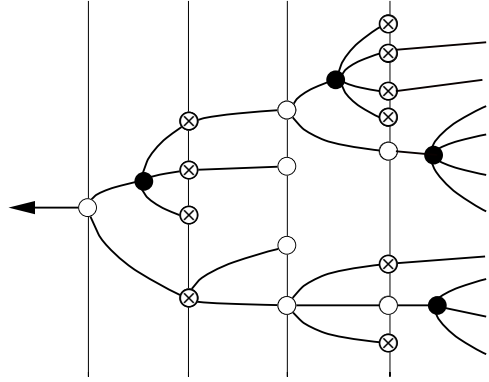
\begin{figure}[t]
\centering 
\psscalebox{.5 .5} 
{
\begin{pspicture}(6.5,-5.75)(18.502516,4.9)
\definecolor{colour0}{rgb}{0.03137255,0.003921569,0.003921569}
\psline[linecolor=black, linewidth=0.02](2.597499,1.7506206)(2.597499,1.7506206)
\pscustom[linecolor=black, linewidth=0.05]
{
\newpath
\moveto(8.9,1.15)
}
\pscustom[linecolor=black, linewidth=0.05]
{
\newpath
\moveto(21.5,2.45)
}

\pscustom[linecolor=black, linewidth=0.05]
{
\newpath
\moveto(13.113208,-3.849814)
\lineto(13.22968,-3.996884)
\curveto(13.287915,-4.0704193)(13.429346,-4.2174892)(13.512539,-4.291024)
\curveto(13.595733,-4.364559)(13.895232,-4.5116296)(14.111536,-4.585165)
\curveto(14.327839,-4.6587)(14.843642,-4.761649)(15.742138,-4.849891)
}
\pscustom[linecolor=black, linewidth=0.05]
{
\newpath
\moveto(13.113208,-3.849814)
\lineto(13.22968,-3.6880364)
\curveto(13.287915,-3.607148)(13.429346,-3.4453712)(13.512539,-3.3644824)
\curveto(13.595733,-3.2835937)(13.895232,-3.1218164)(14.111536,-3.0409276)
\curveto(14.327839,-2.9600391)(14.843642,-2.846795)(15.742138,-2.7497284)
}
\pscustom[linecolor=black, linewidth=0.05]
{
\newpath
\moveto(13.113208,1.3505896)
\lineto(13.22968,1.1888123)
\curveto(13.287915,1.1079236)(13.429346,0.94614655)(13.512539,0.86525786)
\curveto(13.595733,0.78436923)(13.895232,0.62259215)(14.111536,0.54170346)
\curveto(14.327839,0.46081483)(14.843642,0.3475708)(15.742138,0.25050429)
}
\psline[linecolor=black, linewidth=0.05](13.113208,-3.8498137)(13.113208,-3.8498137)
\psline[linecolor=black, linewidth=0.05](13.113208,-3.8498137)(16.871069,-3.8498137)(16.871069,-3.8498137)
\psline[linecolor=black, linewidth=0.05](15.742138,-2.7497284)(18.371069,-2.549713)(18.371069,-2.549713)
\pscustom[linecolor=black, linewidth=0.05]
{
\newpath
\moveto(16.871069,0.15818115)
\lineto(17.006544,0.45190856)
\curveto(17.074282,0.598772)(17.383938,0.8924997)(17.625858,1.0393634)
\curveto(17.867779,1.1862271)(18.182274,1.3624637)(18.4,1.4505819)
}
\pscustom[linecolor=black, linewidth=0.05]
{
\newpath
\moveto(16.97107,-3.849814)
\lineto(17.097683,-3.6907105)
\curveto(17.16099,-3.6111584)(17.450394,-3.4520557)(17.67649,-3.3725038)
\curveto(17.902588,-3.2929523)(18.196514,-3.1974902)(18.4,-3.1497595)
}
\pscustom[linecolor=black, linewidth=0.05]
{
\newpath
\moveto(16.947515,0.15818115)
\lineto(17.076218,0.28056762)
\curveto(17.140568,0.34176087)(17.434742,0.46414733)(17.664566,0.52534056)
\curveto(17.894388,0.5865338)(18.19316,0.6599652)(18.4,0.6966814)
}
\pscustom[linecolor=black, linewidth=0.05]
{
\newpath
\moveto(14.542139,2.1506672)
\lineto(14.633309,2.0824714)
\curveto(14.678894,2.0483735)(14.887285,1.9801776)(15.050092,1.9460797)
\curveto(15.212898,1.9119818)(15.424544,1.8710638)(15.571069,1.8506051)
}
\pscustom[linecolor=black, linewidth=0.05]
{
\newpath
\moveto(9.263346,-0.34972045)
\lineto(9.377826,-0.03152466)
\curveto(9.435066,0.12757324)(9.696737,0.44576904)(9.901168,0.6048669)
\curveto(10.105598,0.76396483)(10.371359,0.9548828)(10.555346,1.0503414)
}
\pscustom[linecolor=black, linewidth=0.05]
{
\newpath
\moveto(16.97107,-3.849814)
\lineto(17.097683,-4.1452913)
\curveto(17.16099,-4.29303)(17.450394,-4.588507)(17.67649,-4.736246)
\curveto(17.902588,-4.8839846)(18.196514,-5.061272)(18.4,-5.1499147)
}
\pscustom[linecolor=black, linewidth=0.05]
{
\newpath
\moveto(16.871069,0.15818115)
\lineto(17.006544,-0.18450074)
\curveto(17.074282,-0.3558423)(17.383938,-0.6985242)(17.625858,-0.8698651)
\curveto(17.867779,-1.041206)(18.182274,-1.2468152)(18.4,-1.3496199)
}
\pscustom[linecolor=black, linewidth=0.05]
{
\newpath
\moveto(16.947515,0.15818115)
\lineto(17.076218,0.03579468)
\curveto(17.140568,-0.02539856)(17.434742,-0.14778504)(17.664566,-0.20897827)
\curveto(17.894388,-0.27017152)(18.19316,-0.34360352)(18.4,-0.38031924)
}
\psline[linecolor=black, linewidth=0.05](16.97107,-3.8498137)(18.4,-4.149837)(18.4,-4.149837)
\psline[linecolor=black, linewidth=0.02](7.8239174,4.250776)(7.8239174,-5.75)(7.8239174,-5.6249905)
\psline[linecolor=black, linewidth=0.02](10.484277,4.250776)(10.484277,-5.75)(10.484277,-5.6249905)
\psline[linecolor=black, linewidth=0.02](13.113208,4.250776)(13.113208,-5.75)(13.113208,-5.6249905)
\psline[linecolor=black, linewidth=0.02](15.842138,4.250776)(15.842138,-5.75)(15.842138,-5.6249905)
\psdots[linecolor=black, fillstyle=solid, dotstyle=o, dotsize=0.010382021](10.5,1.15)
\psdots[linecolor=black, fillstyle=solid, dotstyle=o, dotsize=0.49939746](15.8,3.65)
\psdots[linecolor=black, fillstyle=solid, dotstyle=o, dotsize=0.49939746](13.1,1.35)
\psdots[linecolor=black, fillstyle=solid, dotstyle=o, dotsize=0.49939746](15.8,0.25)
\psdots[linecolor=black, fillstyle=solid, dotstyle=o, dotsize=0.49939746](15.8,-4.85)
\psdots[linecolor=black, fillstyle=solid, dotstyle=o, dotsize=0.502276](15.8,2.85)
\psdots[linecolor=black, fillstyle=solid, dotstyle=o, dotsize=0.49939746](15.8,-2.75)
\psdots[linecolor=black, dotstyle=otimes, dotsize=0.49939746](15.8,0.25)
\psdots[linecolor=black, dotstyle=otimes, dotsize=0.49939746](15.8,-2.75)
\psdots[linecolor=black, dotstyle=otimes, dotsize=0.49939746](15.8,-4.85)
\psdots[linecolor=black, dotstyle=otimes, dotsize=0.49939746](15.8,2.85)
\psdots[linecolor=black, dotstyle=otimes, dotsize=0.49939746](15.8,-3.85)
\psdots[linecolor=black, fillstyle=solid, dotstyle=o, dotsize=0.49939746](13.1,-3.85)
\psline[linecolor=black, linewidth=0.05, arrowsize=0.05291666666666672cm 6.0,arrowlength=2.0,arrowinset=0.0]{<-}(5.6982164,-1.2496121)(7.725617,-1.2496121)
\psdots[linecolor=black, fillstyle=solid, dotstyle=o, dotsize=0.49939746](10.5,-3.65)
\psdots[linecolor=black, dotstyle=otimes, dotsize=0.49939746](10.5,-3.65)
\psdots[linecolor=black, fillstyle=solid, dotstyle=o, dotsize=0.49939746](13.1,-2.25)
\psdots[linecolor=black, dotsize=0.49939746](14.4,2.15)
\psdots[linecolor=black, dotsize=0.49939746](17.0,0.15)
\psdots[linecolor=black, dotsize=0.49939746](17.1,-3.85)
\pscustom[linecolor=black, linewidth=0.05]
{
\newpath
\moveto(10.542139,-3.449333)
\lineto(10.7485,-3.199315)
\curveto(10.85168,-3.0743067)(11.323361,-2.8242884)(11.691864,-2.6992798)
\curveto(12.060367,-2.5742712)(12.539417,-2.4242601)(12.871069,-2.3492553)
}
\pscustom[linecolor=black, linewidth=0.05]
{
\newpath
\moveto(9.236145,-0.5493249)
\lineto(9.350626,-0.48112914)
\curveto(9.4078665,-0.44703126)(9.669538,-0.37883544)(9.873968,-0.34473756)
\curveto(10.078399,-0.31063965)(10.344159,-0.2697217)(10.528146,-0.24926297)
}
\pscustom[linecolor=black, linewidth=0.05]
{
\newpath
\moveto(9.236145,-0.749325)
\lineto(9.350626,-0.93115234)
\curveto(9.4078665,-1.0220654)(9.669538,-1.2038928)(9.873968,-1.2948059)
\curveto(10.078399,-1.3857197)(10.344159,-1.4948157)(10.528146,-1.5493639)
}
\psdots[linecolor=black, fillstyle=solid, dotstyle=o, dotsize=0.49939746](10.5,-1.45)
\psline[linecolor=black, linewidth=0.05](10.542138,-0.24933279)(13.171069,-0.14932503)(13.171069,-0.14932503)
\psdots[linecolor=black, fillstyle=solid, dotstyle=o, dotsize=0.49939746](10.5,-0.25)
\psdots[linecolor=black, fillstyle=solid, dotstyle=o, dotsize=0.49939746](13.1,-0.15)
\psdots[linecolor=black, dotstyle=otimes, dotsize=0.010382021](10.5,1.15)
\psdots[linecolor=black, dotstyle=otimes, dotsize=0.49939746](10.5,-0.25)
\psdots[linecolor=black, dotstyle=otimes, dotsize=0.49939746](10.5,-1.45)
\psdots[linecolor=black, fillstyle=solid, dotstyle=o, dotsize=0.49939746](15.8,-3.85)
\psdots[linecolor=black, fillstyle=solid, dotstyle=o, dotsize=0.49939746](15.8,0.25)
\psdots[linecolor=black, fillstyle=solid, dotstyle=o, dotsize=0.49939746](10.5,1.05)
\psdots[linecolor=black, dotstyle=otimes, dotsize=0.49939746](10.5,1.05)
\pscustom[linecolor=black, linewidth=0.05]
{
\newpath
\moveto(7.860071,-1.449325)
\lineto(8.087773,-1.9493335)
\curveto(8.201625,-2.1993384)(8.722088,-2.699347)(9.128698,-2.9493518)
\curveto(9.53531,-3.199356)(10.063905,-3.4993615)(10.429855,-3.6493638)
}
\pscustom[linecolor=black, linewidth=0.05]
{
\newpath
\moveto(7.887272,-1.0497204)
\lineto(8.001753,-0.93607056)
\curveto(8.058993,-0.8792456)(8.320664,-0.7655945)(8.525095,-0.70876956)
\curveto(8.729526,-0.6519446)(8.995284,-0.58375365)(9.179273,-0.54965866)
}
\pscustom[linecolor=black, linewidth=0.05]
{
\newpath
\moveto(10.671069,-3.749325)
\lineto(10.8862915,-3.772061)
\curveto(10.993901,-3.783429)(11.485837,-3.8061652)(11.870161,-3.817533)
\curveto(12.254486,-3.8289013)(12.754108,-3.8425426)(13.1,-3.8493638)
}
\psdots[linecolor=black, fillstyle=solid, dotstyle=o, dotsize=0.49939746](13.1,-3.85)
\psdots[linecolor=black, fillstyle=solid, dotstyle=o, dotsize=0.49939746](10.5,-3.65)
\psdots[linecolor=black, dotstyle=otimes, dotsize=0.49939746](15.8,1.15)
\psdots[linecolor=black, dotstyle=otimes, dotsize=0.49939746](10.5,-3.65)
\pscustom[linecolor=black, linewidth=0.05]
{
\newpath
\moveto(10.6421385,1.0506672)
\lineto(10.848499,1.1188666)
\curveto(10.951679,1.1529663)(11.423362,1.2211658)(11.791864,1.2552655)
\curveto(12.160366,1.2893653)(12.639418,1.3302851)(12.971069,1.3507448)
}
\psdots[linecolor=black, fillstyle=solid, dotstyle=o, dotsize=0.49939746](15.8,1.15)
\psdots[linecolor=black, fillstyle=solid, dotstyle=o, dotsize=0.49939746](13.1,1.35)
\psdots[linecolor=black, fillstyle=solid, dotstyle=o, dotsize=0.49939746](10.5,1.05)
\psdots[linecolor=black, dotstyle=otimes, dotsize=0.49939746](10.5,1.05)
\psdots[linecolor=black, fillstyle=solid, dotstyle=o, dotsize=0.49939746](15.8,1.85)
\pscustom[linecolor=black, linewidth=0.05]
{
\newpath
\moveto(13.242139,1.4506671)
\lineto(13.342171,1.6097757)
\curveto(13.392186,1.6893299)(13.620831,1.8484387)(13.799458,1.9279928)
\curveto(13.9780855,2.007547)(14.210302,2.1030123)(14.371069,2.1507447)
}
\pscustom[linecolor=black, linewidth=0.05]
{
\newpath
\moveto(14.442139,2.250667)
\lineto(14.542171,2.5461395)
\curveto(14.592186,2.6938758)(14.82083,2.9893475)(14.999458,3.1370838)
\curveto(15.178086,3.2848198)(15.410303,3.4621031)(15.571069,3.5507448)
}
\pscustom[linecolor=black, linewidth=0.05]
{
\newpath
\moveto(14.542139,2.250667)
\lineto(14.633309,2.3870482)
\curveto(14.678894,2.4552393)(14.887285,2.5916204)(15.050092,2.659811)
\curveto(15.212898,2.7280016)(15.424544,2.8098304)(15.571069,2.8507447)
}
\pscustom[linecolor=black, linewidth=0.05]
{
\newpath
\moveto(14.542139,2.050667)
\lineto(14.633309,1.8461078)
\curveto(14.678894,1.7438282)(14.887285,1.5392681)(15.050092,1.4369885)
\curveto(15.212898,1.3347088)(15.424544,1.211973)(15.571069,1.1506052)
}
\psline[linecolor=colour0, linewidth=0.05](16.0,2.95)(18.4,3.15)(18.4,3.15)
\psline[linecolor=colour0, linewidth=0.05](16.0,1.85)(18.3,2.15)(18.3,2.15)
\psdots[linecolor=black, dotstyle=otimes, dotsize=0.49939746](15.8,3.65)
\psdots[linecolor=black, dotstyle=otimes, dotsize=0.49939746](15.8,1.85)
\psdots[linecolor=black, dotstyle=otimes, dotsize=0.49939746](15.8,1.15)
\psdots[linecolor=black, dotsize=0.49939746](9.3,-0.55)
\psdots[linecolor=black, fillstyle=solid, dotstyle=o, dotsize=0.49939746](7.8,-1.25)
\psdots[linecolor=black, fillstyle=solid, dotstyle=o, dotsize=0.49939746](13.1,1.35)
\psline[linecolor=colour0, linewidth=0.05](16.0,0.25)(17.1,0.15)
\end{pspicture}
}
\caption{
\label{w-redukce}
The $w$-image of the tree $\beta \in \LTr$ from Figure~\ref{fig:1}.}
\end{figure}
The following lemma describes all $\beta$'s in $\LTr$ with the same
$w$-image.

\begin{lemma} 
\label{elementary-moves}
  Let $\beta', \beta''\in \LTr$.  Then $w(\beta') =
  w(\beta'')$ if and only if 
$\beta''$ is obtained from $\beta'$ 
by a finite sequence of the following elementary moves and
  their inverses:
  \begin{itemize}
\item[(i)] 
introducing a new level of horizontal vertices of arity  one,
\item[(ii)]
choosing two adjacent levels of horizontal vertices 
and contracting all edges connecting vertices
in these two chosen levels, creating one level of horizontal vertices, 
\item[(iii)] 
replacing an arity-$1$ vertical vertex followed by an arity-$n$ horizontal 
vertex with an arity-$n$ vertical vertex followed by an arity  $1$ horizontal
vertex:
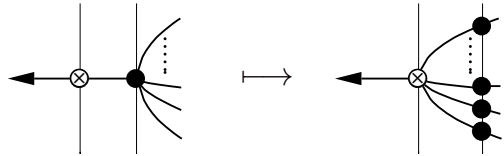
\begin{center}
\psscalebox{.5 .5}
{
\begin{pspicture}(5,-4.0)(14.735774,0)
\definecolor{colour0}{rgb}{0.03137255,0.003921569,0.003921569}
\pscustom[linecolor=black, linewidth=0.05]
{
\newpath
\moveto(8.9,-0.6)
}
\pscustom[linecolor=black, linewidth=0.05]
{
\newpath
\moveto(21.5,0.7)
}
\rput(10.0,-2.0){\textcolor{colour0}{\Huge $\longmapsto$}}

\pscustom[linecolor=black, linewidth=0.05]
{
\newpath
\moveto(6.5633454,-1.7997204)
\lineto(6.6630845,-1.4815247)
\curveto(6.712954,-1.3224268)(6.9409304,-1.004231)(7.1190357,-0.84513307)
\curveto(7.2971416,-0.68603516)(7.5286794,-0.4951172)(7.688975,-0.39965865)
}
\psline[linecolor=black, linewidth=0.02](5.023917,0.0)(5.023917,-4.0)(5.023917,-3.9499903)
\psline[linecolor=black, linewidth=0.05, arrowsize=0.05291666666666672cm 6.0,arrowlength=2.0,arrowinset=0.0]{<-}(3.0982163,-1.9996121)(6.325617,-1.9996121)
\pscustom[linecolor=black, linewidth=0.05]
{
\newpath
\moveto(6.6025167,-2.1993248)
\lineto(6.7022557,-2.517516)
\curveto(6.7521253,-2.6766114)(6.9801006,-2.994801)(7.158207,-3.1538966)
\curveto(7.336313,-3.3129919)(7.56785,-3.5039062)(7.728146,-3.5993638)
}
\pscustom[linecolor=black, linewidth=0.05]
{
\newpath
\moveto(6.6225758,-2.060248)
\lineto(6.7188935,-2.0899072)
\curveto(6.767052,-2.104737)(6.987631,-2.143313)(7.1600494,-2.1670594)
\curveto(7.332468,-2.1908057)(7.5566297,-2.2205396)(7.7118564,-2.2385025)
}
\psdots[linecolor=black, dotsize=0.49939746](15.7,-2.2)

\pscustom[linecolor=black, linewidth=0.05]
{
\newpath
\moveto(6.6225758,-2.1602478)
\lineto(6.7188935,-2.2731006)
\curveto(6.767052,-2.3295276)(6.987631,-2.4763086)(7.1600494,-2.5666625)
\curveto(7.332468,-2.6570165)(7.5566297,-2.7701538)(7.7118564,-2.8385026)
}
\pscustom[linecolor=black, linewidth=0.05]
{
\newpath
\moveto(14.063345,-1.7997204)
\lineto(14.248713,-1.4815247)
\curveto(14.341396,-1.3224268)(14.765093,-1.004231)(15.096106,-0.84513307)
\curveto(15.427119,-0.68603516)(15.857434,-0.4951172)(16.155346,-0.39965865)
}
\psline[linecolor=black, linewidth=0.02](6.523917,0.0)(6.523917,-4.0)(6.523917,-3.9499903)
\psline[linecolor=black, linewidth=0.05, arrowsize=0.05291666666666672cm 6.0,arrowlength=2.0,arrowinset=0.0]{<-}(11.798217,-1.9996121)(13.825617,-1.9996121)
\pscustom[linecolor=black, linewidth=0.05]
{
\newpath
\moveto(14.136145,-2.1993248)
\lineto(14.321512,-2.517516)
\curveto(14.414197,-2.6766114)(14.8378935,-2.994801)(15.168905,-3.1538966)
\curveto(15.499917,-3.3129919)(15.930234,-3.5039062)(16.228146,-3.5993638)
}
\pscustom[linecolor=black, linewidth=0.05]
{
\newpath
\moveto(14.173426,-2.060248)
\lineto(14.290537,-2.0899072)
\curveto(14.349093,-2.104737)(14.617292,-2.143313)(14.826935,-2.1670594)
\curveto(15.036577,-2.1908057)(15.309134,-2.2205396)(15.497871,-2.2385025)
}
\psline[linecolor=colour0, linewidth=0.08, linestyle=dotted, dotsep=0.10583334cm](15.4,-0.9)(15.4,-1.9)(15.4,-1.8)
\psline[linecolor=colour0, linewidth=0.08, linestyle=dotted, dotsep=0.10583334cm](7.282551,-0.9)(7.282551,-1.9)(7.282551,-1.8)
\pscustom[linecolor=black, linewidth=0.05]
{
\newpath
\moveto(14.173426,-2.1602478)
\lineto(14.352434,-2.28974)
\curveto(14.441938,-2.354486)(14.851885,-2.5229077)(15.172329,-2.6265833)
\curveto(15.492773,-2.7302587)(15.90938,-2.8600762)(16.197872,-2.9385026)
}
\psdots[linecolor=black, dotstyle=otimes, dotsize=0.49939746](15.7,-2.8)
\psdots[linecolor=black, dotstyle=otimes, dotsize=0.49939746](15.7,-3.4)
\psline[linecolor=colour0, linewidth=0.05](15.9,-2.2)(16.2,-2.2)(16.2,-2.2)
\psdots[linecolor=colour0, fillstyle=solid, dotstyle=o, dotsize=0.5064466](5.0,-2.0)
\psdots[linecolor=black, dotstyle=otimes, dotsize=0.49939746](5.0,-2.0)
\psdots[linecolor=black, dotsize=0.49939746](6.5,-2.0)
\psdots[linecolor=black, dotsize=0.49939746](15.7,-0.6)
\psdots[linecolor=black, dotsize=0.49939746](15.7,-3.4)
\psdots[linecolor=black, dotsize=0.49939746](15.7,-2.8)
\psline[linecolor=black, linewidth=0.02](14.023917,0.0)(14.023917,-4.0)(14.023917,-3.9499903)
\psline[linecolor=black, linewidth=0.02](15.723917,0.0)(15.723917,-4.0)(15.723917,-3.9499903)
\psdots[linecolor=colour0, dotstyle=o, dotsize=0.5064466](14.0,-2.0)
\psdots[linecolor=black, dotstyle=otimes, dotsize=0.49939746](14.0,-2.0)
\end{pspicture}
}
\end{center}
\item[(iv)] 
replacing an arity-$1$ horizontal vertex followed by an
  arity-$n$ vertical vertex with an arity-$n$ horizontal vertex
  followed by an arity $1$ vertical vertex:
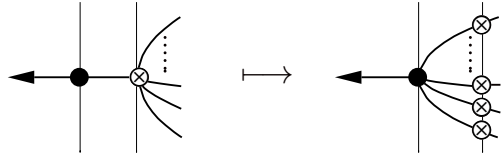
\begin{center}
\psscalebox{.5 .5}
{
\begin{pspicture}(5,-4.0)(14.735774,0) 
\definecolor{colour0}{rgb}{0.03137255,0.003921569,0.003921569}
\pscustom[linecolor=black, linewidth=0.05]
{
\newpath
\moveto(8.9,-0.6)
}
\pscustom[linecolor=black, linewidth=0.05]
{
\newpath
\moveto(21.5,0.7)
}

\pscustom[linecolor=black, linewidth=0.05]
{
\newpath
\moveto(6.5633454,-1.7997204)
\lineto(6.6630845,-1.4815247)
\curveto(6.712954,-1.3224268)(6.9409304,-1.004231)(7.1190357,-0.84513307)
\curveto(7.2971416,-0.68603516)(7.5286794,-0.4951172)(7.688975,-0.39965865)
}
\psline[linecolor=black, linewidth=0.02](5.023917,0.0)(5.023917,-4.0)(5.023917,-3.9499903)
\psline[linecolor=black, linewidth=0.05, arrowsize=0.05291666666666672cm 6.0,arrowlength=2.0,arrowinset=0.0]{<-}(3.0982163,-1.9996121)(6.325617,-1.9996121)
\pscustom[linecolor=black, linewidth=0.05]
{
\newpath
\moveto(6.6025167,-2.1993248)
\lineto(6.7022557,-2.517516)
\curveto(6.7521253,-2.6766114)(6.9801006,-2.994801)(7.158207,-3.1538966)
\curveto(7.336313,-3.3129919)(7.56785,-3.5039062)(7.728146,-3.5993638)
}
\pscustom[linecolor=black, linewidth=0.05]
{
\newpath
\moveto(6.6225758,-2.060248)
\lineto(6.7188935,-2.0899072)
\curveto(6.767052,-2.104737)(6.987631,-2.143313)(7.1600494,-2.1670594)
\curveto(7.332468,-2.1908057)(7.5566297,-2.2205396)(7.7118564,-2.2385025)
}
\psline[linecolor=colour0, linewidth=0.08, linestyle=dotted, dotsep=0.10583334cm](7.282551,-0.9)(7.282551,-1.9)(7.282551,-1.8)
\pscustom[linecolor=black, linewidth=0.05]
{
\newpath
\moveto(6.6225758,-2.1602478)
\lineto(6.7188935,-2.2731006)
\curveto(6.767052,-2.3295276)(6.987631,-2.4763086)(7.1600494,-2.5666625)
\curveto(7.332468,-2.6570165)(7.5566297,-2.7701538)(7.7118564,-2.8385026)
}
\pscustom[linecolor=black, linewidth=0.05]
{
\newpath
\moveto(14.063345,-1.7997204)
\lineto(14.248713,-1.4815247)
\curveto(14.341396,-1.3224268)(14.765093,-1.004231)(15.096106,-0.84513307)
\curveto(15.427119,-0.68603516)(15.857434,-0.4951172)(16.155346,-0.39965865)
}
\psline[linecolor=black, linewidth=0.02](6.523917,0.0)(6.523917,-4.0)(6.523917,-3.9499903)
\psline[linecolor=black, linewidth=0.05, arrowsize=0.05291666666666672cm 6.0,arrowlength=2.0,arrowinset=0.0]{<-}(11.798217,-1.9996121)(13.825617,-1.9996121)
\pscustom[linecolor=black, linewidth=0.05]
{
\newpath
\moveto(14.036145,-2.1993248)
\lineto(14.221513,-2.517516)
\curveto(14.314197,-2.6766114)(14.737893,-2.994801)(15.068905,-3.1538966)
\curveto(15.399917,-3.3129919)(15.830235,-3.5039062)(16.128145,-3.5993638)
}
\pscustom[linecolor=black, linewidth=0.05]
{
\newpath
\moveto(14.173426,-2.060248)
\lineto(14.290537,-2.0899072)
\curveto(14.349093,-2.104737)(14.617292,-2.143313)(14.826935,-2.1670594)
\curveto(15.036577,-2.1908057)(15.309134,-2.2205396)(15.497871,-2.2385025)
}
\psline[linecolor=colour0, linewidth=0.08, linestyle=dotted, dotsep=0.10583334cm](15.4,-0.9)(15.4,-1.9)(15.4,-1.8)
\pscustom[linecolor=black, linewidth=0.05]
{
\newpath
\moveto(14.173426,-2.1602478)
\lineto(14.352434,-2.28974)
\curveto(14.441938,-2.354486)(14.851885,-2.5229077)(15.172329,-2.6265833)
\curveto(15.492773,-2.7302587)(15.90938,-2.8600762)(16.197872,-2.9385026)
}
\psline[linecolor=colour0, linewidth=0.05](15.9,-2.2)(16.2,-2.2)(16.2,-2.2)
\psdots[linecolor=colour0, fillstyle=solid, dotstyle=o, dotsize=0.5064466](6.6,-2.0)
\psdots[linecolor=black, dotstyle=otimes, dotsize=0.49939746](6.6,-2.0)
\psdots[linecolor=black, dotsize=0.49939746](5.0,-2.0)
\psdots[linecolor=black, dotsize=0.49939746](14.0,-2.0)
\psline[linecolor=black, linewidth=0.02](14.023917,0.0)(14.023917,-4.0)(14.023917,-3.9499903)
\psline[linecolor=black, linewidth=0.02](15.723917,0.0)(15.723917,-4.0)(15.723917,-3.9499903)
\psdots[linecolor=colour0, dotstyle=o, dotsize=0.5064466](15.7,-2.8)
\psdots[linecolor=colour0, dotstyle=o, dotsize=0.5064466](15.7,-3.4)
\psdots[linecolor=colour0, dotstyle=o, dotsize=0.5064466](15.7,-2.2)
\psdots[linecolor=colour0, dotstyle=o, dotsize=0.5064466](15.7,-0.6)
\psdots[linecolor=black, dotstyle=otimes, dotsize=0.49939746](15.7,-3.4)
\psdots[linecolor=black, dotstyle=otimes, dotsize=0.49939746](15.7,-2.8)
\psdots[linecolor=black, dotstyle=otimes, dotsize=0.49939746](15.7,-0.6)
\psdots[linecolor=black, dotstyle=otimes, dotsize=0.49939746](15.7,-2.2)
\rput(10.0,-2.0){\textcolor{colour0}{\Huge $\longmapsto$}}
\end{pspicture}
}
\end{center}
\end{itemize}
\end{lemma}

Notice that moves (iii) and (iv) are `local' in that they do not
change the level structure of $\beta$ and that one can be obtained from
the other by interchanging the r\^oles of  $\tlusty$ and $\cerny$.

\begin{proof}[Proof of Lemma~\ref{elementary-moves}]
It is immediate to see that none of the moves changes the $w$-images. Therefore,
if $\beta'$ and $\beta''$ differ by a sequence of the moves
and their inverses, $w(\beta') =  w(\beta'')$.

To prove the opposite implication, let us say that a leveled tree $\beta \in
\LTr$ is in the {\em canonical form\/}, if $\beta$ has no levels
consisting only of \cerny-vertices of arity $1$, and no internal edges
as in~(\ref{eq:31}). It is obvious that, if  $\beta',\beta'' \in
\LTr$ are in the canonical form, then  $w(\beta') =  w(\beta'')$ if and
only if $\beta' = \beta''$ in $\LTr$. The proof is finished by observing
that each $\beta \in \LTr$ can be brought to the canonical form by a
finite sequence of moves (i)--(iv) and their inverses.
\end{proof}

\begin{lemma} 
  Let $\scA$ be a multiplicative $1$-operad $\scA$ in a duoidal category
  $\duo$. If two leveled trees $\beta',\beta'' \in \LTr$ differ by a
  finite sequence of elementary moves listed in
  Lemma~\ref{elementary-moves}, then the
  structure morphisms
\[
\Xi_{\beta'}(I) : 
\scA^{\frT}\to \scA(n) \ \mbox {  and } \ 
\Xi_{\beta''}(I): \scA^{\frT}\to \scA(n),
\]
where $\frT := \Omega(\beta') = \Omega(\beta'')$ and $\Xi = \{\Xi_\beta\}_{\beta \in
 \LTr}$ is the morphism~(\ref{eq:33}), 
 are equal.
\end{lemma}

\begin{proof}
Expanding the definitions we see that we must establish
that the diagram  
\begin{equation}
\label{eq:10}
\xymatrix@C = +3em@R = .1em{
& \scA^{\beta'} \ar[r]^{\omega_{\beta'}}  & \scA(\beta') \ar[rd]^{\gamma_{\beta'}}  &
\\
\scA^\frT    \ar[ur]^{\theta_{\beta''}} \ar[dr]^{\theta_{\beta'}} &&& \scA(n)
\\
& \scA^{\beta''} \ar[r]^{\omega_{\beta''}}  & \scA(\beta'') \ar[ur]^{\gamma_{\beta''}}  &
}
\end{equation}
in which $\theta_{\beta'},\theta_{\beta''}$ are the 
maps~(\ref{eq:tttt}), $\omega_{\beta'}, \omega_{\beta''}$ the
maps~(\ref{eq:20}) and $\gamma_{\beta'},\gamma_{\beta''}$ the operad
compositions,  commutes for
each move of Lemma~\ref{elementary-moves}.

\noindent 
{\em Move~(i).\/}
Assume that $\beta''$ is obtained from $\beta'$ by adding a level of
horizontal vertices of arity~ $1$. It follows from the defining
formula~(\ref{eq:15}) that there are some $\scA^{\beta'}_l, \scA^{\beta'}_r
\in \duo$ such that 
\[
\scA^{\beta'} = \scA^{\beta'}_l \boxx_0 \scA^{\beta'}_r\ \mbox { while
}\ \scA^{\beta''} = \scA^{\beta'}_l \boxx_0 (e\boxx_1 \cdots \boxx_1 e) \boxx_0
\scA^{\beta'}_r.
\]
Likewise, it follows from~(\ref{eq:17}) that there are some
$\scA_l(\beta'),\scA_r(\beta') \in \duo$ such that 
\[
\scA(\beta') = \scA_l(\beta')\boxx_0  \scA_r(\beta') \ \mbox { while }
\scA(\beta'') = 
\scA_l(\beta')\boxx_0\big(\scA(1)\boxx_1 \cdots \boxx_1 \scA(1)\big)\boxx_0
\scA_r(\beta').
\]
The unitality axiom \cite[Def.~4.1]{batanin-markl:Adv} for $1$-operads
then implies the commutativity of the diagram
\[
\xymatrix@C = +1.8em@R = 1em{ \scA^{\beta'}_l \boxx_0 \scA^{\beta'}_r
\ar[d]^\cong  \ar[r]^\omega & \scA_l(\beta') \boxx_0  \scA_r(\beta')  \ar[rd]^\gamma &
  \\
\ar[d]   \scA^{\beta'}_l \boxx_0 e \boxx_0 \scA^{\beta'}_r&& \scA(n)
  \\
  \scA^{\beta'}_l \boxx_0 (e\boxx_1 \cdots \boxx_1 e) \boxx_0 \scA^{\beta'}_r
  \ar[r]^{\omega \hskip 3em} &\scA_l(\beta')\boxx_0\big(\scA(1)\boxx_1 \cdots \boxx_1
  \scA(1)\big)\boxx_0  \scA_r(\beta')   \ar[ur]^\gamma & }
\]
which, along with the obvious commutativity of
\[
\xymatrix@C = +3.8em@R = 1em{&   \scA^{\beta'}_l \boxx_0 \scA^{\beta'}_r
\ar[d]^\cong 
  \\
\scA^\frT  \ar[ur]^{\theta_{\beta'}} \ar[dr]^{\theta_{\beta''}}   &
\ar[d]   \scA^{\beta'}_l \boxx_0 e \boxx_0 \scA^{\beta'}_r
  \\
&  \scA^{\beta'}_l \boxx_0 (e\boxx_1 \cdots \boxx_1 e) \boxx_0 \scA^{\beta'}_r
}
\]
implies  the commutativity of~(\ref{eq:10}).

\noindent 
{\em Move (ii).\/} By analyzing the definitions of the objects
involved in~(\ref{eq:10}), we easily see that its commutativity would
follow from the commutativity of the diagram
\begin{equation}
\label{eq:35}
\xymatrix{
(e\boxx_1 \cdots \boxx_1 e)\boxx_0 e \ar[d] \ar[r] & 
(v\boxx_1 \cdots \boxx_1v)\boxx_0 v \ar[d]^\cong
\cr
e \ar[r] & v
}
\end{equation}
whose left vertical arrow is constructed in
\cite[Example~4.5]{batanin-markl:Adv}, and the remaining arrows are
induced by the monoid structure of $v$ and by the canonical map $e \to v$. 

\noindent 
{\em Move (iii).\/}
The commutativity of~(\ref{eq:10}) would in this case follow from the
commutativity of
\begin{equation}
\label{eq:36}
\xymatrix{(e\boxx_1 \cdots \boxx_1 e)\boxx_0 v \ar[r] & 
(v\boxx_1 \cdots \boxx_1v)\boxx_0 v \ar[d]^\cong
\cr
e \boxx_0 v \ar[r]\ar[u] & v \boxx_0 v
}
\end{equation}
whose maps are induced by the comonoid structure of $e$, monoid
structure of $v$ and the  canonical map $e \to v$.

\noindent 
{\em Move (iv).\/}
The commutativity of~(\ref{eq:10}) would follow from the
commutativity of the diagram
\begin{equation}
\label{eq:37}
\xymatrix{(v\boxx_1 \cdots \boxx_1 v)\boxx_0 e \ar[r]\ar[d]_\cong & 
(v\boxx_1 \cdots \boxx_1v)\boxx_0 v \ar[d]^\cong
\cr
v \boxx_0 e \ar[r] & v \boxx_0 v
}
\end{equation}
whose arrows are given by the module structure of $v$ and the
canonical map $e \to v$. The commutativity of
diagrams~(\ref{eq:35})--(\ref{eq:37}) above is however an easy
consequence of
general properties of duoidal categories.
\end{proof}

\section{Duoidal Deligne's conjecture}
\label{Geminiany}
Let $\duo$ be a complete $V$-category and $\delta : \ttDelta \to V$ a
cosimplicial object in $V$.  Recall
\cite[\S5.2]{batanin-markl:Adv} that the $\delta$-{\em totalization\/} of a
cosimplicial object $\phi : \ttDelta \to \duo$ is the $V$-enriched end
\[
\Tot_\delta(\phi) := \int_{n \in \ttDelta} \phi(n)^{\delta(n)} \in \duo.
\]
By Proposition 5.2 of \cite{batanin-markl:Adv}, any multiplicative
$1$-operad $\scA$ in $\duo$ bears a canonical structure of a
cosimplicial object $\cosA = \{\scA(n),n\ge 0\}$ in $\duo$. In
Definition 5.3 of \cite{batanin-markl:Adv} we introduced the {\em
Hochschild $\delta$-object\/} of a $\scA$ as the $\delta$-totalization
\[
\CH_\delta(\scA) := \Tot_\delta(\cosA).
\] 

We claim that the canonical cosimplicial structure on $\scA$ is
induced by the action~(\ref{eq:23}) of the $2$-operad $\opTm^\bbN_2$. 
By this we mean 
that the underlying category $\ttU(\opTm^\bbN_2)$ is the
simplicial category
$\ttDelta$ and that  $\cosA =\liota^*(\scA)$,
where $\liota^*$ is the functor~(\ref{Jarka_slusne_prudi}) with $\calP
= \opTm^\bbN_2$ and $\ttC = \duo$.  

It follows from definition that the objects of the underlying category
$\ttU(\opTm^\bbN_2)$ are natural numbers. Its hom-sets are
\[
\ttU(\opTm^\bbN_2)(i,n) = \Tm^\bbN_2\big(\EuU_2(i,n)\big),
\ i,n \in \bbN,
\]
where $\EuU_2(i,n)$ is the terminal $2$-tree $\EuU_2$ with its unique
$2$-leaf colored by $i$ and the root by $n$. Morphisms in
$\ttU(\opTm^\bbN_2)(i,j)$ are thus represented by trees as in
Lemma~\ref{tam-as-tree} with one $\bily$-vertex of arity $i$, no
$\tlusty$-vertices, and $n$ leaves.  It is obvious from this description that
$\opTm^\bbN_2$ and $\Lat^{(2)}$ have isomorphic underlying categories,
while $\ttU(\Lat^{(2)}) \cong \ttDelta$ by
\cite[Lemma~2.5]{batanin-berger}.  The identification $\cosA =
\liota^*(\scA)$ 
is now a simple exercise. We conclude that
\[
\CH_\delta(\scA) \cong \Tot_{\delta}(\scA),
\]
where $\Tot_{\delta}(\scA)$ is given by~(\ref{totalization}).

As we already recalled from~\cite[Prop.~4.9]{batanin-markl:Adv}, the
endomorphism $1$-operad $\End_\ttM$ of a monoid $\ttM$ in a
$\duo$-monoidal category is multiplicative. The Hochschild
$\delta$-object of $\End_\ttM$ was called the {\em $\delta$-center\/}
of $\ttM$ and denoted
\[
\CH_\delta(\ttM,\ttM) := \CH_\delta(\End_\ttM).
\]

Since $V$ is a cocomplete symmetric monoidal category it has $Set$-tensors. For a set $S$ the tensor $S\otimes X$ is equal to the coproduct $\coprod_S X$  of $S$-copies of $X.$ For any operadic category $\ttO$ and any  $\ttO$-operad $\calP$ in $Set$ we can construct then an  enrichment of $\calP$ in $V$  which on an object $T\in \ttO$ takes value $\calP(T)\otimes I.$ By abusing notations we will denote such an enriched $\ttO$-operad by the same letter $\calP.$  In particular we will consider the  operad 
$\opTm^\bbN_2$ as a colored $2$-operad in $V$ for any cocomplete $V.$ 

Proposition \ref{v_patek_do_Prahy} immediately
gives:

\begin{theorem} 
  Let $\delta$ be a cosimplicial object in $V.$ Then there is a
  canonical action of the $2$-operad $\Coend^{\sopTm^\bbN_2}_{\delta}$ on the
  Hochschild object $\CH_{\delta}(\scA)$ of a multiplicative $1$-operad
  $\scA.$ In particular, there is a canonical action of $\Coend^{\sopTm^\bbN_2}_{\delta}$ on the
  $\delta$-center $\CH_{\delta}(\ttM,\ttM)$ of any 
monoid $\ttM$ in any $\duo$-monoidal category
\end{theorem}

When $\ttV$ is the category of chain complexes,
$\Coend^{\sopTm^\bbN_2}_{\delta}$ is the chain $2$-operad $\hbox{\O}$ considered by
Tamarkin in \cite[\S5.2]{tamarkin:CM07}.
Let $I$ be the constant cosimplicial 
object whose all terms equal the unit object $I\in \ttV$.

\begin{proposition}
The $2$-operad $\Coend^{\sopTm^\bbN_2}_{I}$ is isomorphic to
the canonical $2$-operad $\bfone^{\ttOmega_2}.$
\end{proposition}

\begin{proof}
  It is sufficient to observe that the value of the multitensor
  $E_\EuT^{\sopTm^\bbN_2}(I,\ldots,I)$ is the constant cosimplicial
  object $I$ for any $\EuT\in {\tt \Omega}_2$ . Indeed, if it is so,
  then clearly
\[
\Coend^{\sopTm^\bbN_2}_{I} = \Nat(I,I) = \ttV(I,I) \cong I
\]
where  $\Nat(I,I)$ means the space of natural transformations (i.e.\
cosimplicial maps) between the constant cosimplicial objects $I$.

It is clear that for each $n\ge 0$ the coend
$E_\EuT^{\sopTm^\bbN_2}(I,\ldots,I)(n)$ in~(\ref{convolut}) 
equals  the colimit of the $k$-simplicial object
$
\opTm^\bbN_2(\EuT)(\bullet,\ldots,\bullet;n) \ot I
$, 
$k := |\EuT|$, so it is enough to check that the colimit of the $k$-simplicial set
$\opTm^\bbN_2(\EuT)(\bullet,\ldots,\bullet;n)$ is a one point
set. This boils down to verification that the equivalence relation
generated by the simplicial operators on
$\opTm^\bbN_2(\EuT)(\rada00;n)$ has only one equivalence class.

Notice that, by the definition~(\ref{Ralph}), 
\[
\opTm^\bbN_2(\EuT)(\Rada i1k;n) = 
\opTm^\bbN_2(\bfT^{\Rada i1k}_n),
\]
where $\bfT^{\Rada i1k}_n \in \ttOmega^\bbN_2$ is the $2$-tree $\EuT \in \ttOmega_2$ with
its $2$-leaves
colored by $\rada{i_1}{i_k} \in \bbN$ and the root by~$n$.
The elements of $\opTm^\bbN_2(\EuT)(\rada 00;n)$ are thus represented by
$\EuT$-labelled leveled trees whose $\bily$-vertices have no incoming
edges. In this description, the two simplicial operators 
\begin{equation}
\label{Kasparek}
\partial^i_0,\partial^i_1 : \opTm^\bbN_2(\bfT^{\rada 01,\ldots,0}_n) 
\to \opTm^\bbN_2(\bfT^{\rada 00}_n), \ 1 \leq i \leq k,
\end{equation}
in $i$-th direction are local operations acting on a tree $\xi \in
\opTm^\bbN_2(\bfT^{\rada 00}_n)$ 
by the following surgery: `cut off' the unique
incoming edge of the $i$th $\bily$-vertex of $\xi$ and then `glue' it to the
outcoming edge of this vertex in two possible ways introducing a new
$\cerny$-vertex, as indicated in
Figure~\ref{two-ways} --  
compare with the differential in the brace operad 
\cite[Example~5.8]{batanin-berger-markl}. 
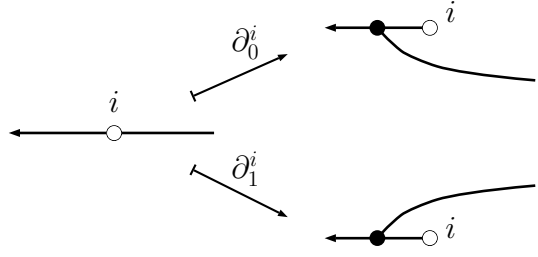
\begin{figure}
\centering
\psscalebox{.7 .7} 
{
\begin{pspicture}(2,-4.575521)(11.002144,.375521)
\pscustom[linecolor=black, linewidth=0.05]
{
\newpath
\moveto(8.9,-0.024479065)
}
\pscustom[linecolor=black, linewidth=0.05]
{
\newpath
\moveto(21.5,1.2755209)
}

\pscustom[linecolor=black, linewidth=0.05]
{
\newpath
\moveto(8.0,-4.424479)
\lineto(8.132912,-4.2774205)
\curveto(8.199368,-4.203891)(8.36076,-4.056832)(8.455696,-3.9833024)
\curveto(8.550632,-3.9097733)(8.892405,-3.762714)(9.13924,-3.689185)
\curveto(9.386076,-3.615656)(9.974683,-3.512714)(11.0,-3.424479)
}
\pscustom[linecolor=black, linewidth=0.05]
{
\newpath
\moveto(8.0,-0.42447907)
\lineto(8.132912,-0.5715378)
\curveto(8.199368,-0.64506745)(8.36076,-0.7921262)(8.455696,-0.86565584)
\curveto(8.550632,-0.9391849)(8.892405,-1.0862442)(9.13924,-1.1597732)
\curveto(9.386076,-1.2333024)(9.974683,-1.3362442)(11.0,-1.424479)
}
\psline[linecolor=black, linewidth=0.05](3.1,-2.424479)(4.9,-2.424479)(4.9,-2.424479)
\psdots[linecolor=black, dotsize=0.3](8.0,-4.424479)
\psdots[linecolor=black, dotsize=0.3](8.0,-0.42447907)
\psline[linecolor=black, linewidth=0.05, arrowsize=0.05291666666666668cm 2.0,arrowlength=1.399997,arrowinset=0.0]{<-}(1.0,-2.424479)(2.9,-2.424479)
\psline[linecolor=black, linewidth=0.05, arrowsize=0.05291666666666668cm 2.0,arrowlength=1.399997,arrowinset=0.0]{<-}(7.0,-4.424479)(8.9,-4.424479)
\psline[linecolor=black, linewidth=0.05, arrowsize=0.05291666666666668cm 2.0,arrowlength=1.399997,arrowinset=0.0]{<-}(7.0,-0.42447907)(8.9,-0.42447907)
\psdots[linecolor=black, dotstyle=o, dotsize=0.3](9.0,-4.424479)
\psdots[linecolor=black, dotstyle=o, dotsize=0.3](9.0,-0.42447907)
\psdots[linecolor=black, dotstyle=o, dotsize=0.3](3.0,-2.424479)
\rput[b](5.5,-1.024479){\Large $\partial^i_0$}
\rput[b](5.5,-3.4244792){\Large $\partial^i_1$}
\rput[b](3.0,-2.0244792){\Large $i$}
\rput[b](9.4,-0.32447907){\Large $i$}
\rput[b](9.4,-4.424479){\Large $i$}
\psline[linecolor=black, linewidth=0.04, tbarsize=0.02cm 5.0,arrowsize=0.05291666666666667cm 2.0,arrowlength=1.4,arrowinset=0.0]{|*->}(4.5,-1.7244791)(6.3,-0.92447907)
\psline[linecolor=black, linewidth=0.04, tbarsize=0.02cm 5.0,arrowsize=0.05291666666666667cm 2.0,arrowlength=1.4,arrowinset=0.0]{|*->}(4.5,-3.124479)(6.3,-4.024479)
\end{pspicture}
}
\caption{The local surgery defining the operators $\partial^i_0,
\partial^i_1$ in~(\ref{Kasparek}).\label{two-ways}}
\end{figure}
It is simple to prove by induction that any two trees from
$\opTm^\bbN_2(\EuT)(\rada 00;n)$ can be connected by a zig-zag of such
elementary surgery operations, so the colimit of
$\opTm^\bbN_2(\EuT)(\bullet,\ldots,\bullet;n)$ is a one-point set as claimed. 
\end{proof}

The center of a monoid $\ttM$ in a
$\duo$-monoidal category $\calK$, 
defined as $Z(M) :=
\CH_{I}(\ttM,\ttM)$ or, more explicitly, as the
equalized~(\ref{equalizer}), has a canonical structure of a duoid in $\duo$ 
by \cite[Theorem 5.6]{batanin-markl:Adv}. On the other hand, according to
\cite[Example~11.15]{batanin-markl:Adv}, duoids in $\duo$ are the same
as $\bfone^{\ttOmega_2}$-algebras in $\duo$.\footnote{The operad
  $\bfone^{\ttOmega_2}$ was denoted $\fAss_2$ in 
\cite{batanin-markl:Adv}.}
The following proposition is easy to prove.

\begin{proposition}
The action of the operad $\Coend^{\sopTm^\bbN_2}_{I} \cong \bfone^{\Omega_2}$ equips $Z(\ttM)$ with its
canonical duoid structure.
\end{proposition}

Let now $V$ be a monoidal model category and $\delta$ be a
\emph{standard system of simplices} for  $\ttV$ in the sense of
\cite[Definition A.6]{berger-moerdijk:BV}. 
Recall that this means that
\begin{itemize}
\item[(i)]
$\delta$ is cofibrant for the Reedy model structure on
  $V^{\ttDelta}$,
\item[(ii)] 
$\delta^0$ is the unit object $I$ of $V$ and the
  simplicial operators $[m]\to[n]$ act via weak equivalences
  $\delta^m\to\delta^n$ in $V$, and
\item[(iii)]
the simplicial realization
  functor $|\dash|_\delta=(\dash)\otimes_\ttDelta\delta:V^{\ttDelta^{\op}}\to
  V$ is a symmetric monoidal functor whose structural maps
  \[
|X|_\delta\otimes_V |Y|_\delta\to|X\otimes_V Y|_{\delta}
\]
are weak
  equivalences for Reedy-cofibrant objects $X,Y \in V^{\ttDelta^{\op}}$.\end{itemize}

Since $E_\EuT^{\sopTm^\bbN_2}(I,\ldots,I) = I$, the 
canonical map of cosimplicial objects $\delta\to I$ induces a map of $2$-operads
\[
\Coend^{\sopTm^\bbN_2}_{\delta}\lra \Coend^{\sopTm^\bbN_2}_{I}.
\]

\begin{theorem}
\label{contractiblecoendomorphism} 
Let $\delta$ be a standard system of simplices for a monoidal model
category $V$ such that the lattice path operad is strongly
$\delta$-reductive in the sense of\/
\cite[Definition~3.7]{batanin-berger}.  Then the canonical morphism of
$2$-operads
\begin{equation}\label{deltatoI}
\Coend^{\sopTm^\bbN_2}_{\delta}\lra \Coend^{\sopTm^\bbN_2}_{I}
\cong \bfone^{\ttOmega_2}
\end{equation} 
is a weak equivalence.  In other words, the $2$-operad
$\Coend^{\sopTm^\bbN_2}_{\delta}$ is {contractible}.
\end{theorem}

\begin{proof} The proof  follows closely the proof of
Theorem 3.8 from \cite{batanin-berger}, with the simplification that we do
not need to take a colimit over the complete graph operads. The only fact we
should know is that the  map of $0$-objects 
\[
f^0:(E^{Tam_2^\bbN}_\EuT(\delta,\ldots,\delta))^0\to I^0 = I
\]
 is a weak equivalence for every $\EuT.$ 
In the notation used in the proof of \cite[Theorem
3.8]{batanin-berger},
the object on the left side is the same as
$\xi_{(\mu,\sigma)}(\delta)^0$ for $(\mu,\sigma)$ equal to $a_\EuT\in \opK^{(2)}.$
It is shown at the end of the proof of
\cite[Theorem 3.8]{batanin-berger} that the map $f^0$ is a weak equivalence.  
\end{proof}

The map $\delta \to I$ induces a canonical map
\begin{equation}\label{centertocenter}
Z(M) = \CH_{I}(\ttM,\ttM) \lra \CH_{\delta}(\ttM,\ttM),
\end{equation}
The map (\ref{deltatoI}) equips the duoid $Z(\ttM)$ with a structure
of $\Coend^{\sopTm^\bbN_2}_{\delta}$-algebra such that
(\ref{centertocenter}) becomes a map of
$\Coend^{\sopTm^\bbN_2}_{\delta}$-algebras.

Let $F(\ttM)$ be a fibrant replacement of $\ttM$ in the category of
monoids with the projective model structure, and $\delta$ a
standard system of simplices for $V$. The $\delta$-center
$\CH_{\delta}\big(F(\ttM),F(\ttM)\big)$ of $F(\ttM)$ was called in
\cite{batanin-markl:Adv} the {\it homotopy center} of $\ttM$. The
above considerations imply the central result of our paper:

\begin{corollary}[Duoidal Deligne's conjecture]\label{duoidalDeligne}
  Under the assumptions of Theorem \ref{contractiblecoendomorphism},
  the Hochschild $\delta$-object of a multiplicative $1$-operad in a
  duoidal category $\duo$ admits an action of a contractible
  $2$-operad.

The homotopy center of a monoid $\ttM$ in a multiplicative
$\duo$-category admits an action of a contractible
$2$-operad that lifts the duoid structure on the center $Z(\ttM)$.
\end{corollary}

The assumptions of Theorem
\ref{contractiblecoendomorphism} are satisfied for instance when  
$V$ is the category of compactly generated topological
spaces or chain complexes over a commutative ring, and $\delta$ the
cosimplicial space of topological simplices or normalized cellular
chains on topological simplices, respectively, see  
\cite[Examples 3.10(a),(c)]{batanin-berger}. It is also not
difficult to show that these assumptions are satisfied for $V={\tt Cat}$ with
the Joyal-Tirney model structure and $\delta$ the cosimplicial chaotic
groupoid on finite sets, cf. \cite[Example 5.10]{batanin-markl:Adv}.

On the other hand, it was shown in \cite[Example
3.10(b)]{batanin-berger} that for the category of simplicial sets and
$\delta = \delta_{\it Yon}$ the cosimplicial simplicial set of
representables, the assumption of strongly $\delta$-reductivity of the
lattice operad fails. We however believe that the second part of
Corollary \ref{duoidalDeligne} remains true even without this
assumption, because taking fibrant replacement of $\ttM$ should
counterweight the poor homotopical property of $\delta.$ We leave this
refined version of Deligne's conjecture as a subject for a future
work.


\begin{thebibliography}{10}

\bibitem{AM} M. Aguiar, S. Mahajan.
\newblock Monoidal functors, species and Hopf algebras.
\newblock {\em CRM Monograph Series}, AMS, 29, 2010.

\bibitem{barwick}
C.~Barwick.
\newblock From operator categories to topological operads.
\newblock Preprint {arXiv:1302.5756}, February 2013.

\bibitem{batanin:globular}
M.A. Batanin.
\newblock {Monoidal globular categories as a natural environment for the theory
  of weak {$n$}-categories}.
\newblock {\em Adv. Math.}, 136(1):39--103, 1998.

\bibitem{batanin:conf}
M.A. Batanin.
\newblock {Symmetrisation of {$n$}-operads and compactification of real
  configuration spaces}.
\newblock {\em Adv. Math.}, 211(2):684--725, 2007.

\bibitem{batanin:AM08}
M.A. Batanin.
\newblock {The {E}ckmann-{H}ilton argument and higher operads}.
\newblock {\em Adv. Math.}, 217(1):334--385, 2008.


\bibitem{batanin13:_homot}
M.A. Batanin and C.~Berger.
\newblock {Homotopy theory for algebras over polynomial monads}.
\newblock Preprint, 2013.


\bibitem{batanin-berger}
M.A. Batanin and C.~Berger.
\newblock {The lattice path operad and {H}ochschild cochains}.
\newblock In {\em {Alpine perspectives on algebraic topology}}, volume 504 of
  {\em {Contemp. Math.}}, pages 23--52. Amer. Math. Soc., Providence, RI, 2009.

\bibitem{batanin-berger-markl}
M.A. Batanin, C.~Berger, and M.~Markl.
\newblock {Operads of natural operations {I}: {Lattice} paths, braces and
  {Hochschild} cochains}.
\newblock In J.-L. Loday and B.~Vallette, editors, {\em {Operads 2009}},
  volume~26 of {\em {S{\'e}min. Congr.}}, pages 1--33. Soc. Math. France,
  Paris, 2011.

\bibitem{batanin-markl:Adv}
M.A. Batanin and M.~Markl.
\newblock {Centers and homotopy centers in enriched monoidal categories}.
\newblock {\em Adv. Math.}, 230(4-6):1811--1858, 2012.

\bibitem{batanin-street}
M.A. Batanin and R.~Street.
\newblock The universal property of the multitude of trees.
\newblock {\em J. Pure Appl. Algebra}, 154(1-3):3--13, 2000.
\newblock Category theory and its applications (Montreal, QC, 1997).

\bibitem{batanin-weber}
M.A. Batanin and M.~Weber.
\newblock Algebras of higher operads as enriched categories.
\newblock {\em Appl. Categ. Structures}, 19(1):93--135, 2011.

\bibitem{berger:CM202}
C.~Berger.
\newblock Combinatorial models for real configuration spaces and
  {$E_n$}-operads.
\newblock In J.L. Loday, J.D Stasheff, and A.A. Voronov, editors, {\em Operads:
  Proceedings of Renaissance Conferences}, volume 202 of {\em Contemporary
  Math.}, pages 37--52. Amer. Math. Soc., 1997.

\bibitem{Berger:personal}
C.~Berger.
\newblock {Private communication}.


\bibitem{berger-fresse}
C.~Berger and B.~Fresse.
\newblock {Combinatorial operad actions on cochains}.
\newblock {\em Math. Proc. Cambridge Philos. Soc.}, 137(1):135--174, 2004.

\bibitem{berger-moerdijk:CM07}
C. Berger and I. Moerdijk.
\newblock Resolution of coloured operads and rectification of homotopy
  algebras.
\newblock In {\em Categories in algebra, geometry and mathematical physics},
  volume 431 of {\em Contemp. Math.}, pages 31--58. Amer. Math. Soc.,
  Providence, RI, 2007.

\bibitem{berger-moerdijk:BV}
C.~Berger and I.~Moerdijk.
\newblock {The {Boardman-Vogt} resolution of operads in monoidal model
  categories}.
\newblock {\em Topology}, 45(5):807--849, February 2006.



\bibitem{boardman-vogt:73}
J.M. Boardman and R.M. Vogt.
\newblock {\em {Homotopy Invariant Algebraic Structures on Topological
  Spaces}}.
\newblock Springer-Verlag, 1973.

\bibitem{DS2}
B. Day and R. Street.
\newblock {Abstract substitution in enriched categories}.
\newblock {\em J. Pure Appl. Algebra}, 179(1-2):49--63, 2003.

\bibitem{day-street:lax}
B.~Day and R.~Street.
\newblock Lax monoids, pseudo-operads, and convolution.
\newblock In {\em Diagrammatic morphisms and applications ({S}an {F}rancisco,
  {CA}, 2000)}, volume 318 of {\em Contemp. Math.}, pages 75--96. Amer. Math.
  Soc., Providence, RI, 2003.


\bibitem{Forcey}
S.~Forcey.
\newblock Enrichment over iterated monoidal categories.
\newblock {\em Algebr. Geom. Topol.}, 4:95--119 (electronic), 2004.


\bibitem{getzler-kapranov:CompM98}
E.~Getzler and M.M. Kapranov.
\newblock {Modular operads}.
\newblock {\em Compos. Math.}, 110(1):65--126, 1998.


\bibitem{GPS}
R.~Gordon, A.J. Power, and R.~Street.
\newblock {\em {Coherence for Tricategories}}, volume 558 of {\em {Mem. Amer.
  Math. Soc.}}
\newblock AMS, 1995.

\bibitem{kapranov:langlands}
M.M. Kapranov.
\newblock Analogies between the {L}anglands correspondence and topological
  quantum field theory.
\newblock In {\em Functional analysis on the eve of the 21st century, {V}ol.\ 1
  ({N}ew {B}runswick, {NJ}, 1993)}, volume 131 of {\em Progr. Math.}, pages
  119--151. Birkh\"auser Boston, Boston, MA, 1995.

\bibitem{kaufmann:Fey}
R.M. Kaufmann and B.C. Ward.
\newblock {Feynman Categories}.
\newblock Preprint {\tt arXiv:1312.1269}, December 2013.


\bibitem{AK} A. Kock.  
\newblock{Strong functors and monoidal monads.}
\newblock  {\em Archiv der Math.} 23: 113--120, (1972).

\bibitem{markl:handbook}
M.~Markl.
\newblock {Operads and {PROP}s}.
\newblock In {\em {Handbook of algebra. {V}ol. 5}}, volume~5 of {\em {Handb.
  Algebr.}}, pages 87--140. Elsevier/North-Holland, Amsterdam, 2008.

\bibitem{markl-shnider-stasheff:book}
M.~Markl, S.~Shnider, and J.D. Stasheff.
\newblock {\em {Operads in algebra, topology and physics}}, volume~96 of {\em
  {Mathematical Surveys and Monographs}}.
\newblock American Mathematical Society, Providence, RI, 2002.

\bibitem{mcclure-smith}
J.E. McClure and J.H. Smith.
\newblock {A solution of {D}eligne's {H}ochschild cohomology conjecture}.
\newblock In {\em {Recent progress in homotopy theory ({B}altimore, {MD},
  2000)}}, volume 293 of {\em {Contemp. Math.}}, pages 153--193. Amer. Math.
  Soc., Providence, RI, 2002.

\bibitem{maclane-moerdijk}
S.~Mac~Lane and I.~Moerdijk.
\newblock {\em Sheaves in geometry and logic}.
\newblock Universitext. Springer-Verlag, New York, 1994.
\newblock A first introduction to topos theory, Corrected reprint of the 1992
  edition.


\bibitem{tamarkin:CM07}
D.~Tamarkin.
\newblock {What do dg-categories form?}
\newblock {\em Compos. Math.}, 143(5):1335--1358, 2007.

\bibitem{tamarkin-tsygan:LMP01}
D.~Tamarkin and B.~Tsygan.
\newblock Cyclic formality and index theorems.
\newblock {\em Lett. Math. Phys.}, 56:85--97, 2001.

\end{thebibliography}

\def\cprime{$'$}\def\cprime{$'$}

\end{document}